\documentclass[11pt]{amsart}

\usepackage{amssymb}
\usepackage{amsthm}
\usepackage{amsbsy}
\usepackage{subfigure}
\usepackage{amsmath}
\usepackage{graphicx}
\usepackage{comment}
\usepackage[usenames,dvipsnames]{xcolor}
\usepackage[margin = 1in]{geometry}
\usepackage{enumerate}
\usepackage[toc,page]{appendix}
\usepackage{color}
\usepackage{hyperref}
\usepackage[ruled,vlined,linesnumbered]{algorithm2e}

\SetKwInOut{Parameter}{Parameters}

\definecolor{mygreen}{rgb}{0.1,0.75,0.2}

 \newtheorem{thm}{Theorem}[section]
 
 \newtheorem{lem}[thm]{Lemma}
 \newtheorem{prop}[thm]{Proposition}
 \newtheorem{assumption}[thm]{Assumption}
 \newtheorem{rem}[thm]{Remark}
\newtheorem{defn}[thm]{Definition}
\numberwithin{equation}{section}

\DeclareMathOperator{\diam}{diam}
\DeclareMathOperator{\vol}{Vol}
\DeclareMathOperator{\ric}{Ric}
\DeclareMathOperator{\gap}{gap}

\newcommand{\Gammaf}{F}
\newcommand{\hs}{\mathcal{H}}

\newcommand{\bR}{\mathbb{R}}

\newcommand{\la}{\langle}
\newcommand{\ra}{\rangle}
\newcommand{\pt}{\partial}
\newcommand{\eps}{\varepsilon}

\newcommand{\ud}{\,\mathrm{d}}
\newcommand{\8}{\infty}
\newcommand{\mm}{\mathcal{M}}
\newcommand{\nn}{\mathcal{N}}
\newcommand{\mx}{\mathbf{x}}
\newcommand{\my}{\mathbf{y}}
\newcommand{\rhon}{\rho^{\nn}}
\newcommand{\lmd}{\eta}

\newcommand{\divn}{\text{div}_{\nn}}

\newcommand{\sch}{\frac{\pi_i+ \pi_j}{2|\my_i-\my_j|} |\Gamma_{ij}|}
\newcommand{\tsch}{\frac{\pi_i+ \pi_j}{2|\my_i-\my_j|} |\tilde{\Gamma}_{ij}|}

\begin{document}

\title[Data-driven  solvers for Langevin dynamics on manifold]{Data-driven efficient solvers  for  Langevin dynamics on manifold in high dimensions}

\author[Y. Gao]{Yuan Gao}
\address{Department of Mathematics, Purdue University, West Lafayette, IN}
\email{gao662@purdue.edu}

\author[J.-G. Liu]{Jian-Guo Liu}
\address{Department of Mathematics and Department of
  Physics, Duke University, Durham, NC}
\email{jliu@math.duke.edu}

\author[N. Wu]{Nan Wu}
\address{Department of Mathematical Sciences, The University of Texas at Dallas, Richardson, TX}
\email{nan.wu@utdallas.edu}

\date{\today}

\begin{abstract}
{ We study the Langevin dynamics of a physical system with manifold structure $\mm\subset\bR^p$ based
on collected sample points   $\{\mx_i\}_{i=1}^n \subset \mm$ that probe the unknown manifold $\mm$. Through the diffusion map, we first learn the reaction coordinates $\{\my_i\}_{i=1}^n\subset \nn$ corresponding to $\{\mx_i\}_{i=1}^n$, where $\nn$ is a manifold diffeomorphic to $\mm$ and isometrically embedded in $\bR^\ell$ with $\ell \ll p$.  The induced Langevin dynamics on $\nn$ in terms of the reaction coordinates captures the slow time scale dynamics such as conformational changes in biochemical reactions.   To construct an efficient and stable approximation for the Langevin dynamics on $\nn$, we  leverage the corresponding Fokker-Planck equation on the manifold $\nn$ in terms of the reaction coordinates $\my$. We propose an implementable, unconditionally stable, data-driven finite volume  scheme for this Fokker-Planck equation,  which automatically incorporates the manifold structure of $\nn$.  Furthermore, we provide a weighted $L^2$ convergence analysis of the finite volume  scheme to the Fokker-Planck equation on $\nn$. The proposed finite volume  scheme leads to a Markov chain on $\{\my_i\}_{i=1}^n$ with an approximated transition probability and jump rate between the nearest neighbor points. After an unconditionally stable explicit time discretization, the data-driven finite volume scheme gives an approximated Markov process for the Langevin dynamics on $\nn$ and the approximated Markov process enjoys detailed balance, ergodicity, and other good properties.}
\end{abstract}

\keywords{Diffusion map, 
reaction coordinates,
Voronoi tessellation,
unconditionally stable explicit finite volume  scheme,
random walk on point clouds,
exponential convergence}

\maketitle
\section{Introduction}

\subsection{Problem set up and goals}

{ We study a complex chemical, biological or physical system $P$ which can be described by $p$-dimensional variables $\mx$ in $\bR^p$ with  $p\gg 1 $. Due to some equality and inequality constraints, we assume the essential structure of the system $P$ is an unknown $d$ dimensional closed smooth Riemannian submanifold $\mm$ of $\mathbb{R}^p$ \cite{coifman2006diffusion, CKLMN}. The manifold $\mm$ is unknown in the sense that we do not know the charts and the metric of $\mm$. The essential physical motions in the system $P$ are those slow time scale structural changes or conformational changes rather than the fast time scale motions such as vibrations.  Therefore, despite the high dimensionality of $P$ in practice, we can find some intrinsic low dimensional variables, called reaction coordinates, to represent those essential structural or conformational changes in a low dimensional space \cite{coifman2006diffusion, CKLMN, Mauro2011}. For instance,  a typical one dimensional reaction coordinate is the distance between a carbon center and a nucleophile in an $\text{S}_\text{N}2$ reaction (one simple  nucleophilic substitution reaction mechanism); {see also the conformational transitions of alanine dipeptide representing by two backbone dihedral angles \cite{gao21}.}   There are many other collective physical/chemical quantities, such as bond length/angle, dihedral angles, and intermolecular distance, that can be used as the reaction coordinates to represent the whole process of conformational transitions or chemical reactions.  Mathematically, the reaction coordinates should be a smooth embedding $\my=\Phi(\mathbf x): \mm\hookrightarrow \mathbb{R}^\ell$ with $\ell \ll p$ to preserve the topological structure of the underlying manifold.  Then, $\nn= \Phi(\mm)$ is a submanifold of  $\bR^\ell$ with the metric induced by the Euclidean metric of $\bR^\ell$. The reaction coordinates can be realized through the nonlinear dimension reduction algorithms {\cite{nadler2006diffusion}}. Suppose $\{\mathbf x_i\}_{i=1}^n$ are $n$ data points well distributed on the unknown manifold $\mm \subset \mathbb{R}^p$, while these $n$ points are collected by some sampling methods. A nonlinear dimension reduction algorithm constructs an embedding $\Phi$ by using the coordinates of $\{\mathbf x_i\}_{i=1}^n$ in $\mathbb{R}^p$ so that we can present the high dimensional data $\{\mathbf x_i\}_{i=1}^n\subset\mm\subset\bR^p$ as $\my_i=\Phi(\mx_i)\subset \nn\subset \bR^\ell$ in the low dimensional space.}

{ 
We assume the dynamics for the physical system $P$ can  be described by a continuous strong Markov process on $\mm\subset\bR^p$.
Particularly, the simplest and widely used physical model is the following  over-damped Langevin dynamics  with a drift determined by some potentials $U$ on $\mm$:
\begin{equation}
\label{sde-x}
\ud{\mx_t}= -\nabla_\mm U(\mx_t)\ud t + \sqrt{2kT} \ud_{\mm} B_t.
\end{equation}
We explain the notations  in \eqref{sde-x} below.
Let  $\{\tau^\mm_i;\, 1\leq i \leq d\}$ be an orthonormal basis of the tangent plane $T_{\mx_t}\mm$. Denote  $\nabla_\mm:= \sum_{i=1}^d \tau^{\mm}_i \nabla_{\tau^{\mm}_i}= \sum_{i=1}^d \tau^{\mm}_i \otimes \tau^{\mm}_i \nabla$ as { the  surface gradient } and $\nabla_{\tau^{\mm}_i}=\tau^{\mm}_i \cdot \nabla$ as the tangential derivative in the direction of $\tau^{\mm}_i$.
Let  $k$ be the Boltzmann constant, $T$ be the temperature \cite{Ebook} and $\ud_{\mm} B_t$ be a Brownian motion on $\mm$. This Brownian motion on manifold can be realized by the projection of the standard Brownian motion $B_t$ in the ambient space $\bR^p$:
\begin{equation}
\ud_\mm B_t :=\sum_{i=1}^d \tau^\mm_{i}(\mx_t)\otimes \tau^\mm_i(\mx_t) \circ \ud B_t,
\end{equation}
where $\circ$ is understood in the Stratonovich sense in the stochastic integral \cite[p.19, p.78]{hsu2002stochastic}. We refer the readers to  \cite{bakry2014analysis} for the general Langevin SDEs on Riemannian manifolds.

The potential $U(\mx)$ is also known as the energy landscape for the physical system $P$, which is usually very complicated and indicates all the possible (meta)stable states of a physical system. For instance, in a simple nucleophilic substitution reaction mechanism, the states of reactants and products are two stable states in the  energy landscape \cite{Ebook}. A saddle point state on the energy landscape is called the transition state, whose value determines the energy barrier in a chemical reaction.  In this paper, we assume the output of the potential $U$ at each data point $\{\mx_i\}_{i=1}^n$ is known.  The equilibrium  of this system $P$, also known as {the invariant} probability density measure, is $\rho_{\8}(\mathbf x) \propto e^{-\frac{U(\mathbf x)}{kT}},\, \mathbf x \in \mm$.

Suppose we learn the reaction coordinates $\my=\Phi(\mx), \mx \in \mm$. The diffeomorphism $\Phi:\mm\to \nn$ {induces} a map $\Phi_*$ from the space $\Gamma(T\mm)$ of the smooth vector fields on $\mm$ to the space $\Gamma(T\nn)$ of the smooth vector fields on $\nn$ such that for any $f\in C^{\8}(\nn)$ and $V\in \Gamma(T\mm)$
\begin{equation}
(\Phi_* V)f(\my)= V(f\circ \Phi)(\mx), \quad \my=\Phi(\mx).
\end{equation}
The Stratonovich formulation {transforms} consistently under diffeomorphism $\Phi$ \cite[p. 20]{hsu2002stochastic}.  {Notice $\tau^\nn_i\in \mathbb{R}^\ell$ is defined by the induced map $\Phi_*$ and $B_t$ is the $\ell$-dimensional Brownian motion.} 
Therefore, instead of considering \eqref{sde-x} on $\mm$ directly, we consider the SDE on $\nn$ 
\begin{equation}\label{sde-y}
\ud {\my_t} = - \nabla_\nn U_{\nn}(\my_t) \ud t + \sqrt{2kT} \sum_{i=1}^d \tau^\nn_{i}(\my_t)\otimes \tau^\nn_i(\my_t) \circ \ud B_t,
\end{equation}
where $\nabla_\nn:= \sum_{i=1}^d \tau^{\nn}_i \nabla_{\tau^{\nn}_i}= \sum_{i=1}^d \tau^{\nn}_i \otimes \tau^{\nn}_i \nabla$ is the surface gradient, $\nabla_{\tau^{\nn}_i}=\tau^{\nn}_i \cdot \nabla$ is the tangential derivative in the direction of $\tau^{\nn}_i$, and  $U_\nn$ is the induced potential on manifold $\nn$ by the composition
\begin{equation}
U_\nn(\my):=U(\mx)=U(\Phi^{-1}(\my)).
\end{equation}

The main goal of this paper is to accurately simulate the induced Langevin dynamics \eqref{sde-y} in terms of the reaction coordinates $\my$ and the information of the potential $U$. To simulate the SDE  \eqref{sde-y} on $\nn$ without exact manifold information,  one of the most natural ways is to construct an approximated stochastic process on the points $\{\my_i\}_{i=1}^n$. However, the standard Euler–Maruyama method on manifold  can not provide a stable simulation. Hence, our strategy for constructing a stochastic process over $\{\my_i\}_{i=1}^n$ is described as follows: (i) we detour to approximate the corresponding Fokker-Planck equation on the manifold with a finite volume scheme; (ii) we construct an approximated Voronoi tessellation associated with $\{\my_i\}_{i=1}^n$; (iii) we construct the transition probability and the jump rate from the finite volume scheme.

By Ito's formula, the SDE \eqref{sde-y} gives the following Fokker-Planck equation, which is the governing equation for the { density $\rhon_t:=\rho^\nn(\my,t)$},
\begin{equation}\label{FP-N}
\pt_t \rhon_t = \divn (kT\nabla_\nn \rhon_t + \rhon_t \nabla_\nn U_{\nn})=: \text{FP}^{\nn} \rhon_t,
\end{equation}
where $\divn$ is the surface divergence defined as $\divn \xi = \sum_{i=1}^d \tau_i^\nn \cdot \nabla_{\tau_i^\nn} \xi.$
One equivalent form of \eqref{FP-N} is the relative entropy formulation
\begin{equation}\label{FP-N-equivlant}
\pt_t \rhon_t = kT \divn \left( e^{-\frac{U_\nn}{kT}} \nabla_\nn (\rhon_t e^{\frac{U_\nn}{kT}}) \right)= \text{FP}^\nn \rhon_t.
\end{equation}

Then the main issue is  to
simulate the Fokker-Planck equation \eqref{FP-N} on $\nn$ whose solution  $\rho^{\nn}_t(\my)$ drives any initial density $\rho^{\nn}_0$ to the  invariant measure $\rho^{\nn}_\8(\my) \propto  e^{-\frac{U_{\nn}(\my)}{kT}}$.
After designing a finite volume scheme for the Fokker-Planck equation \eqref{FP-N} on $\nn$,   we construct an  approximated transition probability and jump rate from it.
This approximated Markov process on the manifold automatically incorporates both the manifold structure and the equilibrium  information. It enables some implementable data-driven computations on the manifold  such as finding the optimal cluster-cluster coarse-grained network,  cf. \cite{Deuflhard_Weber_2005, E_Li_Vanden-Eijnden_2008, Lafon06, li2009probabilistic, NoeLu11, Noe11} and  finding the transition path and energy landscape of chemical reactions, cf. \cite{ekeland1999convex, weinan2006, MMS2009, E_Vanden-Eijnden_2010, gao21, GL22t, GL22, gao2022selection}.

\subsection{Practical difficulties and mathematical implementations}
  The first difficulty is that we are not able to acquire all the information about the system $P$. Hence, we assume that we can sample $n$ points $\{\mx_i\}_{i=1}^n$ from $\mm$ based on a density function on $\mm$ with lower and upper bounds so that the data points are well distributed on $\mm$. In Section \ref{sec2}, { we first show that the diffusion map can approximate an embedding $\Phi$ of the manifold $\mm$. Then,} we apply the diffusion map algorithm  \cite{coifman2006diffusion} on $\{\mx_i\}_{i=1}^n$ to find the reaction coordinates so that we have $\{\my_i=\Phi(\mx_i)\}_{i=1}^n \subset \nn =\Phi(\mm) \subset \mathbb{R}^\ell$. Note that $\{\my_i\}_{i=1}^n$ can also be regarded as the samples based on a  density function on $\nn$ with lower and upper bounds.

Next,  we focus on simulating the Fokker-Planck equation \eqref{FP-N} with a given equilibrium potential $U_\nn(\my)$. To find the trajectory $\rho_t^\nn$, we need to solve the Fokker-Planck equation  on {the manifold} $\nn \subset \mathbb{R}^\ell$. Our method uses the data points $\{\my_i\}_{i=1}^n\subset\nn$ to construct a discrete approximation of the Fokker-Planck equation \eqref{FP-N}. It is proved that the data points are well-distributed on $\nn$  whenever the points are sampled based on a density function with lower and upper bounds \cite{Dejan15, liu2019rate}.  Hence, we can construct a ``regularly shaped'' Voronoi tessellation on $\nn$ from $\{\my_i\}_{i=1}^n\subset\nn$. With the help of such Voronoi tessellation, we introduce a finite volume scheme by applying { the} relative entropy formulation and finite volume method  to \eqref{FP-N}.  The finite volume  scheme assigns a transition probability and a jump rate for an approximated Markov process  on $\{\my_i\}_{i=1}^n$, i.e., random walk between the nearest neighbor points. In Section \ref{sec_mp}, we prove all the good properties of the approximated Markov process on  $\{\my_i\}_{i=1}^n$ including detailed balance, ergodicity, $L^1$-contraction and $\chi^2$-divergence dissipation law.  

To obtain an implementable finite volume scheme, an approximated Voronoi tessellation associated with $\{\my_i\}_{i=1}^n$ needs to be constructed with high accuracy. By using the Euclidean coordinates of $\{\my_i\}_{i=1}^n$, each Voronoi cell can be approximated by a polygon in a tangent space of $\nn$; {see Section \ref{sec5.3} for detailed error estimates for the approximated cell volume and face area.}  Therefore, an approximated transition probability based on the volume of each polygon and the areas of its faces can be assigned over $\{\my_i\}_{i=1}^n$ and leads to an implementable finite volume  scheme \eqref{mp_pron} for the Fokker-Planck equation \eqref{FP-N}; see Section \ref{sec5} and Theorem \ref{mainthm_mp} for consistence and convergence analysis for this implementable finite volume  scheme. We also provide an unconditionally stable explicit time descretization for the finite volume  scheme  based on the  detailed balance property of the Markov process in Section \ref{sec_explicit}. { This explicit scheme is very efficient and enjoys a mass conservation law, unconditional maximum principle and exponential convergence to  equilibrium.  At last, to show the accuracy of the finite volume  scheme, we simulate challenging numerical examples including datasets on a dumbbell, a Klein bottle and a sphere in Section \ref{sec_simu}}.

 The approximated transition probability between the nearest neighbor points  for the Markov process on $\{\my_i\}_{i=1}^n\subset\nn$ reveals the manifold structure and enables us to efficiently conduct computations such as { clustering, coarse-graining and finding the minimal energy path on the manifold.}   Notice this transition probability between the nearest neighbor points not only incorporates the manifold information but also gives an adapted graph network on the manifold.

The remaining part of the  paper will be organized as follows. In Section \ref{sec2}, we use diffusion map
to learn the reaction coordinates $\my$. In Section \ref{sec5}, we  propose { a data-driven solver for the Fokker-Planck equation on  manifold $\nn$, which assigns an approximated transition probability and a jump rate for an approximated Markov process  on $\{\my_i\}_{i=1}^n$.  In Section \ref{sec_simu}, we also provide several simulation results for the Fokker-Planck dynamics on manifolds learned from point clouds.} All the technical lemmas are provided in Appendix for completeness.   All the commonly used notations are listed in Table \ref{Table:Notations} for the sake of clarity.
}

\begin{table}[ht]
\caption{Commonly used notations in this paper.}\label{Table:Notations}
\begin{tabular}{|l|l|} 
\hline $Symbols$ & $Meaning$\\ 
\hline\hline 
$\mathbb{R}^p$, $\mathbb{R}^\ell$  & High (low) dimensional ambient spaces \\ 
$d$ & Dimension of the Riemannian manifolds \\ 
$\mm$, $\nn$ &  $d$-dimensional smooth closed Riemannian submanifolds of the Euclidean spaces\\ 
$\mx$, $\my$ & Points on $\mm$ and $\nn$ respectively\\
$dV_{\mm}$, $dV_{\nn}$ & Volume forms on $\mm$ and $\nn$ respectively \\
$\Delta$ & Laplace Beltrami operator of a manifold \\
$\lambda_i$, $\psi_i$ & The eigenvalues and the corresponding orthonormal (in $L^2$) eigenfunctions of $\Delta$ \\
$\Phi$ & Reaction coordinates (Smooth embedding of a manifold) \\ 
$X$, $Y$ & Random variables with the range $\mm$ and $\nn$ respectively \\
{$\rho^*$, $\rho^{\mm}_t$}  & Probability density functions on $\mm$ \\
{$\rho^{**}$, $\rho^{\nn}_t$} & Probability density functions on $\nn$ \\
$n\in\mathbb{N}$ & Number of data points sampled from $\mm$ based on $\rho$ \\ 
$\{\mx_1, \cdots, \mx_n\}$ & Data points sampled from $\mm$ based on $\rho$ \\ 
$\epsilon$ & The bandwidth in the diffusion map \\
$K_\epsilon$ & Kernel in the diffusion map \\
$W_{\epsilon, \alpha}$ & Affinity matrix in diffusion map with $\alpha$ normalization\\
$L_{\epsilon, \alpha}$  & diffusion map matrix \\
$\lambda_{i,n,\epsilon}$, $v_{i,n,\epsilon}$ & The eigenvalues and the corresponding orthonormal eigenvectors in $l^2$ of { $\frac{I-L_{\epsilon,1}}{\epsilon^2}$}\\
$C_i$ & the Voronoi cell around the point $\my_i$ on the manifold $\nn$ \\
$\Gamma_{ij}$ & the Voronoi face between $\my_i$ and $\my_j$ on the manifold $\nn$ \\
$r$ & bandwidth in the  Voronoi cell approximation algorithm \\
$s$ & threshold in the  Voronoi cell approximation algorithm \\
$\iota_k$ & the projection map in the  Voronoi cell approximation algorithm \\
$P_{ij}, \lmd_i$ & the transition probability and jump rate of constructed Markov chain \\
\hline 
\end{tabular}
\end{table}

\section{Review of nonlinear dimension reduction and diffusion map}\label{sec2}
In this section, we focus on learning the reaction coordinates $\my$ for the $d$-dimensional manifold $\nn \subset \mathbb{R}^\ell$ to extract the  conformational changes with slow time scale from other fast time scale vibrations.
We first introduce the basic idea about the nonlinear dimension reduction under the following assumption.
\begin{assumption}\label{assumption DM}
Let $\mm$ be a $d$ dimensional smooth closed Riemannian submanifold of $\mathbb{R}^p$. Suppose that $\rho^*$ is a smooth probability density function on the manifold $\mm$. We assume that $\rho^*$ is bounded from below and from above, i.e. $\rho_m \leq \rho* \leq \rho_M$.  Let $\{\mx_1 \cdots, \mx_n\} \subset \mm \overset{i.i.d.}{\sim} \rho^*$.
\end{assumption}

Nonlinear dimension reduction algorithms construct maps which map  $\{\mx_1 \cdots, \mx_n\}$ to some low dimensional space $\mathbb{R}^\ell$ {while preserving the topological or geometric structure of the underlying manifold.} There are a lot of well known dimension reduction algorithms, for instance, ISOMAP \cite{Tenenbaum_deSilva_Langford:2000}, eigenmap \cite{Belkin_Niyogi:2003}, locally linear embedding (LLE) \cite{Roweis_Saul:2000} and its variations like Hessian LLE \cite{Donoho_Grimes:2003}, vector diffusion map \cite{singer2012vector, singer2016spectral}. In this work, we focus on the algorithm diffusion map which is introduced by Coifman and Lafon \cite{coifman2006diffusion}.  The algorithm of the diffusion map can be described in the following steps:
\begin{enumerate}[(i)]
\item
For $\mx,\mx' \in \mm$, we define $K_{\epsilon}(\mx, \mx')=\exp(-\frac{\|\mx-\mx'\|^2_{\mathbb{R}^p}}{4\epsilon^2})$, where $\epsilon>0$ is the bandwidth.
\item 
Define 
$
q_{\epsilon}(\mx):=\sum_{i=1}^{n} K_{\epsilon}(\mx,\mx_i).
$
We define the  affinity matrix which is { the} $n \times n$ matrix $W_{\epsilon,\alpha}$:
$
W_{\epsilon,\alpha,ij}\,:=\frac{K_{\epsilon}(\mx_i,\mx_j)}{q^\alpha_{\epsilon}(\mx_i) q^\alpha_{\epsilon}(\mx_j)}.
$
This step is called the $\alpha$-normalization.
\item
Define {the $n \times n$ diagonal matrix $D$ with diagonal entries}
$
D_{\epsilon,\alpha,ii}=\sum_{j=1}^{n} W_{\epsilon,\alpha,ij}.
$
Let 
\begin{align}\label{L matrix}
L_{\epsilon,\alpha}=D_{\epsilon,\alpha}^{-1}W_{\epsilon,\alpha}.
\end{align}
\item
To reduce the dimension of the dataset $\{\mx_1 \cdots, \mx_n\} $. We choose $\alpha=1$. Denote 
\begin{align}
\lambda_{0,n,\epsilon} \leq \cdots, \leq \lambda_{n-1,n,\epsilon}
\end{align}
to be the eigenvalues of $\frac{I-L_{\epsilon,1}}{\epsilon^2}$. We  find the first $\ell$  corresponding eigenvectors of $\frac{I-L_{\epsilon,1}}{\epsilon^2}$, namely, $\{v_{j,n,\epsilon}\}^\ell_{j=1}$. Then the map
\begin{align}
\mx_i \rightarrow (v_{1,n,\epsilon}(i), \cdots, v_{\ell,n,\epsilon}(i))
\end{align} 
reduces the dimension of the dataset into the Euclidean space $\mathbb{R}^\ell$.
\end{enumerate}

\begin{rem}
Note that the matrix $L_{\epsilon,1}$ in \eqref{L matrix} may not be symmetric in general. Therefore, in the implementation,  we use the matrix $\tilde{L}_{\epsilon,1}=D_{\epsilon,1}^{-1/2}W_{\epsilon,1} D_{\epsilon,1}^{-1/2}$. $\tilde{L}_{\epsilon,1}$ is similar to $L_{\epsilon,1}$ and is symmetric. Therefore, they share the same eigenvalues and the corresponding eigenvectors differ by $D_{\epsilon,1}^{-1/2}$. 
\end{rem}

{Let $\Delta$ be the Laplace-Beltrami operator of $\mm$.  Let $\{\lambda_i\}_{i=0}^\infty$ be the eigenvalues of $-\Delta$, and 
\begin{align}
\Delta \psi_i =-\lambda_i \psi_i,
\end{align}
where $\psi_i$ is the corresponding eigenfunction normalized in $L^2(\mm)$. We have $0=\lambda_0 < \lambda_1\leq \lambda_2 \leq  \cdots$. Note that $\psi_0=\frac{1}{\sqrt{\mm}}$ is a constant. 

In the rest of this section, we will provide a justification that the diffusion map 
\begin{align}
\mx_i \rightarrow (v_{1,n,\epsilon}(i), \cdots, v_{\ell,n,\epsilon}(i))
\end{align} 
approximates an embedding of $\mm$ into a Euclidean space.  The justification consists of two steps. First, we review the results about the spectral convergence from $\frac{I-L_{\epsilon,1}}{\epsilon^2}$  to $-\Delta$. Intuitively, these results show that the eigenpairs of  $\frac{I-L_{\epsilon,1}}{\epsilon^2}$ approximate the corresponding eigenpair of $-\Delta$.  Second, we discuss the result that shows the eigenfunctions of $-\Delta$ can be used to construct an embedding of $\mm$. Since $\psi_0$ is a constant, based on the justification, the first eigenvector $v_{0,n,\epsilon}$  of $\frac{I-L_{\epsilon,1}}{\epsilon^2}$ is not used in the construction of the diffusion map.

We start from the theoretical results that relate the diffusion map to the Laplace Beltrami operator when the samples are from a submanifold. It is proved in \cite{coifman2006diffusion} and \cite{singer2012vector} that  $\frac{I-L_{\epsilon,1}}{\epsilon^2}$ converges pointwisely to $-\Delta$ in the following sense.
\begin{thm}\label{pointwise convergence}(Coifman-Lafon,\cite{coifman2006diffusion}, Singer-Wu,\cite{singer2012vector})
Suppose $\alpha=1$. Under Assumption \ref{assumption DM}, for $f \in C^3(\mm)$, if $\frac{\sqrt{\log n}}{\sqrt{n}\epsilon^{\frac{d}{2}+2}} \rightarrow 0$ and $\epsilon \rightarrow 0$ as $n \rightarrow \infty$, then with probability greater than $1-\frac{1}{n^2}$, for all $i=1, \cdots, n$, we have
\begin{align}
\frac{f(\mx_i)-\sum_{j=1}^n L_{\epsilon,1}(i,j)f(\mx_j)}{\epsilon^2}=-\Delta f(\mx_i)+O(\epsilon)+O(\frac{\sqrt{\log n}}{\sqrt{n}\epsilon^{\frac{d}{2}+2}}).
\end{align}
\end{thm}
The $\alpha=1$ normalization in the diffusion map comes from the idea of density estimation. When $\alpha$ is chosen to be $1$,  the impact of the nonuniform density $\rho^*$ is removed. Hence, the Laplace-Beltrami operator in the previous theorem is not contaminated by the probability density function $\rho^*$.

A stronger version of the convergence theorem in \cite{singer2016spectral} shows the spectral convergence of the diffusion map in $L^2$ sense. At last, in \cite{dunson2019spectral}, it shows the  $L^\infty$ spectral convergence result based on the following definition. 
\begin{defn}\label{1/rho normalization}
Under Assumption \ref{assumption DM}, suppose $v_{j,n,\epsilon}$ is an eigenvector of $\frac{I-L_{\epsilon,1}}{\epsilon^2}$ which is normalized in the $l^2$ norm. Let 
$N_k=|B^{\mathbb{R}^p}_{\epsilon}(\mx_k) \cap \{\mx_1, \cdots, \mx_n\}|$, the number of points in the $\epsilon$ ball in the ambient space.
Then, we define the $l^2$ norm of $v_{j,n,\epsilon}$ with respect to the inverse estimated probability density $1/\hat{\rho^*}$ as: 
\begin{align}
\|v_{j,n,\epsilon}\|_{l^2(1/\hat{\rho^*})}:=\sqrt{\frac{|S^{d-1}|\epsilon^d}{d} \sum_{k=1}^n \frac{v_{j,n,\epsilon}(k)}{N_k}}, \nonumber 
\end{align}
where $|S^{d-1}|$ is the volume of the $d-1$ dimensional standard sphere.
Define the renormalization of $v_{j,n,\epsilon}$ in the $l^2$ norm  with respect to the inverse estimated probability density $1/\hat{\rho^*}$ as: 
\begin{align} \label{normalized v i n epsilon}
V_{j,n,\epsilon}:=\frac{v_{j,n,\epsilon}}{\|v_{j,n,\epsilon}\|_{l^2(1/\hat{\rho^*})}}.  
\end{align}

\end{defn}

Intuitively, $v_{j,n,\epsilon}$ is a  discretization of some function on $\mm$ while $\|v_{j,n,\epsilon}\|_{l^2(1/\hat{\rho^*})}$ is an approximation of the $L^2(\mm)$ norm of the function. Hence, $V_{j,n,\epsilon}$ can be regarded as a discretization of some function that is normalized in  $L^2(\mm)$. On the other hand, the vector $\vec{\psi}_j=(\psi_j(\mx_1), \cdots, \psi_j(\mx_n))^\top$ is a discretization of $\psi_j$ which is also  normalized in the $L^2(\mm)$. Therefore, it is reasonable to compare $V_{j,n,\epsilon}$ and $\vec{\psi}_j$  rather than $v_{j,n,\epsilon}$ and $\vec{\psi}_j$. In the following theorem, it shows that, on $\mm$, if we fix $K$, the bandwidth $\epsilon$ is small enough based on $K$ and the number of data points $n$ is large enough based on $\epsilon$,  then for all $0 \leq j < K$,  with high probability, $\lambda_{j,n,\epsilon}$ is  an approximation of the $j$-th eigenvalue $\lambda_j$ of $-\Delta$ and $V_{j,n,\epsilon}$ is an approximation of $\vec{\psi}_j$ . 

\begin{thm} \label{spectral convergence of L on closed manifold} (Dunson-Wu-Wu, \cite{dunson2019spectral}) Under Assumption \ref{assumption DM}, suppose all the eigenvalues of $\Delta$ are simple. Let $(\lambda_j, \psi_j)$ be the j-th eigenpair of $-\Delta$ with $\psi_j$ normalized in $L^2(\mm)$. Let $L_{\epsilon,1}$ be the matrix in \eqref{L matrix}. Let $(\lambda_{j,n,\epsilon}, V_{j,n,\epsilon})$ be the j-th eigenpair of $\frac{I-L_{\epsilon,1}}{\epsilon^2}$ with $ V_{j,n,\epsilon}$ normalized as in  Definition \ref{1/rho normalization}.  Fix a positive integer $K$, let $\mathsf \Gamma_K:=\min_{1 \leq j \leq K}\textup{dist}(\lambda_j, \sigma(-\Delta)\setminus \{\lambda_j\})$, where $\sigma(-\Delta)$ is the spectrum of $-\Delta$. Suppose  
\begin{align}
\epsilon \leq \mathcal{K}_1 \min \Bigg(\bigg(\frac{\min(\mathsf\Gamma_K,1)}{\mathcal{K}_2+\lambda_K^{d/2+5}}\bigg)^2, \frac{1}{(\mathcal{K}_3+\lambda_K^{(5d+7)/4})^2}\Bigg),  \label{relation epsilon and K}
\end{align}
where  $\mathcal{K}_1$ and $\mathcal{K}_2,\mathcal{K}_3 >1$  are constants depending on $d$, the lower bound of the p.d.f. $\rho_m$, the $C^2$ norm of p.d.f. , the volume, the injectivity radius, the  curvature, and the second fundamental form of the manifold.

If $n$ is sufficiently large so that $\epsilon=\epsilon(n) \geq (\frac{\log n}{n})^{\frac{1}{4d+13}}$,  then with probability greater than $1-n^{-2}$, 
 for all $0 \leq j < K$  
\begin{align}
|\lambda_{j,n,\epsilon}-\lambda_{j}|\leq \mathcal{K}_4 \epsilon^{3/2}. \nonumber
\end{align}

If $n$ is sufficiently large so that $\epsilon=\epsilon(n) \geq (\frac{\log n}{n})^{\frac{1}{4d+8}}$,  then with probability greater than $1-n^{-2}$, 
 there are $a_j \in \{1,-1\}$ such that  for all $0 \leq j < K$
 \begin{align}
\max_{1 \leq i\leq n}|a_j V_{j,n,\epsilon}(i)-\psi_{j}(x_i)|\leq  \mathcal{K}_5 \epsilon^{1/2}. \nonumber 
\end{align}
$\mathcal{K}_4$  depends  on $d$, the diameter of the manifold and the lower bound and the $C^2$ norm of the p.d.f..  $\mathcal{K}_5$  depends  on $d$, the diameter and the volume of the manifold,  and the lower bound and the $C^2$ norm of the p.d.f..   
\end{thm}

\begin{rem} Note that in the above theorem,  the coefficients $\mathcal{K}_4$ and $\mathcal{K}_5$ only depend on the geometry of the manifold and the data points distribution on the manifold. They are independent of the eigenvalues and the eigengaps of $\Delta$. In the spectral convergence analysis,  the dependence on  the eigenvalues and the eigengaps of $\Delta$ is reflected from the relation \eqref{relation epsilon and K}.
The relation implies that $\epsilon$ should be smaller when $K$ increases. 

Moreover, the above theorem assumes that the eigenvalues of $\Delta$ are simple for notational simplicity. In the case when the eigenvalues are not simple, the same theorem still works by introducing an orthogonal transformation on the eigenspace. The readers may refer to Remark 4  in \cite{cheng2021eigen} for details. 
\end{rem}

The matrix $\frac{I-L_{\epsilon,1}}{\epsilon^2}$ can be regarded as the density corrected graph Laplacian on the complete undirected graph with vertices $\{\mx_1 \cdots, \mx_n\}$ and Gaussian weights on the edges. Hence, the above theorem discusses the spectral convergence of a density corrected graph Laplacian to the Laplace-Beltrami operator in the $L^\infty$ sense. We also refer the readers to \cite{trillos2020error, calder2019improved, cheng2021eigen} which discuss the spectral convergence rates of the graph Laplacians to the Laplace Beltrami operator in the $L^2$ sense and \cite{calder2020lipschitz} which is another work discussing  the spectral convergence rate of the graph Laplacian to the Laplace-Beltrami operator in the $L^\infty$ sense. 

Next, we review some results in spectral geometry. Based on the work of \cite{berard1994embedding}, \cite{jones2008manifold}, \cite{bates2014embedding} and \cite{portegies2016embeddings}, we know that the eigenfunctions of $\Delta$ can be used to construct an embedding of the manifold into a Euclidean space. More explicitly, we describe the following theorem in \cite{bates2014embedding}. The readers may refer to Appendix \ref{appendix A} for more detailed discussions about the relevant theorems in \cite{berard1994embedding}, \cite{jones2008manifold} and \cite{portegies2016embeddings}. 

\begin{thm}(Bates, \cite{bates2014embedding})\label{Bates embedding}
Suppose $\mm$ is a $d$ dimensional smooth closed Riemannian manifold.  Suppose that the Ricci curvature of $\mm$ has lower bound $\ric_{\mm} \geq (d-1)k$,  the injectivity radius of $\mm$ has lower bound $i(\mm) \geq i_0$ and the volume of $\mm$ has upper bound $\vol(\mm) \leq V$.  There is a $C(d, k, i_0, V)$ such that if $q\geq C$, for $\mx \in \mm$
\begin{align}
\Psi_q(\mx)=(\psi_1(\mx), \cdots,\psi_q(\mx)), 
\end{align}
is a smooth embedding of $\mm$ into $\mathbb{R}^q$.
\end{thm}

Recall that $\psi_0$ is a constant, so it is not used in the construction of the embedding. Based on the above theorem, let $\ell$ be the smallest integer $q$ such that $\Psi_q(\mx)$ is an embedding and we define
 \begin{align}
\Psi(\mx)=(\psi_1(\mx), \cdots,\psi_\ell(\mx)). 
\end{align}
Hence,  we have $d \leq \ell \leq C(d, k, i_0, V)$. In other words, $\ell$ can be bounded above by the dimension, Ricci curvature lower bound, the injectivity radius lower bound and the volume upper bound of the manifold $\mm$. 

Based on Definition \ref{1/rho normalization} and Theorem \ref{spectral convergence of L on closed manifold}, we provide the following definition of the reaction coordinates which we use in this work. 
\begin{defn}[Reaction coordinates] \label{definition of the reaction coordinates}
Let $(\lambda_i, \psi_i)$ be the i-th eigenpair of the Laplace Beltrami operator on $\mm$, $-\Delta$, with $\psi_j$ normalized in $L^2(\mm)$.  Suppose for $x \in \mm$
\begin{align}
\Psi(\mx)=(\psi_1(\mx), \cdots,\psi_\ell(\mx)),  
\end{align}
is a smooth embedding of $\mm$ into $\mathbb{R}^\ell$. Let $A$ be a $\ell \times \ell$ diagonal matrix such that { $A_{jj}=a_j\|v_{j,n,\epsilon}\|_{l^2(1/\hat{\rho^*})}$ where $\|v_{j,n,\epsilon}\|_{l^2(1/\hat{\rho^*})}$ is defined in Definition \ref{1/rho normalization} and $a_j$ is defined in Theorem \ref{spectral convergence of L on closed manifold}.} Under Assumption \ref{assumption DM}, we define 
\begin{align}
\my_i = \Phi(\mx_i) := A \circ \Psi(\mx_i) \in \mathbb{R}^\ell,
\end{align}
to be the reaction coordinates of $\mx_i$
\end{defn}

Note that $ A \circ \Psi $  is also a smooth embedding of $\mm$ into $\mathbb{R}^\ell$. Hence, by Theorem \ref{spectral convergence of L on closed manifold}, we have a justification of the diffusion map. Let $\{v_{1,n,\epsilon}, \cdots, v_{\ell,n,\epsilon}\}$ be the first $\ell$ eigenvectors of $ \frac{I-L_{\epsilon,1}}{\epsilon^2}$ in Step (iv) of the algorithm of the diffusion map. Then,  the diffusion map 
\begin{align}
\mx_i \rightarrow (v_{1,n,\epsilon}(i), \cdots, v_{\ell,n,\epsilon}(i)),
\end{align} 
is an approximation of $\my_i = \Phi(\mx_i) := A \circ \Psi(\mx_i)$ over the data points $\{\mx_1 \cdots, \mx_n\}$.

\

Although the diffusion map is applied to construct the reaction coordinates in this work,  it also can be used to solve the Fokker-Planck equations on $\mm$. Under Assumption \ref{assumption DM},  for $f \in C^3(\mm)$, it is shown in \cite{coifman2006diffusion} that $\frac{I-L_{\epsilon, \alpha}}{\epsilon^2}$ converges pointwisely (in the sense of Theorem \ref{pointwise convergence}) to the Kolmogrov backward operator 
$$\mathcal{L}_{\alpha}f =-\Delta f+2(1-\alpha) \nabla U\cdot \nabla f,$$
where $U=-\log \rho^*$ and $\rho^*$ is the unknown sample density defined in Assumption \ref{assumption DM}.
Hence,  the eigenpairs of  $\frac{I-L_{\epsilon, \alpha}}{\epsilon^2}$  approximate the corresponding eigenpairs of $\mathcal{L}_{\alpha}$.  When $\alpha=1/2$, let $\mathcal{L}=\mathcal{L}_{1/2}$. Let 
$$\mathcal{L}^*f=-\Delta f- \nabla \cdot (f \nabla U)$$
be the Kolmogrov forward (Fokker-Planck) operator.  Then, $\mathcal{L}$ and $\mathcal{L}^*$ share the same eigenvalues and their eigenfunctions differ by a factor $\rho^*$. The solution to the Fokker-Planck equation 
$$\partial \rho_t= -\mathcal{L}^* \rho_t$$
can be expressed as a series sum in terms of the eigenpairs of the $\mathcal{L}^*$. The coefficients in the series are determined by the projection of the initial condition onto each eigenspace.  Suppose the unknown sample density $ \rho^*$ is approximated through a density estimation.  Then, the eigenpairs of $\mathcal{L}^*$ over the sample points can be approximated by using the eigenpairs $\frac{I-L_{\epsilon, 1/2}}{\epsilon^2}$. In \cite{berry2015nonparametric1,berry2015nonparametric2}, the authors construct an approximation of the solution to the Fokker-Planck equation by using the spectral method and they explore the dynamics on the manifold.  The solution is constructed by projecting the discretization of the initial condition onto the approximation of the finite dimensional eigenspaces of $\mathcal{L}^*$. 
It is worth mentioning that the setup and methods applied in our work are different from \cite{coifman2006diffusion,berry2015nonparametric1,berry2015nonparametric2}.  First, we will use the diffusion map to find the reaction coordinates and reconstruct a manifold $\nn$ in a low dimensional space. As we describe in the introduction, instead of solving the Fokker-Planck equation on $\mm$ in the high dimensional space, we will solve the Fokker-Planck equation on $\nn$. Second, we assume that the equilibrium potential $U_{\nn}$ in the Fokker-Planck equation on $\nn$  is equal to $-\log \rho^{\nn}_{\infty}$, where $\rho^{\nn}_{\infty}$ is a known equilibrium density. However, it is not necessary that $\rho^{\nn}_{\infty}$ is equal to the sample density on $\nn$.  At last, we will propose a finite volume scheme rather than apply the spectral method to solve the Fokker-Planck equation .

}

\section{Solution to the Fokker-Planck equation on $\nn$ }\label{sec5}
Suppose $\nn$ is a $d$ dimensional smooth closed Riemannian submanifold of  $\mathbb{R}^\ell$ with the coordinates $\my$ obtained in Section \ref{sec2}.  In this section, given an  equilibrium potential  $U_{\nn}(\my)$,  we will focus on designing a data-driven solver for   the Fokker-Planck equation on $\nn$ which drives any initial data $\rho_0$ to   the equilibrium density on $\nn$,  $\rho^{\nn}_\8(\my)\propto e^{-U_{\nn}(\my)}$ (after taking $kT=1$). To study the trajectory of $\rho_t$ driving any initial data $\rho_0$ to the equilibrium density $\rho^{\nn}_\8(\my)$, it is sufficient to solve the following Fokker-Planck equation on manifold $\nn$
\begin{equation}\label{FP_N1}
\pt_t \rhon_t = \divn (\nabla_\nn \rhon_t + \rhon_t \nabla_\nn U_{\nn}).
\end{equation} 
For {notational} simplicity, in the {remainder} of this section, we will denote the equilibrium density for the  Fokker-Planck equation \eqref{FP-N} as $\pi(\my):= \rho_\8^\nn(\my)$.

{ As mentioned before, since we do not have exact information of $\nn$, the only implementable method is to use the  data $\{\my_i\}_{i=1}^n\subset\nn$  to construct directly a good discrete approximation to the Fokker-Planck equation \eqref{FP_N1}. We know that $\my_i = \Phi(\mx_i)$, where $\Phi$ is the reaction coordinates defined in Definition \ref{definition of the reaction coordinates} and $\{\mx_i\}_{i=1}^n$ are the samples on $\mm$ based on the density function $\rho^*$ in Assumption \ref{assumption DM}. Hence,  $\{\my_i\}_{i=1}^n$ are the samples on $\nn$ based on a density function $\rho^{**}$, where $\rho^{**}$ is the induced density function of $\rho^*$ by the reaction coordinates $\Phi$. Since $\rho^*$ has  an upper bound and a positive lower bound, $\rho^{**}$ also has an upper bound and a postive lower bound.  It can be proved that  $\{\my_i\}_{i=1}^n$ are well-distributed on $\nn$ when they are sampled based on a density function with such bounds \cite{Dejan15, liu2019rate}.}

In Section \ref{sec_5.1}, we will construct a Voronoi tessellation for $\nn$ from $\{\my_i\}_{i=1}^n\subset\nn$ and then assign the transition probability for an approximated Markov process on $\{\my_i\}_{i=1}^n$ between the nearest neighbor points. This transition probability with detailed balance property also gives {a}  finite volume  scheme for { solving the Fokker-Planck} equation \eqref{FP_N1}. We give the  stability and convergence analysis for this scheme in Section \ref{sec_5.2}. However, without the exact metric on $\nn$, to propose an implementable scheme, the Voronoi tessellation needs to be further approximated. In Section  \ref{sec5.3},     thanks to the metric on $\nn$  induced by the low dimensional Euclidean distance in $\bR^{\ell}$,  the volumes of the Voronoi cells and the areas of the  Voronoi faces    can be  further approximated by polygons in its tangent plane in $\bR^\ell$ with high order accuracy. Therefore the  new transition probability based on polygons  can be assigned and leads to an implementable  finite volume  scheme for Fokker-Planck equation \eqref{FP_N1}; see Theorem \ref{mainthm_mp}.  In Section \ref{sec_explicit}, we provide an unconditionally stable explicit time discretization for the finite volume  scheme  based on the  detailed balance property of the Markov process, which satisfies {a mass conservation law and  exponentially converges to the equilibrium. As a consequence, we obtained an approximated Markov process  on $\{\my_i\}_{i=1}^n$, i.e., random walk between the nearest neighbor points with  an approximated transition probability and  jump rate that enjoys { good properties such as the conservation of mass, $L^1$ contraction for the forward equation, $L^\8$ maximal principle for the backward equation and the $L^2$ error estimates.}}

\subsection{Construction of the Voronoi tessellation and the finite volume  scheme  on manifold $\nn$}\label{sec_5.1}
In this section, we construct a finite volume scheme based on the Voronoi tessellation for manifold $\nn$. We will see the advantage is that the  Voronoi tessellation automatically gives a {positivity-preserving} finite volume  scheme for the Markov process with detailed balance; see Lemma \ref{lem_Markov}.

{Suppose $(\nn, d_\nn)$ is a $d$ dimensional smooth closed submanifold of $\mathbb{R}^\ell$ and $d_{\nn}$ is  induced by the Euclidean metric in $\bR^\ell$. Let $S \subset \nn$. We have
\begin{align}
\hs^k_{\delta}(S)= \inf \{\sum_{i=1}^\infty \diam(U_i)^k,  S \subset \cup_{i=1}^\infty U_i, \diam(U_i) < \delta \},
\end{align}
where the infimum is taken over all countable covers of $S$ in $\nn$ and the diameter of the set $U_i$ is measured in metric $d_\nn$.
Then,  the $k$ dimensional Hausdorff measure $\hs^k(S)$ of  $S$ in $\nn$ is defined as 
\begin{align}
\hs^k(S)= \lim_{\delta \rightarrow 0} \hs^k_{\delta}(S).
\end{align}
}
{For the samples $\{\my_i\}_{i=1}^n \subset \nn$, we define the Voronoi cell as}
\begin{equation}
C_i:= \{\my\in \nn ; \ud_\nn(\my,\my_i)\leq \ud_\nn(\my,\my_j) \text{ for all }\my_j,\,\, j=1, \cdots, n\},
\end{equation} 
with the volume $|C_i|=\hs^d(C_i)$.
Then $\nn=\cup_{i=1}^n  C_i$ is a Voronoi tessellation of manifold $\nn$. Denote the Voronoi face for cell $C_i$ as 
\begin{equation}
\Gamma_{ij}:= C_i\cap C_j,  \text{ and its area as } |\Gamma_{ij}|=\hs^{d-1}(\Gamma_{ij})
\end{equation} 
for any $j=1, \cdots, n$. If $\Gamma_{ij}= \emptyset$ or {$i = j$} then we set $|\Gamma_{ij}|=0$. We define the bisector between $\my_i$ and $\my_j$ to be the set 
\begin{equation}
G_{ij}:= \{\my \in \nn ; \ud_\nn(\my,\my_i) = \ud_\nn(\my,\my_j) \}.
\end{equation} 
Obviously, we have $\Gamma_{ij} \subset G_{ij}$.

Define the associated adjacent sample points as 
\begin{equation}
\text{VF}(i):=\{j; ~\Gamma_{ij}\neq \emptyset\}.
\end{equation}

{First, we have the following basic facts about the Voronoi cells on a manifold.}
\begin{prop}
If $C_i$ is the Voronoi cell containing the $\my_i$ and $C_i$ is contained in the geodesic ball centered at $\my_i$ whose radius is equal to the injectivity radius of $\nn$ at $\my_i$, then $C_i$ is star shaped.
\end{prop}

\begin{proof}
For any $\my \in C_i$, if there is a point $\my'$ on the minimizing geodesic between $\my$ and $\my_i$ such that $\my'  \not\in C_i$ and $\my' \in C_j$, then $d_\nn(\my', \my_j) < d_\nn(\my', \my_i)$. Therefore, we have $d_\nn(\my,\my_j) \leq d_\nn(\my,\my')+d_\nn(\my',\my_j)< d_\nn(\my,\my')+d_\nn(\my',\my_i) =d_\nn(\my,\my_i)$. This contradicts to $\my \in C_i$. Hence, any point  on the minimizing geodesic between $\my$ and $\my_i$ is in $C_i$.
\end{proof}

{Note that the above fact holds regardless how $\{\my_i\}_{i=1}^n$ are sampled. 

\

Next, we want to discuss the geometric properties of the Voronoi faces. We start from the bisector between two points. A natural question is whether a bisector between two points on a closed $d$ dimensional manifold is a $d-1$ dimensional submanifold. Unfortunately, the answer is negative for two arbitrary points due to the topological and geometrical structure of the manifold.  In Figure \ref{bisectorexample}, we show an example that on  a manifold diffeomorphic to a torus, the bisector between two points $A$ and $B$ is a figure ``8'' curve.  
\begin{figure}[h!]
\centering
\includegraphics[width=0.3\columnwidth]{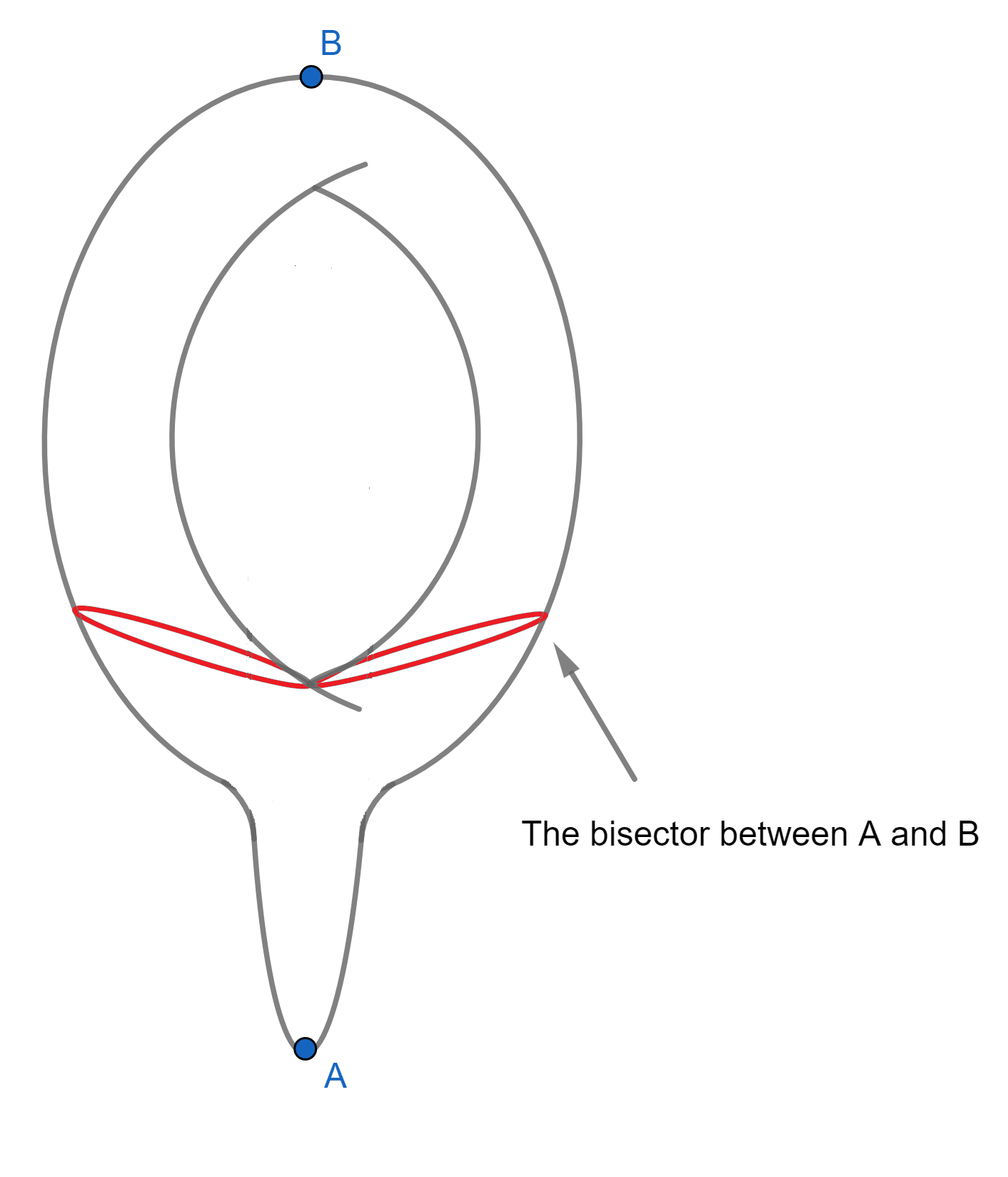}
\caption{An example in which the bisector between two points is not a submanifold. }\label{bisectorexample}
\end{figure}
There are special cases when the bisector between any two points is a submanifold globally. In \cite{beem1975pseudo}, the author proves that any bisector between two points is a totally geodesic submanifold if and only if the manifold has constant curvature. An obvious example of this case is  the round sphere, where any bisector is a great hypersphere. Hence, on the round sphere,  any Voronoi surface is a part of a great hypersphere. 

In this work, we prove the following local regularity result for the bisectors on any manifold. In fact, we show that when two points are close enough, then an open neighborhood on the bisector around the midpoint of the minimizing geodesic connecting those two points is a $d-1$ dimensional submanifold.   The proof of the proposition with a figure to illustrate the statement of the proposition is in Appendix \ref {proof local regularity of VFace}.

\begin{prop}\label{perpendicular to Voronoi face}
Suppose $\delta$ is small enough depending on the bounds of the sectional curvatures and the injectivity radius of $\nn$. For any $\my_i \in \nn$, let $B_{\delta}(\my_i)$ be an open geodesic ball of radius $\delta$ at $\my_i$. Suppose $\my_j \in B_{\delta}(\my_i)$ and $G_{ij}$ is the bisector between $\my_i$ and $\my_j$. Then, $M_{ij}=B_{\delta}(\my_i) \cap G_{ij}$ is a $d-1$ dimensional submanifold. Let $\my^*$ be the midpoint of the minimizing geodesic between $\my_i$ and $\my_j$. Then $\my^* \in M_{ij}$ and the minimizing geodesic between $\my_i$ and $\my_j$ is perpendicular to $M_{ij}$ at $\my^*$.
\end{prop}

Since $\{\my_i\}_{i=1}^n$ are sampled based on a density function $\rho^{**}$ with a positive lower bound, when $n$ is large enough, with high probability, there are enough points in a small geodesic ball. Hence, we can assume that each Voronoi cell is small enough so that it is contained in a geodesic ball. We propose the following assumption.

\begin{assumption}\label{assumption on voronoi face}
For $\delta$ defined in Proposition \ref{perpendicular to Voronoi face}, let $B_{\frac{\delta}{2}}(\my_i)$ be an open geodesic ball centered at $y_i$ with radius $\frac{\delta}{2}$. We assume that when $n$ is large enough, we have $C_i \subset B_{\frac{\delta}{2}}(\my_i)$ for $i=1, \cdots , n$.
\end{assumption}

Suppose $\Gamma_{ij}$ is the Voronoi face between $\my_i$ and $\my_j$,  Assumption \ref{assumption on voronoi face} implies that $\my_j \in B_{\delta}(y_i)$. Based on Assumption \ref{assumption on voronoi face} and Proposition \ref{perpendicular to Voronoi face}, the interior of the Voronoi face $\Gamma_{ij}$ is an open subset of a submanifold $M_{ij}$. Hence, there is a well defined unit outward normal vector field $ \mathbf n$ on each $\partial C_i$  and we can apply the divergence Theorem on each Voronoi cell. 
}

\

{ 
Recall the {Fokker-Planck equation} on $\nn$ \eqref{FP_N1}.
We first recast \eqref{FP_N1} in the relative entropy form
\begin{equation}
\pt_t \rhon_t = \divn \left(\pi  \nabla_\nn \left(\frac{\rhon_t}{\pi} \right) \right).
\end{equation}
We drop the dependence $t, \nn$ in the short hand notation $\rho=\rho_t^\nn$. We integrate this on $C_i$ and use the divergence theorem on cell $C_i$ to obtain
\begin{equation}\label{tt310}
\frac{\ud }{\ud t} \int_{C_i} \rho \hs^d(C_i) = \sum_{j\in \text{VF}(i) } \int_{\Gamma_{ij}} \pi \mathbf n \cdot \nabla_\nn\left(\frac{\rho}{\pi} \right)  \hs^{d-1}(\Gamma_{ij}),
\end{equation}
where $\mathbf{n}$ is the unit outward normal vector field on $\partial C_i$.

To design the numerical algorithm, first, let us clarify the probability on each cell.   Then the 
probability in $C_i$ can be approximated as
\begin{equation}\label{alg11}
\int_{C_i} \rho(\my)  \hs^d(C_i) \approx \rho(\my_i) (1+\diam (C_i)) |C_i|.
\end{equation}
Second, we use $\rho_i$ to approximate the exact solution $\rho(\my_i)$ on each cell $C_i$. Let $\pi_i$ be the approximated equilibrium density at $\my_i$ satisfying $\sum_{i=1}^n \pi_i|C_i|=1$. Notice $ \rho_\8(\my)\propto e^{-{U_{\nn}(\my)}}$, so  $\pi_i>0$ for all $i$.
Then the surface gradient in \eqref{tt310} can be approximated by
\begin{equation}\label{alg22}
 \sum_{j\in \text{VF}(i) } \int_{\Gamma_{ij}} \pi \mathbf n \cdot \nabla_\nn\left(\frac{\rho}{\pi} \right)  \hs^{d-1}(\Gamma_{ij}) \approx \frac12 \sum_{j\in \text{VF}(i)} \frac{\pi_i+ \pi_j}{|y_i-y_j|} |\Gamma_{ij}|\left( \frac{\rho_j}{\pi_j}- \frac{\rho_i}{\pi_i} \right).
\end{equation}

Therefore, combining \eqref{alg11} and \eqref{alg22}, we give the following finite volume scheme. 
For $i=1, \cdots, n$,
\begin{equation}\label{mp}
\frac{\ud}{\ud t}\rho_i |C_i|= \frac12 \sum_{j\in \text{VF}(i)} \frac{\pi_i+ \pi_j}{|y_i-y_j|} |\Gamma_{ij}|\left( \frac{\rho_j}{\pi_j}- \frac{\rho_i}{\pi_i} \right).
\end{equation}

Let $\chi_{\scriptscriptstyle{C_i}}$ be the characteristic function such that $\chi_{\scriptscriptstyle{C_i}}=1$ for $\my\in C_i$ and $0$ otherwise.
Then
$$\rho^{\text{approx}}(\my):=\sum_{i=1}^n \rho_i \chi_{\scriptscriptstyle{C_i}}(\my)$$
 is the piecewise constant probability distribution on $\nn$ provided $\sum_{i=1}^n \rho_i|C_i|=1$ and $\rho_i\geq 0$.  We will prove later in the convergence analysis Theorem \ref{mainthm_conv} that the exact solution $\rho$ can be approximated by the numerical piecewise constant probability distribution constructed from $\rho_i, i=1, \cdots, n$.
}

We will first formulate finite volume  scheme \eqref{mp} as the forward equation for a Markov process with basic properties such as ergodicity in Section \ref{sec_mp}. We then show the truncation error analysis and stability analysis and thus convergence of the scheme \eqref{mp} later in Section \ref{sec_5.2}.

\subsection{Associated Markov process, detailed balance and ergodicity}\label{sec_mp}
We will first formulate finite volume  scheme \eqref{mp} as the forward equation for a Markov process and then  in Proposition \ref{decay}, we study the generator of the Markov process, which leads to ergodicity of  $\frac{\rho_i(t)}{\pi_i}$. Roughly speaking, the forward equation leads to the conservation law while the backward equation leads to maximum norm estimates for $\frac{\rho_i}{\pi_i}$.

\begin{lem}\label{lem_Markov}
{ Let $\pi_i>0$ for all $i=1, \cdots, n$.} The finite volume  scheme \eqref{mp} is the forward equation for a Markov Process with transition probability $P_{ji}$ (from state $j$ to $i$) and jump rate $\lmd_j$
{\begin{equation}\label{mp1}
\frac{\ud}{\ud t}\rho_i |C_i| = \Big( \sum_{j\in \text{VF}(i)} \lmd_j P_{ji} \rho_j |C_j| \Big)- \lmd_i \rho_i |C_i|,
\end{equation}}
where 
{\begin{equation}\label{def59}
\begin{aligned}
\lmd_i = \sum_{j\neq i} Q_{ij}=: \frac1{2|C_i|\pi_i}\sum_{j\in \text{VF}(i)} \frac{\pi_i+ \pi_j}{|y_i-y_j|}|\Gamma_{ij}|~{ >0}, \quad i=1, 2, \cdots, n; \\
 P_{ij}:=\left\{ \begin{array}{cc}
 \displaystyle \frac{Q_{ij}}{\lmd_i}=\frac{1}{\lmd_i}\frac{\pi_i+ \pi_j}{2\pi_i |C_i|}\frac{|\Gamma_{ij}|}{|y_i-y_j|}, & j\in \text{VF}(i);\\
 ~
 \\
  \displaystyle 0, & j\notin \text{VF}(i).
 \end{array}
 \right.
 \end{aligned}
\end{equation}
}
\begin{enumerate}[(i)]
\item
$P$ is transition probability matrix satisfying $\sum_j P_{ij}=1$ and the detailed balance property
\begin{equation}\label{db}
 P_{ji} \lmd_j\pi_j |C_j| =  P_{ij} \lmd_i\pi_i |C_i| = \frac{\pi_i+ \pi_j}{2}\frac{|\Gamma_{ij}|}{|y_i-y_j|} \quad \forall i,\, j.
\end{equation}
\item With $\{w_i\}_{i=1}^n:=\{\rho_i |C_i|\}_{i=1}^n$, we recast the forward equation  \eqref{mp1} as
\begin{equation}
\frac{\ud}{\ud t} w = Q^* w,
\end{equation}
{ where $Q^*$ is the transpose of $Q$-matrix defined as}
\begin{equation}\label{AA}
Q=(a_{ij})_{n\times n}, \quad a_{ij}:= \left\{\begin{array}{cc}
-\lmd_i , \quad & j=i;\\
\lmd_i  P_{ij}, \quad & j\neq i.
\end{array}\right.
\end{equation}
\item $\sum_{i=j}^n a_{ij}=0$, which gives the conservation law for $\sum_i w_i$
\begin{equation}
\frac{\ud}{\ud t} \sum_{i=1}^n w_i =\sum_{i=1}^n \sum_{j=1}^n a_{ji}w_j= 0;
\end{equation}
\item We have the dissipation relation { for $\chi^2$-divergence}
 \begin{align}\label{dissi}
\frac{\ud}{\ud t} \sum_i \frac{\rho^2_i}{\pi_i} |C_i|  
=-\sum_{i,j} \frac{\pi_i+ \pi_j}{2}\frac{|\Gamma_{ij}|}{|y_i-y_j|} \left(\frac{\rho_j}{\pi_j} -\frac{\rho_i}{\pi_i} \right)^2. 
\end{align}
\end{enumerate}
\end{lem}
\begin{proof}
First,  one can rewrite \eqref{mp} as \eqref{mp1} with $\lmd_i = \frac1{2|C_i|\pi_i}\sum_{j\in \text{VF}(i)} \frac{\pi_i+ \pi_j}{|y_i-y_j|}|\Gamma_{ij}|$ and $P_{ji}\lmd_j=\frac{\pi_i+ \pi_j}{2|y_i-y_j|}\frac{|\Gamma_{ij}|}{\pi_j |C_j|}$. Then since $\frac{\pi_i+ \pi_j}{|y_i-y_j|} |\Gamma_{ij}|$ is symmetric, we have
\begin{equation}
\frac{\ud}{\ud t} (\sum_{i=1}^n  |C_i|\rho_i) = \sum_{i,j}  \frac12 \frac{\pi_i+ \pi_j}{|y_i-y_j|} |\Gamma_{ij}|\left( \frac{\rho_j}{\pi_j}- \frac{\rho_i}{\pi_i} \right) = 0.
\end{equation}

Second we can check
\begin{equation}\label{Lemma sum Pij}
\sum_{i} P_{ji}=\sum_{i\in \text{VF}(j)} P_{ji} = \frac{1}{\lmd_j} \sum_{i\in \text{VF}(j)} \frac{\pi_i+ \pi_j}{2|y_i-y_j|} \frac{|\Gamma_{ij}|}{\pi_j |C_j|}=1.
\end{equation}

Third the detailed balance property comes from the symmetric property of $\frac{\pi_i+ \pi_j}{|y_i-y_j|} |\Gamma_{ij}|$ and
\begin{equation}
\lmd_j P_{ji}\pi_j |C_j| =\frac{\pi_i+ \pi_j}{2|y_i-y_j|} |\Gamma_{ij}|= \lmd_i P_{ij} \pi_i |C_i|.
\end{equation}

Next, the conservation law follows directly from  $\sum_{i=1}^na_{ji}=0$ {by \eqref{AA} and \eqref{Lemma sum Pij}}.
{ Denote the   diagonal rate matrix as ${R}=\text{diag}( {\lmd}_j)$, then we obtain  $Q$-matrix $ {Q} =  {R}( {P}-I)$.}

Finally, by detailed balance property \eqref{db} and $\sum_j P_{ij}=1$, we recast \eqref{mp1} as
 \begin{equation}\label{323}
\begin{aligned}
\frac{\ud}{\ud t}\rho_i |C_i| &=\sum_{j\in \text{VF}(i)} \lmd_i P_{ij} \pi_i |C_i| \frac{\rho_j}{\pi_j} - \lmd_i \pi_i |C_i| \frac{\rho_i}{\pi_i}\\
& =  \sum_{j\in \text{VF}(i)} \lmd_i \pi_i |C_i|  P_{ij} \left(\frac{\rho_j}{\pi_j} -\frac{\rho_i}{\pi_i} \right) = \sum_{j\in \text{VF}(i)} \frac{\pi_i+ \pi_j}{2|y_i-y_j|} |\Gamma_{ij}| \left(\frac{\rho_j}{\pi_j} -\frac{\rho_i}{\pi_i} \right)
\end{aligned}
\end{equation}
Multiplying this by {$\frac{\rho_i}{\pi_i}$} and  show that
\begin{align}
\frac{\ud}{\ud t} \sum_i \frac{\rho^2_i}{\pi_i} |C_i|  
=-\sum_{i,j} \frac{\pi_i+ \pi_j}{2}\frac{|\Gamma_{ij}|}{|y_i-y_j|} \left(\frac{\rho_j}{\pi_j} -\frac{\rho_i}{\pi_i} \right)^2. 
\end{align}
This gives the dissipation relation \eqref{dissi}.

\end{proof}

\begin{prop}\label{decay}
{ Let $\pi_i>0$ for all $i=1, \cdots, n$.}  { Let $Q:=(a_{ij})_{n\times n}$ be the $Q$-matrix defined in \eqref{AA}.}
Then $\{u_i\}_{i=1}^n:=\{\frac{\rho_i}{\pi_i}\}_{i=1}^n$ is the solution to the backward equation
\begin{equation}
\frac{\ud}{\ud t} u = Q u.
\end{equation}
Moreover, let $\lmd_i$ be the jump rate defined in \eqref{def59}. We conclude $0$ is the simple, principle eigenvalue of $Q$ with the ground state $\{1,1,\cdots, 1\} $. We thus have the exponential decay of $\rho_i(t)$ with respect to time $t$,
\begin{equation}
\max_i \frac{|\rho_i(t)-\pi_i|}{\pi_i}\leq c e^{-(\lambda_1-|\lambda_2|)t},
\end{equation}
{ where $\lambda_1>\max_i \lmd_i$ is the principle eigenvalue of $\lambda I + Q$, $|\lambda_2|<\lambda_1$ is the second largest eigenvalue of $\lambda_1I +Q$, and $\lambda_1-|\lambda_2|>0$ is the spectral gap  of $\lambda_1 I +Q$.}
\end{prop}
\begin{proof}
Recall \eqref{323}. We recast the forward equation \eqref{mp1} as
\begin{equation}
\begin{aligned}
\frac{\ud}{\ud t}\rho_i |C_i| =\sum_{j\in \text{VF}(i)} \lmd_i P_{ij} \pi_i |C_i| \frac{\rho_j}{\pi_j} - \lmd_i \pi_i |C_i| \frac{\rho_i}{\pi_i},
\end{aligned}
\end{equation}
which gives
\begin{equation}\label{back}
\frac{\ud}{\ud t} \frac{\rho_i}{\pi_i}=  \sum_{j\in \text{VF}(i)}  \lmd_i P_{ij}  \frac{\rho_j}{\pi_j} - \lmd_i \frac{\rho_i}{\pi_i}.
\end{equation}

Next, we show
$\{u_i\}_{i=1}^n:=\{\frac{\rho_i}{\pi_i}\}_{i=1}^n$ is the solution to this backward equation.  Recast \eqref{back} as 
\begin{equation}
\frac{\ud}{\ud t} u_i = \sum_{j=1}^n a_{ij} u_j.
\end{equation}
Here $Q=\{a_{ij}\}$ is the generator of the backward equation
\begin{equation}
\frac{\ud}{\ud t}u=Qu.
\end{equation}
One can check $\sum_{j=1}^n a_{ij}=0,$ $a_{ij}\geq 0$ for $j\neq i$ and $a_{ii}<0.$

Moreover, due to the detailed balance property \eqref{db}, we know $Q$ is self-adjoint in the weighted $l^2$-space
\begin{equation}
\la u, Qv \ra_{\pi|C|}:= \sum_{i,j} u_i a_{ij} v_j \pi_i |C_i| = \sum_{i,j} a_{ji} u_i v_j \pi_j |C_j| =: \la Qu, v \ra_{\pi |C|} .
\end{equation}
Thus we know the eigenvalues of $Q$ are real.
For the matrix $\lambda_1 I + Q$ with $\lambda_1>\max_i \lmd_i$, we know each element is non-negative and {$\lambda_1+ \sum_{j=1}^n a_{ij}=\lambda_1>0$.} {Since the Voronoi tessellation $\nn= \cup_{i=1}^n  C_i$, when $\nn$ is strongly connected, $\lambda_1 I +Q$ is irreducible.} 
{By} the Perron-Frobenius theorem for $\lambda_1 I + Q$, we know the Perron–Frobenius eigenvalue (i.e. the principal eigenvalue) of $\lambda_1 I + Q$ is 
 $\lambda_1$ and $\lambda_1>0$ is {a} simple eigenvalue with the ground state $u^*:=(1,1,\cdots, 1)$ and other  eigenvalues  $\lambda_i$ of $\lambda_1 I +Q$  satisfy $|\lambda_i|<\lambda_1$. {Therefore, we have
\begin{equation}\label{tmW}
\lambda_1(u-u^*) + \frac{\ud}{\ud t}(u-u^*) = (\lambda_1 I + Q)(u-u^*).
\end{equation} 
Let $\|\cdot\|$ be the weighted $l^2$ norm defined as
$\|v\|^2:=\sum_i v_i^2 \pi_i |C_i|.$
Taking weighted inner product of \eqref{tmW} with $u-u^*$ and from $\la u-u^*, u^* \ra_{\pi |C|}=0$, we have
\begin{equation}
\lambda_1 \|u-u^*\|^2 + \frac12\frac{\ud}{\ud t}(\|u-u^*\|^2) = \la (\lambda_1 I +Q)(u-u^*), ~u-u^* \ra_{\pi|C|} \leq |\lambda_2| \|u-u^*\|^2,
\end{equation}
{ where we used $|\lambda_2|<\lambda_1$ is the second largest eigenvalue of $\lambda_1I+Q$.}
This gives
\begin{equation}
\frac{\ud}{\ud t} \|u-u^*\| \leq (|\lambda_2|-\lambda_1)\|u-u^*\|.
\end{equation}
Therefore we obtain
  the exponential decay of $u$ to its ground state $u^*$
 \begin{equation}
 \|u-u^*\| \leq c e^{(|\lambda_2|-\lambda_1)t} .
 \end{equation}
Here $\lambda_1-|\lambda_2|>0$ is the spectral gap  of $\lambda_1 I +Q$.}
\end{proof}
We refer to \cite{li2018large} for the ergodicity of  finite volume  schemes in {an} unbounded space. { We refer to \cite{Maas_2011, chow2012fokker, Mielke_Renger_Peletier_2014, Erbar_2014, esposito2019nonlocal, yg20} for  more discussions on the corresponding generalized gradient flow of the relative entropy with graph Wasserstein distance on discrete space and Benamou-Brenier formula.  See also \cite{lai2018point, yuan2020continuum, gao2021random} for some related   data-driven algorithms when solving equations on a network graph and for irreversible processes.}

\subsection{Truncation error estimate, stability and convergence of the finite volume  scheme \eqref{mp}}\label{sec_5.2}
 In this section we prove  the stability of \eqref{mp} in Lemma \ref{stable}. Then we obtain the convergence of the solution to finite volume  scheme \eqref{mp} to the solution of Fokker-Planck {equation} \eqref{FP_N1} in Theorem \ref{mainthm_conv}.

First, we have the following stability property, which corresponds to the Markov chain version of  the Crandall-Tartar lemma for monotone schemes. This lemma is also known as the total variation diminishing for two density solutions.

\begin{lem}\label{stable}
Any two solutions $\rho_i$ and $\tilde{\rho}_i$ to  finite volume  scheme \eqref{mp1} have the following stability properties
\begin{equation}
\begin{aligned}
\frac{\ud}{\ud t} \sum_{i=1}^n  \left|\rho_i-\tilde{\rho}_i|\right|C_i|  \leq 0;\\
\frac{\ud}{\ud t} \sum_{i=1}^n  \left|\dot{\rho}_i \right| |C_i|\leq 0,
\end{aligned}
\end{equation}
{ where $\dot{\rho}_i$ is the time derivative of $\rho_i(t)$.}
\end{lem}
\begin{proof}
First assume $\rho_i$ and $\tilde{\rho}_i$ are two solutions to  finite volume  scheme \eqref{mp1}. We have
\begin{equation}
\frac{\ud}{\ud t} (\rho_i |C_i|- \tilde{\rho}_i|C_i|) = \sum_{j\in \text{VF}(i)} P_{ji} \lmd_j |C_j|(\rho_j-\tilde{\rho}_j) - \lmd_i|C_i|(\rho_i- \tilde{\rho}_i).
\end{equation}
{ Notice that for any function $u$,   multiplying $u$ by its sign gives an absolute value $|u|$. From \cite[Lemma 7.6]{gilbarg2015elliptic}, the derivative of $|u|$ equals the derivative of $u$ multiplied by the sign of $u$, i.e., $D|u|= \text{sgn} (u ) Du.$}Multiply $\text{sgn}(\rho_i-\tilde{\rho}_i)$ to both sides and then take summation with respect to $i$
\begin{equation}
\frac{\ud}{\ud t} \sum_{i=1}^n  |C_i||\rho_i-\tilde{\rho}_i| \leq \sum_{i,j} P_{ji} \lmd_j |C_j| |\rho_j-\tilde{\rho}_j|-\sum_{i=1}^n \lmd_i |C_i| |\rho_i-\tilde{\rho}_i|=0,
\end{equation}
where we used $\sum_{i\in \text{VF}(j)} P_{ji}=1$.
Second, take time derivative in \eqref{mp1}, then we have
\begin{equation}
\frac{\ud^2}{\ud t^2} \rho_i |C_i| = \sum_{j\in \text{VF}(i)} P_{ji} \lmd_j |C_j|\dot{\rho}_j - \lmd_i|C_i|\dot{\rho}_i.
\end{equation}
Then similarly we can multiply $\text{sgn}(\dot{\rho}_i)$ to both sides and obtain
\begin{equation}
\frac{\ud}{\ud t} \sum_{i=1}^n  |C_i||\dot{\rho}_i| \leq \sum_{i,j} P_{ji} \lmd_j |C_j| |\dot{\rho}_j|-\sum_{i=1}^n \lmd_i |C_i| |\dot{\rho}_i|=0,
\end{equation}
where we used $\sum_{i\in \text{VF}(j)} P_{ji}=1$.
\end{proof}

We conclude this section by the following  convergence theorem in the weighted $L^2$ sense. { Although the proof of the convergence theorem  is rather  standard, cf. \cite{eymard2000finite}, however the error estimation on manifold requires some careful treatments for the symmetric cancellation and some approximation lemmas for the Voronoi tessellation.}

\begin{thm}[Convergence]\label{mainthm_conv}
Suppose $\rho(\my,t)$, $t\in[0,T]$ is a smooth solution to Fokker-Planck equation \eqref{FP_N1} on manifold $\nn\subset \bR^\ell$ with initial density $\rho^0(\my)$.  Let $\nn=\cup_{i=1}^n  C_i$ be the Voronoi tessellation of manifold $\nn$ based on $\{\my_i\}_{i=1}^n$. Let 
\begin{align}
h=\max\big( \max_{i=1, \ldots, n}(\diam(C_i)), \max_{i=1, \ldots, n}(\max_{j\in \text{VF}(i)} d_{\nn}(\my_i,\my_j))\big)
\end{align}
Let $\{\rho_i\}_{i=1}^n$ be the solution to the finite volume  scheme \eqref{mp}  with initial data $\{\rho_i^0\}_{i=1}^n$ and  $e_i:= \rho(\my_i) - \rho_i$. Under Assumption \ref{assumption on voronoi face},   we have the following error estimate
{\begin{equation}
\max_{t\in[0,T]} \sum_i e_i(t)^2 \frac{|C_i|}{\pi_i} \leq \big (\sum_i e_i(0)^2 \frac{|C_i|}{\pi_i} +O( h^2 (n h \max_{i} |\partial C_i|+1 )) \big)  e^T,
\end{equation}}
where the constant in $O( h^2 (n h \max_{i} |\partial C_i|+1 ))$ depends on $\vol(\nn)$, the minimum of  $\pi$,  the $C^1$ norm of $\pi$, the $L^\infty$ norm of $\partial_t \nabla_{\nn} \rho$ and  the $C^2$ norm of $\frac{\rho}{\pi}$.
\end{thm}

\begin{proof}
Let  $\rho_i^e:=\frac{1}{|C_i|}\int_{C_i} \rho \ud y$ be the cell average. Plug the exact solution into the numerical scheme
\begin{equation}\label{mp-exact}
\begin{aligned}
\pt_t (\rho_i^e |C_i|) &=  \sum_{j\in \text{VF}(i)} \sch \left(  \frac{\rho(\my_j) }{\pi_j}- \frac{\rho(\my_i)}{\pi_i}\right) + \sum_{j\in \text{VF}(i)}\eps_{ij},\\
&\eps_{ij}:= \int_{\Gamma_{ij}}  \pi \mathbf{n}_{ij} \cdot \nabla_\nn \frac{\rho}{\pi} \ud \hs^{d-1}  ~ - ~ \sch \left(  \frac{\rho(\my_j) }{\pi_j}- \frac{\rho(\my_i) }{\pi_i}\right),
\end{aligned}
\end{equation}
where $\mathbf{n}_{ij}$ is the restriction of the unit outward normal vector field on $\Gamma_{ij}$. Exchanging $i,j$ above, we can see that $\eps_{ij}$ is anti-symmetric. 

Subtracting the numerical scheme \eqref{mp} from \eqref{mp-exact}, 
we have
\begin{equation}
\frac{\ud}{\ud t} (e_i|C_i|) =  \sum_{j\in \text{VF}(i)} \sch \left(  \frac{e_j}{\pi_j}- \frac{e_i}{\pi_i}\right) + \sum_{j\in \text{VF}(i)} \eps_{ij}+  \pt_t ( (\rho(\my_i) -\rho_i^e) |C_i|).
\end{equation}
Similar to the dissipation relation \eqref{dissi}, multiplying $\frac{e_i}{\pi_i}$ shows that
\begin{equation}
 \frac{\ud}{\ud t} \sum_i e_i^2\frac{|C_i|}{\pi_i} = -\sum_i \sum_{j\in \text{VF}(i)}  \sch    \left(  \frac{e_j}{\pi_j}- \frac{e_i}{\pi_i}\right)^2 \,+ \sum_i \sum_{j\in \text{VF}(i)} 2\eps_{ij} \frac{e_i}{\pi_i} +\sum_i 2 \pt_t(\rho(\my_i) -\rho_i^e) e_i  \frac{ |C_i|}{\pi_i}.
\end{equation}
 Since $\eps_{ij}$ is anti-symmetric,
\begin{equation}
\begin{aligned}
 \frac{\ud}{\ud t} \sum_i e_i^2\frac{|C_i|}{\pi_i} &= -\sum_i \sum_{j\in \text{VF}(i)}  \sch    \left(  \frac{e_j}{\pi_j}- \frac{e_i}{\pi_i}\right)^2\\
 &  \quad + \sum_i \sum_{j\in \text{VF}(i)} \eps_{ij} \left(\frac{e_i}{\pi_i} - \frac{e_j}{\pi_j} \right)+ \sum_i 2 \pt_t(\rho(\my_i) -\rho_i^e) e_i  \frac{ |C_i|}{\pi_i}.
 \end{aligned}
\end{equation}
Applying  Young's inequality to the last two terms, we have
\begin{equation}
\begin{aligned}
&\sum_i \sum_{j\in \text{VF}(i)}  \eps_{ij} \left(\frac{e_i}{\pi_i} - \frac{e_j}{\pi_j} \right)\\
&\leq \frac12\sum_{i}  \sum_{j\in \text{VF}(i)} \sch  \left(  \frac{e_j}{\pi_j}- \frac{e_i}{\pi_i}\right)^2 + \frac12 \sum_{i} \sum_{j\in \text{VF}(i)} \frac{\eps^2_{ij}}{\sch};
\end{aligned}
\end{equation}
\begin{equation}
\sum_i 2 \pt_t(\rho(\my_i) -\rho_i^e) e_i  \frac{ |C_i|}{\pi_i}\leq \sum_i  [\pt_t(\rho(\my_i) -\rho_i^e)]^2  \frac{ |C_i|}{\pi_i}  + \sum_i  e_i^2  \frac{ |C_i|}{\pi_i}.
\end{equation}
Thus we have
\begin{align}
 \frac{\ud}{\ud t} \sum_i e_i^2\frac{|C_i|}{\pi_i} \leq & -\frac12\sum_i \sum_{j\in \text{VF}(i)}  \sch    \left(  \frac{e_j}{\pi_j}- \frac{e_i}{\pi_i}\right)^2 + \frac12 \sum_{i} \sum_{j\in \text{VF}(i)} \frac{\eps^2_{ij}}{\sch} \\
 & +\sum_i  [\pt_t(\rho(\my_i) -\rho_i^e)]^2  \frac{ |C_i|}{\pi_i}  + \sum_i  e_i^2  \frac{ |C_i|}{\pi_i} \nonumber \\
\leq & \sum_{i} \sum_{j\in \text{VF}(i)} \frac{\eps^2_{ij}}{\sch} +\sum_i  [\pt_t(\rho(\my_i) -\rho_i^e)]^2  \frac{ |C_i|}{\pi_i}  + \sum_i  e_i^2  \frac{ |C_i|}{\pi_i}.
 \nonumber
\end{align}

Next, we bound the term $\sum_{i}  \sum_{j\in \text{VF}(i)} \frac{\eps^2_{ij}}{\sch}$.

Let $G_{ij}$ be the bisector between $\my_i$ and $\my_j$.  Suppose $\my^*$ is the intersection point of the minimizing geodesic from $\my_i$ to $\my_j$ and $G_{ij}$. We have $d_{\nn}(\my^*, \my_i)=d_{\nn}(\my^*, \my_j)$.  Suppose $T$ is the unit tangent vector of the minimizing geodesic at $\my^*$.  From the Taylor expansion of $\frac{\rho}{\pi}$ along the geodesic, we have
\begin{align}
&\frac{\rho}{\pi}(\my_j)-\frac{\rho}{\pi}(\my^*)=T \cdot \nabla_{\nn}\frac{\rho}{\pi}(\my^*)d_{\nn}(\my^*, \my_j)+O(d^2_{\nn}(\my^*, \my_j)), \\
&\frac{\rho}{\pi}(\my^*)-\frac{\rho}{\pi}(\my_i)=T \cdot \nabla_{\nn}\frac{\rho}{\pi}(\my^*)d_{\nn}(\my^*, \my_i)+O(d^2_{\nn}(\my^*, \my_i)). 
\end{align}
By Assumption \ref{assumption on voronoi face} and  Proposition \ref{perpendicular to Voronoi face}, $\mathbf{n}_{ij}$ can be extended to a unit normal vector field on the $d-1$ dimensional submanifold $M_{ij} \subset G_{ij}$. We also call the extension to be $\mathbf{n}_{ij}$. We have $T=\mathbf{n}_{ij}(\my^*)$. Therefore, if we add the above two equations , we have
\begin{align}
&\frac{\rho}{\pi}(\my_j)-\frac{\rho}{\pi}(\my_i)=\mathbf{n}_{ij} \cdot \nabla_{\nn}\frac{\rho}{\pi}(\my^*) d_{\nn}(\my_i, \my_j)+O(d^2_{\nn}(\my_i, \my_j)).
\end{align}
Hence,
\begin{align}
\mathbf{n}_{ij} \cdot \nabla_{\nn}\frac{\rho}{\pi}(\my^*)=\frac{\frac{\rho(\my_j)}{\pi_j}-\frac{\rho(\my_i)}{\pi_i}}{ d_{\nn}(\my_i, \my_j)}+O(d_{\nn}(\my_i, \my_j))=\frac{\frac{\rho(\my_j)}{\pi_j}-\frac{\rho(\my_i)}{\pi_i}}{|\my_i-\my_j|}+O(d_{\nn}(\my_i, \my_j)),
\end{align}
where we apply Lemma \ref{geodesic vs euclidean} in the last step. Similarly,
\begin{align}
&\pi(\my_j)-\pi(\my^*)=O(d_{\nn}(\my^*, \my_j)), \\
&\pi(\my^*)-\pi(\my_i)=O(d_{\nn}(\my^*, \my_i)). 
\end{align}
Hence,
\begin{align}
\pi(\my^*)=\frac{\pi_i+\pi_j}{2}+O(d_{\nn}(\my_i, \my_j)).
\end{align}
Therefore, 
\begin{align}\label{bounds on nabla rho pi}
\pi(\my^*) \mathbf{n}_{ij} \cdot \nabla_{\nn}\frac{\rho}{\pi}(\my^*)=\frac{\pi_i+\pi_j}{2} \frac{\frac{\rho(\my_j)}{\pi_j}-\frac{\rho(\my_i)}{\pi_i}}{|\my_i-\my_j|} +O(d_{\nn}(\my_i, \my_j)).
\end{align}
For any $\my$ on $\Gamma_{ij}$, 
\begin{align}
\pi(\my) \mathbf{n}_{ij} \cdot \nabla_{\nn}\frac{\rho}{\pi}(\my)= & \pi(\my^*) \mathbf{n}\cdot \nabla_{\nn}\frac{\rho}{\pi}(\my^*)+O(d_{\nn}(\my,\my^*)) \\
= &  \pi(\my^*) \mathbf{n}\cdot \nabla_{\nn}\frac{\rho}{\pi}(\my^*)+O(d_{\nn}(\my_i,\my)+d_{\nn}(\my_i,\my_j)) \nonumber \\ 
=& \pi(\my^*) \mathbf{n}\cdot \nabla_{\nn}\frac{\rho}{\pi}(\my^*)+O(\diam(C_i)+d_{\nn}(\my_i,\my_j)),
\end{align}
where $\diam(C_i)$ is the diameter of $C_i$ measured with respect to the distance in $\nn$. Thus,
\begin{align}
\pi(\my) \mathbf{n}_{ij} \cdot \nabla_{\nn}\frac{\rho}{\pi}(\my)
=\frac{\pi_i+\pi_j}{2} \frac{\frac{\rho(\my_j)}{\pi_j}-\frac{\rho(\my_i)}{\pi_i}}{|\my_i-\my_j|} +O(\diam(C_i)+d_{\nn}(\my_i,\my_j)).
\end{align}

We conclude that 
\begin{align}\label{bound on epsilon 1}
{\eps_{ij}}=O((\diam(C_i)+d_{\nn}(\my_i,\my_j)) |\Gamma_{ij}|). 
\end{align}
Therefore,
\begin{align}
\frac{\eps^2_{ij}}{\sch}=O(d_{\nn}(\my_i,\my_j)(\diam(C_i)+d_{\nn}(\my_i,\my_j))^2 |\Gamma_{ij}|) .
\end{align}
If we sum up all $j \in \text{VF}(i)$,
\begin{align}
\sum_{j \in \text{VF}(i)} \frac{\eps^2_{ij}}{\sch}=\max_{j \in \text{VF}(i)} d_{\nn}(\my_i,\my_j)(\diam(C_i)+d_{\nn}(\my_i,\my_j))^2 O( |\partial C_i|). 
\end{align}
Hence, 
\begin{align}
\sum_{i} \sum_{j \in \text{VF}(i)} \frac{\eps^2_{ij}}{\sch}=O ( n h^3  \max_{i} |\partial C_i|),
\end{align}
where the constant depends on the minimum of  $\pi$,  the $C^1$ norm of $\pi$ and  the $C^2$ norm of $\frac{\rho}{\pi}$.

Next, we bound $\sum_i  [\pt_t(\rho(\my_i) -\rho_i^e)]^2  \frac{ |C_i|}{\pi_i}$
\begin{align}
\pt_t(\rho(\my_i) -\rho_i^e)=O(\diam(C_i))=O(h), 
\end{align}
where the constant depends on the $L^\infty$ norm of $\partial_t \nabla_{\nn} \rho$. Since $\sum_i |C_i|= \vol(\nn)$, 
\begin{align}\label{bound on epsilon 2}
\sum_i  [\pt_t(\rho(\my_i) -\rho_i^e)]^2  \frac{ |C_i|}{\pi_i}=O(h^2),
\end{align}
where the constant depends on the $L^\infty$ norm of $\partial_t \nabla_{\nn} \rho$, $ \vol(\nn)$ and minimum of $\pi$.
Hence , 
\begin{align}
 \frac{\ud}{\ud t} \sum_i e_i^2\frac{|C_i|}{\pi_i} \leq O( h^2 (n h \max_{i} |\partial C_i|+1 )) + \sum_i  e_i^2  \frac{ |C_i|}{\pi_i}. 
\end{align}
{In conclusion,
\begin{equation}
\max_{t\in[0,T]} \sum_i e_i(t)^2 \frac{|C_i|}{\pi_i} \leq \big (\sum_i e_i(0)^2 \frac{|C_i|}{\pi_i} +O( h^2 (n h \max_{i} |\partial C_i|+1 )) \big)  e^T .
\end{equation}}
\end{proof}

\subsection{Approximation of Voronoi cells on manifold}\label{sec5.3}
Recall that $\{\my_i\}_{i=1}^n$ are samples on the smooth closed submanifold $\nn$ in $\mathbb{R}^\ell$ based on the density function $\rho^{**}$. In this section, we introduce an algorithm to approximate the volumes of the Voronoi cells and the areas of the Voronoi faces constructed from $\{\my_i\}_{i=1}^n$. 

First, we need the following definition.

\begin{defn}\label{definition of the map iota}
For any $0<r<1$ and $\my_k \in \{\my_i\}_{i=1}^n$, suppose $B^{\mathbb{R}^\ell}_{\sqrt{r}}(\my_k) \cap \{\my_i\}_{i=1}^n=\{\my_{k,1}, \cdots, \my_{k,\bar{N}_k}\}$. We define the discrete local covariance matrix at $\my_k$, 
\begin{align}\label{def_c38}
C_{n,r}(\my_k):=\frac{1}{n}\sum_{i=1}^{\bar{N}_k}(\my_{k,i}-\my_{k})(\my_{k,i}-\my_{k})^\top \in \mathbb{R}^{\ell \times \ell}.
\end{align}
Suppose $\{\beta_{n,r,1}, \cdots, \beta_{n,r,d}\}$  are the first $d$ orthonormal eigenvectors corresponding to $C_{n,r}(\my_k)$'s largest $d$ eigenvalues. Define a map $\iota_k(u): \mathbb{R}^\ell \rightarrow \mathbb{R}^d$ as 
\begin{align}\label{def_c39}
\iota_k(u):=
(u^\top\beta_{n,r,1}, \cdots, u^\top\beta_{n,r,d}).
\end{align}
For any $\my \in \mathbb{R}^\ell$, define $\tilde{\iota}_k(\my)=\iota_k(\my-\my_k)$. 
\end{defn}

Based on the above definition, we propose the following algorithm to find the approximated volumes $|\tilde{C}_k|$ of the Voronoi cells $C_k$  and the approximated areas $|\tilde{\Gamma}_{k\ell}|$ of the Voronoi faces $\Gamma_{k\ell}$.

\vspace{0.4cm}
\begin{algorithm}[H] \label{Voronoi cell on manifold algorithm}
\SetAlgoLined
\Parameter{Algorithm inputs are the bandwidth $r$ and the threshold $s$}

Choose $0<r<1$. For each $\my_k \in \{\my_i\}_{i=1}^n$, find $$B^{\mathbb{R}^\ell}_{\sqrt{r}}(\my_k) \cap \{\my_i\}_{i=1}^n=:\{\my_{k,1}, \cdots, \my_{k,\bar{N}_k}\}, \quad B^{\mathbb{R}^\ell}_{r}(\my_k) \cap \{\my_i\}_{i=1}^n=:\{\my_{k,1}, \cdots, \my_{k, N_k}\}.$$
 
Construct the matrix $C_{n,r}(\my_k)$ as in \eqref{def_c38} by using the $\{\my_{k,1}, \cdots, \my_{k,\bar{N}_k}\}$.
Find the orthonormal eigenvectors corresponding to $C_{n,r}(\my_k)$'s largest $d$ eigenvalues. Denote them as $\{\beta_{n,r,1}, \cdots, \beta_{n,r,d}\}$.

Use $\{\beta_{n,r,1}, \cdots, \beta_{n,r,d}\}$ to construct $\tilde{\iota}_k$ as in \eqref{def_c39}.  Find $v_{k,i}=\tilde{\iota}_k(\my_{k,i})$, for $i=1,\cdots, N_k$. 

Find the Voronoi cell decomposition of $\{0, v_{k,1}, \cdots, v_{k, N_k}\}$ in $\mathbb{R}^d$. Denote the Voronoi cell containing $0$ to be $\tilde{C}_{k,0}$ and the  Voronoi cell containing $v_{k,i}$ to be $\tilde{C}_{k,i}$. Denote the face $\tilde{\Gammaf}_{k,i}=\tilde{C}_{k,0} \cup \tilde{C}_{k,i}$.

Find the approximation of $|C_k|$ as
\begin{equation}\label{vol_cell}
|\tilde{C}_k|:=|\tilde{C}_{k,0}|:=\hs^d(\tilde{C}_{k,0}).
\end{equation}

Find $|\tilde{\Gammaf}_{k,i}|=\hs^{d-1}(\tilde{\Gammaf}_{k,i})$. Define $\tilde{\Gamma} \in \mathbb{R}^{n \times n}$ such that 
\begin{align}\label{area_cell}
A_{k \ell}:= \frac{\tilde{A}_{k\ell}+\tilde{A}_{\ell k}}{2}, \quad \tilde{A} _{k\ell} = \left\{ 
\begin{array}{ll} 
|\tilde{\Gammaf}_{k,i}| & \mbox{if $\my_\ell=\my_{k,i} \in B^{\mathbb{R}^\ell}_{r}(\my_k)$ ;}\\ 
0 & \mbox{otherwise}. 
\end{array} \right. 
\end{align}

If  $A_{k \ell} \geq s$, then $|\tilde{\Gamma}_{k\ell}|=A_{k \ell}$. Otherwise $|\tilde{\Gamma}_{k\ell}|=s$.
Then $|\tilde{\Gamma}_{k\ell}|$  is an approximation of $|\Gamma_{k\ell}|$.
\caption{Approximation of the Voronoi cell}
\end{algorithm}
\vspace{0.3cm}
The idea of the above algorithm can be summarized as follows. For each $\my_k$, by using the points in a larger ball $B^{\mathbb{R}^\ell}_{\sqrt{r}}(\my_k)$, we construct the matrix $C_{n,r}(\my_k)$. Then, the first $d$ orthonormal eigenvectors will be an approximation of an orthonormal basis of $T_{\my_k}\nn$. Next, we project the points in a smaller ball $B^{\mathbb{R}^\ell}_{r}(\my_k)$ onto this tangent space approximation. Now the points around $\my_k$ are projected into a $d$ dimensional Euclidean space and $\my_k$ is projected to the origin. If we find  the Voronoi cell around the origin in the Euclidean space, then it gives the approximation of the Voronoi cell around $\my_k$ in $\nn$.  Obviously, the better estimation of the tangent space we have, there are smaller errors in the approximation of the volumes of the Voronoi cells and  the areas of the Voronoi faces. 

Next, we provide a justification of the above algorithm. When the geodesic distance between two points on $\nn$ is small, the next lemma relates the Euclidean distance and the geodesic distance between them. The proof can be found in Lemma B.3 in \cite{wu2018think}.

\begin{lem}\label{geodesic vs euclidean}
Suppose $\my, \my' \in \nn$ such that $d_\nn(y,y')$ is small enough. Then 
\begin{align}
\|\my'-\my\|_{\mathbb{R}^\ell}=d_\nn(\my,\my')(1+O(d^2_\nn(\my,\my'))), 
\end{align}
where the constant in $O(d^2_\nn(\my,\my'))$ depending on the second fundamental form of $\nn$ in $\mathbb{R}^\ell$ at $\my$. 
\end{lem}

{The above} lemma implies that if $r$ is small enough, then for all  $\my_k$ and any $\my \in B^{\mathbb{R}^\ell}_{r}(\my_k) \cap \nn$, there is a constant $D_1>1$ depending on the second fundamental form of $\nn$ in $\mathbb{R}^\ell$, such that 
\begin{align} \label{euclidean vs geodesic}
d_\nn(\my,\my_k) \leq D_1 \|\my_k-\my\|_{\mathbb{R}^\ell}.
\end{align}

  We further make the following assumption about the Voronoi cells and the distribution of $\{\my_i\}_{i=1}^n $ on $\nn$.
\begin{assumption}\label{assumption on voronoi cell}
For $n$ large enough, there exists $r$ { depending on $n$ such that $nr^d$ is bounded from above and has a positive lower bound for all $n$ and $\frac{nr^{\frac{d}{2}}}{\log n} \rightarrow \infty$ as $n \rightarrow \infty$.  Moreover, when $n$ is large enough, the following conditions about $r$ hold for any $\my_k$:}
\begin{enumerate}[(1)]
\item Suppose $B^{\mathbb{R}^\ell}_{r}(\my_k) \cap \{\my_i\}_{i=1}^n=\{\my_{k,1}, \cdots, \my_{k, N_k}\}$.  We have $C_k \subset B^{\mathbb{R}^\ell}_{r}(\my_k)$. Moreover, if $\Gamma_{kj}$ is  a Voronoi surface of $C_k$ between $\my_k$ and $\my_j$, then $\my_j \in B^{\mathbb{R}^\ell}_{r}(\my_k)$. Suppose $\my_j=\my_{k,m}$, then we introduce the notation $\Gamma_{k,m}=\Gamma_{kj}$.
\item  For any $i=1, \cdots N_k$, there is a constant $D_2<1$ such that $d_{\nn}(\my_{k,i}, \my_k) \geq D_2 r$. 
\end{enumerate}
\end{assumption}

{ Next, we intuitively explain the relation between Assumption \ref{assumption on voronoi cell} and Algorithm \ref{Voronoi cell on manifold algorithm}. Recall that $\{\my_i\}_{i=1}^n$ are sampled based on a density function $\rho^{**}$ with a positive lower bound and upper bound. In Algorithm \ref{Voronoi cell on manifold algorithm}, we use the points in a larger ball $B^{\mathbb{R}^\ell}_{\sqrt{r}}(\my_k)$ to approximate the tangent space $T_{\my_k}\nn$. Since $\rho^{**}$ has a positive lower bound, the condition $\frac{nr^{\frac{d}{2}}}{\log n} \rightarrow \infty$ as $n \rightarrow \infty$ implies that the number of points in $B^{\mathbb{R}^\ell}_{\sqrt{r}}(\my_k)$ goes to infinity as $n$ goes to infinity. Hence, we can have a good estimation of the tangent space. The condition that  $nr^d$ is bounded from above and has a positive lower bound for all $n$ implies $r \rightarrow 0$ as $n \rightarrow \infty$.  Since $\rho^{**}$ has an upper bound and a positive lower bound, it also implies that we will have enough but not too many points in the smaller ball $B^{\mathbb{R}^\ell}_{r}(\my_k)$. Hence, (1) and (2) become mild assumption with this relation between $r$ and $n$. In fact, since $r \rightarrow 0$ as $n \rightarrow \infty$, (1) says that the Voronoi cell is in a small ball $B^{\mathbb{R}^\ell}_{r}(\my_k)$. In (2), since there are not too many points in  $B^{\mathbb{R}^\ell}_{r}(\my_k)$, it is reasonable to assume the distance between the points in $B^{\mathbb{R}^\ell}_{r}(\my_k)$ and $\my_k$ has a lower bound $D_2 r$.  With (1) and (2), we can show that the approximation to Voronoi cell in the tangent space is accurate enough for our analysis. 

Consider the geodesic ball $B_{\frac{\delta}{2}}(y_k)$ in Assumption \ref{assumption on voronoi face}. By Lemma \ref{geodesic vs euclidean}, when $r$ is small enough, we have $B^{\mathbb{R}^\ell}_{r}(\my_k) \cap \nn \subset B_{\frac{\delta}{2}}(y_i)$. Since $r \rightarrow 0$ as $n \rightarrow \infty$,  we know that when $n$ is large  enough,  (1) in Assumption \ref{assumption on voronoi cell} implies Assumption \ref{assumption on voronoi face}. Hence, when $n$ is large  enough, Assumption \ref{assumption on voronoi cell} with Proposition \ref{perpendicular to Voronoi face} implies that the interior of each Voronoi face of $C_k$ is an open subset of a $d-1$ dimensional submanifold.}

The following lemma is a consequence of  (2) in Assumption \ref{assumption on voronoi cell}.

\begin{lem}
Under Assumption \ref{assumption on voronoi cell}, $d_{\nn}(\partial C_k, y_k) \geq \frac{1}{2}D_2 r$. There are constants $K_1$ and $K_2$ depending on $D_1$, $D_2$ and the Ricci curvature of $\nn$, such that 
\begin{align}\label{volume bounds on the cells}
K_1 r^d \leq |C_k| \leq K_2 r^d,
\end{align}
\end{lem}
\begin{proof}
Suppose $G_{k,i}$ is the bisector between $\my_k$ and $\my_{k,i}$. Then $d_{\nn}(\Gamma_{k,i}, y_k) \geq d_{\nn}(G_{k,i}, y_k) \geq  \frac{1}{2}D _2 r$. Hence, $d_{\nn}(\partial C_k, y_k) \geq \frac{1}{2}D _2 r$. Therefore, each $C_i$ contains a geodesic ball of radius $\frac{1}{2}D _2 r$ and is contained in the geodesic ball of radius $D _1 r$. By Lemma B.1 in \cite{wu2018think} when $r$ is small enough, the volume {of a} geodesic ball of radius $r$ can be bounded from below by $K'_1 r^d$ and from above by $K'_2 r^d$ where $K'_1$ and $K'_2$ depend on the Ricci curvature of $\nn$. The conclusion follows.
\end{proof}

In the next proposition, we show that $|\tilde{C}_k|$ is a good approximation of $|C_k|$. The proof of the proposition is in the Appendix.

\begin{prop}\label{approximation of the volume of a Voronoi cells}
Let $|\tilde{C}_k|$ be the approximated volume of $C_k$  in \eqref{vol_cell}. {If $n$ is large enough, for $r$ satisfying Assumption \ref{assumption on voronoi cell},} with probability greater than $1-\frac{1}{n^2}$, for all $\my_k$, we have $|\tilde{C}_k|=|\tilde{C}_{k,0}|=|C_k|(1+O(r))$.  
\end{prop}

Since we are approximating the tangent plane of the manifold $\nn$, the error between $|\Gamma_{k i}|$ and $|\tilde{\Gamma}_{k i}|$ will not be much smaller than $|\Gamma_{k i}|$ itself when $|\Gamma_{k i}|$ is too small. However, in the next proposition, we show that if $|\Gamma_{k i}|$ is large enough, then $|\tilde{\Gamma}_{k i }|$ is a good approximation of $|\Gamma_{k i}|$. The proof of the proposition is in the appendix. 

\begin{prop}\label{approximation of the area of a Voronoi face}
Let $|\tilde{\Gamma}_{k i}|$ be the approximated area of $\Gamma_{k i}$  in \eqref{area_cell}. {If $n$ is large enough, for $r$ satisfying Assumption \ref{assumption on voronoi cell},} let $s=a_1 r^d$ in the last step of Algorithm \ref{Voronoi cell on manifold algorithm} for some constant $a_1$,  with probability greater than $1-\frac{1}{n^2}$, for all $\my_k$, we have 
\begin{align}
|\Gamma_{k i}|=|\tilde{\Gamma}_{k i}|+O(r^d).
\end{align}
Hence, if $|\Gamma_{k i}| \geq a_2 r^{d-1}$ for some constant $a_2$,  then
\begin{align}
|\Gamma_{k i}|=|\tilde{\Gamma}_{k i}|(1+O(r))
\end{align}
\end{prop}

At last, if we use our approximation of the volumes of the Voronoi cells and  the areas of the Voronoi faces in \eqref{mp} we have the following implementable  finite volume  scheme based only on the collected dataset $\{\my_i\}\subset\nn$
\begin{equation}\label{mp_pron}
\frac{\ud}{\ud t}\tilde{\rho}_i |\tilde{ C}_i|= \frac12 \sum_{j\in \text{VF}(i)} \frac{\pi_i+ \pi_j}{|y_i-y_j|} |\tilde{\Gamma}_{ij}|\left( \frac{\tilde{\rho}_j}{\pi_j}- \frac{\tilde{\rho}_i}{\pi_i} \right).
\end{equation}
 Moreover, same as Lemma \ref{lem_Markov}, we know
the  finite volume  scheme \eqref{mp_pron} is the forward equation for a Markov Process with transition probability $\tilde{ P}_{ji}$ and jump rate $\tilde{\lmd}_i$
\begin{equation}\label{mp1_pron}
\frac{\ud}{\ud t} \tilde{\rho}_i | \tilde{ C}_i| = \sum_{j\in \text{VF}(i)} \tilde{\lmd}_j \tilde{ P}_{ji} \tilde{\rho}_j |\tilde{ C}_j| - \tilde{\lmd}_i \tilde{\rho}_i |\tilde{ C}_i|,
\end{equation}
where for $i=1,\cdots, n,\, j=1,\cdots, n$,
\begin{equation}\label{definition of delta tilde and P tilde}
\begin{aligned}
&\tilde{\lmd}_i := \frac1{2|\tilde{ C}_i|\pi_i}\sum_{j\in \text{VF}(i)} \frac{\pi_i+ \pi_j}{|y_i-y_j|}|\tilde{\Gamma}_{ij}|, \\
&  \tilde{ P}_{ji}:=\frac{1}{\tilde{\lmd}_j}\frac{\pi_i+ \pi_j}{2\pi_j |\tilde{ C}_j|}\frac{|\tilde{\Gamma}_{ij}|}{|y_i-y_j|}, \quad j\in \text{VF}(i); \quad \tilde{ P}_{ji}=0, \quad j\notin \text{VF}(i).
\end{aligned}
\end{equation}
{ Similar to Lemma \ref{lem_Markov}, we know $\tilde{P}$ is the transition probability matrix with row sum $1$. Denote the   diagonal rate matrix as $\tilde{R}=\text{diag}(\tilde{\lmd}_j)$, then we also obtain an approximated $Q$-matrix $\tilde{Q} = \tilde{R}(\tilde{P}-I)$.}
{ Notice  $\pi_i>0$ for all $i=1, \cdots, n$, so we always have $\tilde{\lmd}_i>0$ for all $i$.}
It also
satisfies the detailed balance property
\begin{equation}\label{db_cfl}
\tilde{\lmd}_j \tilde{ P}_{ji}\pi_j |\tilde{ C}_j| = \tilde{\lmd}_i \tilde{ P}_{ij} \pi_i |\tilde{ C}_i|,
\end{equation}
conservation laws and the stability analysis in Lemma \ref{stable}. 

Now we state and prove the convergence of the implementable finite volume  scheme \eqref{mp_pron}.
The bound of the error in the weighted $\ell^2$ norm is summarized in the following theorem. Due to the estimation error in the Voronoi cells and faces, the error in Theorem \ref{mainthm_conv} $e^T$  is replaced by $e^{2T}$.
Assume for $i=1, \cdots, n$, $|\text{VF}(i)|$, the cardinality of $\text{VF}(i)$, is order $1$.
\begin{thm}\label{mainthm_mp}
Suppose $\rho(\my,t)$, $t\in[0,T]$ is a smooth solution to the Fokker-Planck equation \eqref{FP_N1} on manifold $\nn\subset \bR^\ell$ with initial density $\rho^0(\my)$. Let $\{\tilde{\rho}_i(t)\}_{i=1}^n$ be the solution of the finite volume  scheme \eqref{mp_pron}. Let $\tilde{e}_i:= \rho(\my_i) - \tilde{\rho}_i$. {If $n$ is large enough, for $r$ satisfying Assumption \ref{assumption on voronoi cell}, we choose threshold $s=a_1 r^{d}$ for some constant $a_1$ in Algorithm \eqref{Voronoi cell on manifold algorithm},} with probability greater than $1-\frac{1}{n^2}$, we have 
\begin{equation}
\max_{t\in[0,T]} \sum_i \tilde{e}_i(t)^2 \frac{|C_i|}{\pi_i} \leq \big (\sum_i \tilde{e}_i(0)^2 \frac{|C_i|}{\pi_i} +cr \big)  e^{2T},
\end{equation}
where $c$ is a constant independent of $r$ and $n$.
\end{thm}

\begin{proof}

Define $\rho_i^e:=\frac{1}{|C_i|}\int_{C_i} \rho \ud y$.
Plug the exact solution into the numerical scheme
\begin{equation}\label{mp-exact 2}
\begin{aligned}
\pt_t (\rho_i^e |C_i|) = & \sum_{j\in \text{VF}(i)} \tsch \left(  \frac{\rho(\my_j) }{\pi_j}- \frac{\rho(\my_i)}{\pi_i}\right) \\
+& \sum_{j\in \text{VF}(i)} \int_{{\Gamma}_{ij}}  \pi \mathbf{n}_{ij} \cdot \nabla_\nn \frac{\rho}{\pi} \ud \hs^{d-1}  ~ -\sum_{j\in \text{VF}(i)} ~ \tsch \left(  \frac{\rho(\my_j) }{\pi_j}- \frac{\rho(\my_i) }{\pi_i}\right),
\end{aligned}
\end{equation}
where $\mathbf{n}_{ij}$ is the restriction of the unit outward normal vector field on $\Gamma_{ij}$. 
Subtracting the numerical scheme \eqref{mp_pron} from \eqref{mp-exact 2}, 
we have
\begin{equation}
\begin{aligned}
\frac{\ud}{\ud t}\tilde{e}_i |C_i| =&\sum_{j\in \text{VF}(i)} \frac{\pi_i+ \pi_j}{2|y_i-y_j|} |{\tilde{\Gamma}}_{ij}|\left( \frac{\tilde{e}_j}{\pi_j}- \frac{\tilde{e}_i}{\pi_i} \right) + \sum_{j\in \text{VF}(i)} \eps_{ij} +   \pt_t ( (\rho(\my_i) -\rho_i^e) |C_i|)+\frac{\ud}{\ud t} \tilde{\rho}_i(|\tilde{C}_i|-|C_i|),
\end{aligned}
\end{equation}
where
\begin{align}
\eps_{ij}:= \int_{\Gamma_{ij}}  \pi \mathbf{n}_{ij} \cdot \nabla_\nn \frac{\rho}{\pi} \ud \hs^{d-1}  ~ -& ~ \sch \left(  \frac{\rho(\my_j) }{\pi_j}- \frac{\rho(\my_i) }{\pi_i}\right)\\
&+\frac{\pi_i+ \pi_j}{2|y_i-y_j|} (|{\Gamma}_{ij}|-|{\tilde{\Gamma}}_{ij}|)\left( \frac{\rho(\my_j)}{\pi_j}- \frac{\rho(\my_i)}{\pi_i} \right)\nonumber
\end{align}
Note that  $\eps_{ij}$ is anti-symmetric, hence by the same argument in Theorem \ref{mainthm_conv}, we have 
\begin{align}
 \frac{\ud}{\ud t} \sum_i \tilde{e}_i^2\frac{|C_i|}{\pi_i} \leq & -\frac12\sum_i \sum_{j\in \text{VF}(i)}  \frac{\pi_i+ \pi_j}{2|y_i-y_j|} |{\tilde{\Gamma}}_{ij}|    \left(  \frac{\tilde{e}_j}{\pi_j}- \frac{\tilde{e}_i}{\pi_i}\right)^2 + \frac12 \sum_{i} \sum_{j\in \text{VF}(i)} \frac{\eps^2_{ij}}{\frac{\pi_i+ \pi_j}{2|y_i-y_j|} |{\tilde{\Gamma}}_{ij}|} \\
 & +\sum_i  [\pt_t(\rho(\my_i) -\rho_i^e)]^2  \frac{ |C_i|}{\pi_i}  + \sum_i  \big(\frac{\ud}{\ud t} \tilde{\rho}_i(\frac{|\tilde{C}_i|-|C_i|}{|C_i|}) \big)^2   \frac{ |C_i|}{\pi_i} + 2 \sum_i  e_i^2  \frac{ |C_i|}{\pi_i} \nonumber \\
\leq & \sum_{i} \sum_{j\in \text{VF}(i)} \frac{\eps^2_{ij}}{\frac{\pi_i+ \pi_j}{2|y_i-y_j|} |{\tilde{\Gamma}}_{ij}|} +\sum_i  [\pt_t(\rho(\my_i) -\rho_i^e)]^2  \frac{ |C_i|}{\pi_i}  \nonumber \\
& + \sum_i  \big(\frac{\ud}{\ud t} \tilde{\rho}_i(\frac{|\tilde{C}_i|-|C_i|}{|C_i|}) \big)^2   \frac{ |C_i|}{\pi_i} + 2 \sum_i  e_i^2  \frac{ |C_i|}{\pi_i} \nonumber\\
=&: \epsilon_1+\epsilon_2+ \epsilon_3+ 2 \sum_i  e_i^2  \frac{ |C_i|}{\pi_i}.  \nonumber
\end{align}
First, we estimate the term $\epsilon_1$, in particular, $ \frac{\eps^2_{ij}}{\frac{\pi_i+ \pi_j}{2|y_i-y_j|} |{\tilde{\Gamma}}_{ij}|}$ for $j\in \text{VF}(i)$. Since the exact solution is smooth such that
\begin{align}
|\rho(\my_i,t)-\rho(\my_j,t)| \leq C_{Lip} |\my_i-\my_j|,
\end{align}  by \eqref{bound on epsilon 1},
\begin{align}
{\eps_{ij}}=O((\diam(C_i)+d_{\nn}(\my_i,\my_j)) |\Gamma_{ij}|)+O(|{\Gamma}_{ij}|-|{\tilde{\Gamma}}_{ij}|). 
\end{align}
Hence,
\begin{align}
 \frac{\eps^2_{ij}}{\frac{\pi_i+ \pi_j}{2|y_i-y_j|} |{\tilde{\Gamma}}_{ij}|}=O(d_{\nn}(\my_i,\my_j)(\diam(C_i)+d_{\nn}(\my_i,\my_j))^2 \frac{ |\Gamma_{ij}|^2}{|{\tilde{\Gamma}}_{ij}|})+O(d_{\nn}(\my_i,\my_j) \frac{(|{\Gamma}_{ij}|-|{\tilde{\Gamma}}_{ij}|)^2}{|{\tilde{\Gamma}}_{ij}|}). 
\end{align}
Note that $|{\tilde{\Gamma}}_{ij}|\geq s= a_1 r^d$. Hence, by Proposition \ref{approximation of the area of a Voronoi face},
\begin{equation}
 \frac{(|{\Gamma}_{ij}|-|{\tilde{\Gamma}}_{ij}|)^2}{|{\tilde{\Gamma}}_{ij}|}=O(r^{d}).
\end{equation}

By Assumption \ref{assumption on voronoi cell} and Lemma \ref{geodesic vs euclidean}, $d_{\nn}(\my_i,\my_j) $ and $\diam(C_i)$ are of order $r$. {By Assumption \ref{assumption on voronoi cell} and Proposition \ref{perpendicular to Voronoi face}}, since $\nn$ is compact, there is a constant $K$ such that $|\Gamma_{ij}| \leq K r^{d-1}$.  Therefore, $\frac{\eps^2_{ij}}{\frac{\pi_i+ \pi_j}{2|y_i-y_j|} |{\tilde{\Gamma}}_{ij}|}=O(r^{d+1})$ and
\begin{align}
\epsilon_1=\sum_{i} \sum_{j\in \text{VF}(i)} \frac{\eps^2_{ij}}{\frac{\pi_i+ \pi_j}{2|y_i-y_j|} |{\tilde{\Gamma}}_{ij}|}=O(n r^{d+1} \max_{i}|\text{VF}(i)| )=O(r \max_{i}|\text{VF}(i)|),
\end{align}
where we use $n r^d$ goes to some constant in the last step.

Second, we estimate $\epsilon_2+\epsilon_3$. By Proposition \ref{approximation of the volume of a Voronoi cells} and  \eqref{volume bounds on the cells},
\begin{align}
\sum_i  \big(\frac{\ud}{\ud t} \tilde{\rho}_i(\frac{|\tilde{C}_i|-|C_i|}{|C_i|}) \big)^2   \frac{ |C_i|}{\pi_i} =O(r^2).
\end{align}
By \eqref{bound on epsilon 2} and Assumption \ref{assumption on voronoi cell},
\begin{align}
\sum_i  [\pt_t(\rho(\my_i) -\rho_i^e)]^2  \frac{ |C_i|}{\pi_i} =O(r^2).
\end{align}

We sum up all the terms, 
\begin{align}
 \frac{\ud}{\ud t} \sum_i \tilde{e}_i^2\frac{|C_i|}{\pi_i} \leq O(r \max_{i}|\text{VF}(i)|) + 2 \sum_i  e_i^2  \frac{ |C_i|}{\pi_i}.
\end{align}
In conclusion
\begin{align}
\max_{t\in[0,T]} \sum_i \tilde{e}_i(t)^2 \frac{|C_i|}{\pi_i} \leq \big (\sum_i \tilde{e}_i(0)^2 \frac{|C_i|}{\pi_i} +O(r \max_{i}|\text{VF}(i)| ) \big)  e^{2T}.
\end{align}
\end{proof}

\subsection{{Unconditionally} stable explicit time stepping and exponential convergence}\label{sec_explicit}
 To the end of this section, we show that the  detailed balance property \eqref{db_cfl} leads to stability and exponential convergence of a discrete-in-time Markov process.
 
{Under the detailed balance condition \eqref{db_cfl}, we recast \eqref{mp1_pron} to
\begin{equation}
\frac{\ud}{\ud t} \frac{\tilde{\rho}_i}{\pi_i} = \sum_{j\in \text{VF}(i)} \tilde{\lmd}_i \tilde{ P}_{ij} \frac{\tilde{\rho}_j}{\pi_j} - \tilde{\lmd}_i \frac{\tilde{\rho}_i}{\pi_i},
\end{equation}
}
Let $\rho^{k}_i $ be the discrete density at the discrete time $k\Delta t$. To achieve both the stability and the efficiency, we introduce the following unconditional stable explicit scheme 
\begin{equation}\label{559semi}
\frac{\rho_i^{k+1}}{\pi_i} = \frac{\rho^k_i}{\pi_i}-\tilde{\lmd}_i \Delta t \frac{\rho^{k+1}_i}{\pi_i}  +  \Delta t\sum_{j\in \text{VF}(i)} \tilde{\lmd}_i \tilde{ P}_{ij}   \frac{\rho^{k}_j }{\pi_j},
\end{equation}
where $\tilde{\lmd}_i$ and $\tilde{ P}_{ij}$ are defined in \eqref{definition of delta tilde and P tilde}.
The above equation is equivalent to 
\begin{equation}\label{semi_1}
\frac{\rho_i^{k+1}}{\pi_i} =\frac{\rho^k_i}{\pi_i} +  \frac{ \tilde{\lmd}_i  \Delta t }{1+\tilde{\lmd}_i \Delta t} \left(  \sum_{j\in \text{VF}(i)} \tilde{ P}_{ij}   \frac{\rho^{k}_j }{\pi_j} -\frac{\rho^k_i}{\pi_i} \right).
\end{equation}
For $u_i^{k+1}:=\frac{\rho_i^{k+1}}{\pi_i}$,  the matrix formulation of \eqref{semi_1} is
\begin{equation}\label{matrix_semi}
u^{k+1} = (I+\Delta t Q) u^k,
\end{equation}
where 
\begin{equation}\label{BBn}
Q:= \{\hat{b}_{ij}\}= \left\{\begin{array}{cc}
-\frac{ \tilde{\lmd}_i  \ }{1+\tilde{\lmd}_i \Delta t}  , \quad &j=i;\\
\frac{ \tilde{\lmd}_i   }{1+\tilde{\lmd}_i \Delta t} \tilde{P}_{ij}, \quad &j\neq i
\end{array}\right.
\end{equation}
satisfies $\sum_j \hat{b}_{ij}=0$. 

{
Below, we first summarize the explicit time stepping \eqref{559semi} as an algorithm and then prove the unconditionally stability.

\vspace{0.4cm}
\begin{algorithm}[H] \label{explicit FP algorithm}
\SetAlgoLined
\Parameter{Algorithm inputs: error tolerance $\epsilon$,  time step $\Delta t$, the initial distribution $(\rho^0_i)$, the target invariant measure $(\pi_i)$, the approximated volume $|\tilde{C}_k|$ of Voronoi cell  and ares $|\tilde{\Gamma_{k\ell}}|$ }

Compute transition probability matrix  $\tilde{P}_{ij}$ and $\tilde{\lmd}$ defined in \eqref{definition of delta tilde and P tilde}.

Compute discrete time transition probability matrix $Q$ defined in \eqref{BBn}.

$k\to k+1$ iteration: $\frac{\rho^{k+1}}{\pi} = (I+\Delta t Q) \frac{\rho^k}{\pi}$. Repeat until $\|\frac{\rho^{k+1}}{\pi}-1\|_{\8} < \epsilon.$

\caption{Explicit time stepping for Markov process}
\end{algorithm}
\vspace{0.3cm}
}

Now we show  $Q$ defined in \eqref{BBn} is the generator of a new Markov process.

For $w_i^{k+1}:=\rho_i^{k+1} |\tilde{C}_i|$, \eqref{559semi}, together with detailed balance property \eqref{db_cfl}, yields
\begin{equation}\label{tm567}
\rho_i^{k+1}|\tilde{C}_i| - \rho_{i}^k |\tilde{C}_i| = \Delta t \left( \sum_{j\in \text{VF}(i)} \tilde{\lmd}_j \tilde{ P}_{ji} \rho_j^{k} |\tilde{ C}_j| - \tilde{\lmd}_i \rho_i^{k+1} |\tilde{ C}_i|   \right),
\end{equation}
which can be recast as
\begin{equation}\label{tm568}
(1 + \Delta t \tilde{\lmd}_i ) \rho_i^{k+1}|\tilde{C}_i|  = (1+\Delta t \tilde{\lmd}_i) \rho_{i}^k|\tilde{C}_i| +\Delta t\left( \sum_{j\in \text{VF}(i)} \tilde{\lmd}_j \tilde{ P}_{ji} \rho_j^{k} |\tilde{ C}_j|  - \tilde{\lmd}_i|\tilde{C}_i| \rho_{i}^k  \right).
\end{equation}
Denote $g_i^{k+1}:= (1 + \Delta t \tilde{\lmd}_i ) \rho_i^{k+1}|\tilde{C}_i|$. \eqref{tm568} can be simplified as
\begin{equation}
g_{i}^{k+1} = g_i^k + \Delta t \left( \sum_j \frac{\tilde{\lmd}_j }{1+\Delta t \tilde{\lmd}_j}  \tilde{ P}_{ji} g_j^k -  \frac{\tilde{\lmd}_i }{1+\Delta t \tilde{\lmd}_i} g_i^k   \right).
\end{equation}
This is a new Markov process for $g_i$ with  transition probability $\tilde{P}_{ji}$ and a new jump rate $s_j=\frac{\tilde{\lmd}_j }{1+\Delta t \tilde{\lmd}_j}$.
With $Q$ in \eqref{BBn}, the matrix formulation for $g$ is
\begin{equation}
g^{k+1}= (I+\Delta t Q)^* g^k.
\end{equation}
One can check $(1 + \Delta t \tilde{\lmd}_i ) \pi_i|\tilde{C}_i|$ is a new equilibrium.

\begin{prop}\label{error bound for the unconditional stable explicit scheme}
{ Assume $\pi_i>0$ for all $i=1, \cdots, n$.} {Let $\tilde{\lmd}_i$ be the approximated jump rate and $\tilde{P}_{ij}$ be the approximated transition probability defined in \eqref{definition of delta tilde and P tilde}.} Let $\Delta t$ be the time step and consider the explicit time stepping \eqref{559semi}, { i.e., Algorithm \ref{explicit FP algorithm}}. Assume the initial data satisfies 
\begin{equation}\label{adjust}
\sum_i (1+ \tilde{\lmd}_i \Delta t) \rho_i^0 |\tilde{C}_i| = \sum_i (1+ \tilde{\lmd}_i \Delta t )\pi_i |\tilde{C}_i|.
\end{equation}
Then
we have
\begin{enumerate}[(i)]
\item the {conversation} law for $g_i^{k+1}:= (1 + \Delta t \tilde{\lmd}_i ) \rho_i^{k+1}|\tilde{C}_i|$, i.e.
\begin{equation}\label{conser1}
\sum_i (1+ \lmd_i \Delta t) \rho_i^{k+1} |\tilde{C}_i| = \sum_i (1+ \lmd_i \Delta t )\rho_i^{k} |\tilde{C}_i|;
\end{equation}
\item  the unconditional maximum principle for $\frac{\rho_i}{\pi_i}$
\begin{equation}\label{maxP}
\max_i \frac{\rho^{k+1}_j }{\pi_j}\leq \max_j \frac{\rho^{k}_j }{\pi_j}.
\end{equation}
\item the $\ell^\8$ contraction
\begin{equation}\label{l8semi}
\max_i \left| \frac{\rho^{k+1}_i}{\pi_i}-1\right| \leq \max_i \left| \frac{\rho^{k}_i}{\pi_i}-1\right| ;
\end{equation}
\item the exponential convergence
\begin{equation}\label{gap_error}
\left\| \frac{\rho^{k}_i}{\pi_i}-1\right\|_{\ell^\8} \leq c  |\lambda_2|^k, \quad |\lambda_2|<1,
\end{equation}
where $\lambda_2$ is the second eigenvalue (in terms of the magnitude) of $I+\Delta t Q$, i.e. $\lambda_2=1-\gap_{Q}\Delta t$ and $\gap_{Q}$ is the spectral gap of $Q$. 
\end{enumerate}
\end{prop}
\begin{proof}
First, recast \eqref{semi_1} as
\begin{equation}\label{tm547}
\frac{\rho_i^{n+1}}{\pi_i} =\frac{ 1}{1+\tilde{\lmd}_i \Delta t}\frac{\rho^n_i}{\pi_i} +  \frac{ \tilde{\lmd}_i  \Delta t }{1+\tilde{\lmd}_i \Delta t} \left(  \sum_{j\in \text{VF}(i)} \tilde{ P}_{ij}   \frac{\rho^{n}_j }{\pi_j}  \right),
\end{equation}
which gives the unconditional  maximum principle
\eqref{maxP}.

Second, from \eqref{tm547}, we have
\begin{equation}
\frac{\rho_i^{k+1}}{\pi_i}-1 =\frac{ 1}{1+\tilde{\lmd}_i \Delta t}\left( \frac{\rho^k_i}{\pi_i}-1 \right) +  \frac{ \tilde{\lmd}_i  \Delta t }{1+\tilde{\lmd}_i \Delta t} \sum_{j\in \text{VF}(i)} \tilde{ P}_{ij}  \left(   \frac{\rho^{k}_j }{\pi_j} -1 \right).
\end{equation}
Then we have
\begin{equation}
\begin{aligned}
\left|\frac{\rho_i^{k+1}}{\pi_i}-1\right| \leq \frac{ 1}{1+\tilde{\lmd}_i \Delta t}\left| \frac{\rho^k_i}{\pi_i}-1 \right|  +  \frac{ \tilde{\lmd}_i  \Delta t }{1+\tilde{\lmd}_i \Delta t} \sum_{j\in \text{VF}(i)} \tilde{ P}_{ij}   \left|\frac{\rho^k_j }{\pi_j} -1 \right|
\leq \max_i  \left|\frac{\rho^n_j }{\pi_j} -1 \right|,
\end{aligned}
\end{equation}
which gives \eqref{l8semi}.

Third, recall the matrix formulation \eqref{matrix_semi}.
Every element in $(I+\Delta t Q)^m$ is strictly positive for some $m$.  By Perron-Frobenius theorem, $\lambda_1=1$ is the 
simple, principal eigenvalue of $I+\Delta t Q$ with the ground state $u^*\equiv \{1, 1, \cdots, 1\}$ and other  eigenvalues $\lambda_i$  satisfy $|\lambda_i|<\lambda_1$. 
On one hand,  the mass conservation for initial data $u^0=\frac{\rho^0}{\pi}$ satisfies \eqref{conser1}, i.e.,
\begin{equation}
\sum_i( u^0_i - u^*_i) u^*_i (1+\Delta t \lmd_i)\pi|\tilde{C}|_i =0.
\end{equation}
On the other hand, $I+\Delta t Q$ is self-adjoint operator in the weighted  $l^2((1+\Delta t\lambda)\pi|C|)$ space,  we can express $u^0$ using 
\begin{equation}
u^0 - u^* = \sum_{j=2} c_j u_j , \quad u_j  \text{ is the eigenfunction corresponding to } \lambda_j.
\end{equation}
 Therefore, we have
\begin{equation}
u^{k}-u^* = (I+\Delta t Q)^k (u^0 - u^*) = \sum_{j=2} c_j \lambda_j^k u_j,
\end{equation}
which concludes
\begin{equation}
\left\| \frac{\rho^{k}_i}{\pi_i}-1\right\|_{\ell^\8} \leq c  |\lambda_2|^k\quad \text{ with } |\lambda_2|<1.
\end{equation}
Here $\lambda_2$ is the second eigenvalue (in terms of the magnitude) of $I+\Delta t Q$ sitting in the ball with radius $\lambda_1=1$ and thus $|\lambda_2|<1$.

Finally, taking summation with respect to $i$ in \eqref{tm567} shows
\begin{equation}
\begin{aligned}
\sum_i \left(\rho_i^{k+1}|\tilde{C}_i| - \rho_{i}^k |\tilde{C}_i| \right)=& \Delta t \left( \sum_{i,j} \tilde{\lmd}_j \tilde{ P}_{ji} \rho_j^{k} |\tilde{ C}_j| - \sum_i \tilde{\lmd}_i \rho_i^{k+1} |\tilde{ C}_i|   \right)\\
=& \Delta t \left( \sum_{j} \tilde{\lmd}_j  \rho_j^{k} |\tilde{ C}_j| - \sum_i \tilde{\lmd}_i \rho_i^{k+1} |\tilde{ C}_i|   \right),
\end{aligned}
\end{equation}
which gives \eqref{conser1}.

\end{proof}

{As a comparison, we also give some other standard stability estimates for both explicit and implicit schemes in Appendix \ref{otherS} and show that only the unconditional stable explicit scheme \eqref{559semi} achieves both the efficiency and  the stability.}
{We  refer to \cite{gao2020inbetweening} for successful applications of Algorithm \ref{explicit FP algorithm} to image morphing problems with 2D structured spacial grids. With structured grids, instead of Voronoi cell approximations obtained from sample points, the computations of explicit time stepping for Markov process using Algorithm \ref{explicit FP algorithm} are more accurate. \cite{gao2020inbetweening} also combines Algorithm \ref{explicit FP algorithm} with a thresholding dynamics to simulate mass-conserved shape dynamics for distribution with binary values ${1,2}$. }

\section{Simulations for Fokker-Planck solver}\label{sec_simu}
In this section, we conduct some challenging numerical simulations with reaction coordinates for the dumbbell, the Klein bottle and sphere.   We use the dataset $\{\my_i\}_{i=1}^{2000}$ with the reaction coordinates on the underlying manifolds including  dumbbell, Klein bottle and sphere to solve the Fokker-Planck equation \eqref{FP_N1} following the unconditionally stable explicit scheme \eqref{559semi}.

{ 
\subsection{Comparison with a ground-truth dynamics on sphere}
In this section, we construct a ground-truth exact solution given by an oscillated von Mises-Fisher distribution on the $2$-sphere in $\mathbb{R}^3$. This distribution is a commonly used distribution in physics and bioinformatics, for instance, to model the electric field-induced dipole interaction. For other complicated applications, it is hard to construct a ground-truth exact solution with an exact source term. So we refer to \cite{berry2015nonparametric1,berry2015nonparametric2} for other comparison methods without knowing an exact solution.

We choose the spherical coordinates as
\begin{equation}
 \theta \in [0,\pi],\quad  \varphi \in[0,2\pi],\quad \text{ with } \,\,  x=\cos \varphi \sin \theta,\quad y = \sin\varphi \sin \theta,\quad z = \cos \theta. 
\end{equation}
For $t\in[0, 2]$,
define three   parameters
\begin{equation}
\kappa(t) = 1+ 0.2\sin(t), \quad a(t) = \pi/2 + 0.2 \sin (3t), \quad b(t) =  5t. 
\end{equation}
Define the polar angle
\begin{equation}
\eta(\theta, \varphi, t) = \cos a(t) \cos \theta + \sin a(t) \sin \theta \cos(\varphi - b(t)).
\end{equation} 
Then we choose the exact solution as the von Mises–Fisher distribution
\begin{equation}\label{VonE}
\rho_v(\theta, \varphi, t) = C(\kappa(t))e^{\kappa(t) \eta(\theta, \varphi, t) },
\end{equation}
where $C(\kappa)=\frac{\kappa}{4\pi \sinh \kappa}.$

Based on the surface gradient and surface divergence on sphere
\begin{equation}
\nabla_{\nn} f = \frac{\pt f}{\pt \theta} \hat{\theta} + \frac{1}{\sin \theta} \frac{\pt f}{\pt \varphi} \hat{\varphi}, \quad \nabla_{\nn} \cdot \vec{F} = \frac{1}{\sin\theta} \frac{\pt}{\pt \theta} [\sin\theta F_\theta] + \frac{1}{\sin \theta} \frac{\pt F_\varphi}{\pt \varphi},
\end{equation}
then it satisfies  Fokker-Planck equation \eqref{FP_N1} with source term
\begin{equation}\label{source46}
\begin{aligned}
&g(\theta, \varphi,t)= \divn (\nabla_\nn \rhon + \rhon_t \nabla_\nn U_{\nn}) -\pt_t \rhon \\
=&\frac{\pt^2 \rhon}{\pt \theta^2} + \frac{\pt \rhon}{\pt \theta} \frac{\pt U_{\nn}}{\pt \theta}+ \rhon \frac{\pt^2 U_{\nn}}{\pt\theta^2} +  {\cot \theta} [\frac{\pt\rhon}{\pt \theta} + \rhon \frac{\pt U_{\nn}}{\pt \theta}] + \frac{1}{\sin^2 \theta} [\frac{\pt^2 \rhon}{\pt \varphi^2} + \frac{\pt \rhon}{\pt \varphi} \frac{\pt U_{\nn}}{\pt \varphi} + \rhon \frac{\pt^2 U_{\nn}}{\pt \varphi^2} ] - \pt_t \rhon. 
\end{aligned}
\end{equation}
Take $U=1$,  plugging \eqref{VonE} into the RHS of \eqref{source46},
we obtain
\begin{equation}\label{gss}
g(\theta, \varphi, t) = \rho \Bigg[ \kappa^2 \big[\eta^2_\theta + \sin^2 a \sin^2(\varphi-b) \big] -2\kappa  \eta   - \frac{C'}{C} \kappa'- (\kappa' \eta + \kappa \eta_t) \Bigg];
\end{equation}
see detailed in Appendix \ref{app_von}.

With 
$\pi = e^{-U}$, and the source term $g$ computed from the exact solution \eqref{gss}.
Then Algorithm \ref{explicit FP algorithm}, i.e., the explicit scheme \eqref{semi_1}, becomes
\begin{equation} 
\frac{\rho_i^{k+1}}{\pi_i} =\frac{\rho^k_i}{\pi_i} +  \frac{ \tilde{\lmd}_i  \Delta t }{1+\tilde{\lmd}_i \Delta t} \left(  \sum_{j\in \text{VF}(i)} \tilde{ P}_{ij}   \frac{\rho^{k}_j }{\pi_j} -\frac{\rho^k_i}{\pi_i} \right) - \Delta t \frac{g^k_i}{\pi_i}.
\end{equation}
Here  for the discrete source term $g^k_i$, we use continuous time derivatives at time step $k$ and discrete spacial derivative on grid $i$. 
For $u_i^{k+1}:=\frac{\rho_i^{k+1}}{\pi_i}$, with the additional source term $g(\theta,\varphi,t)$,  the matrix formulation with $Q$ defined in \eqref{BBn} is
$
u^{k+1} = (I+\Delta t Q) u^k  - \Delta t\frac{g^k}{\pi}.
$

To compare the numerical solution and the exact solution with a long time validation. We take $\Delta t = 0.001$ and  final time as $T = 2000*\Delta t$ with iteration number $2000$.  We first sample $2000$ data points on a unit sphere $\nn = S^2\subset \mathbb{R}^3$, then we  compute the approximated Voronoi cell volumes $|\tilde{C}_i|_{i=1}^n$ and areas $\tilde{\Gamma}_{ij}$ from Algorithm \ref{Voronoi cell on manifold algorithm} by taking the bandwidth $r=0.3.$ The equilibrium $\{\pi_i\}$ is taking to be constant, which is normalized so that the total mass condition \eqref{adjust} is satisfied. In Fig. \ref{fig:groud}. We plot 6  snapshots at $t = 0, 0.4, 0.8, 1.2, 1.6, 2.0$ for both numerical solution $\rho_i$ and exact solution ${\rho_v}(\theta, \varphi,t)$ in \eqref{VonE} starting from the same initial data given by $\rho_v(\theta, \varphi,0)$. We also list Table \ref{Table:e} to show the root mean square error (RMSE) $e:=\sqrt{ \frac{1}{2000}\sum_{i=1}^{2000}|\rho_i-\rho_v(i)|^2 }$ at these 6 times. A video showing the dynamics of both numerical solution $\rho_i$ and exact solution ${\rho_v}$  is provided in \url{https://youtu.be/x98J8CSYBq8}. 
 \begin{table}[ht]
\caption{The root mean square error $e$}\label{Table:e}
\begin{tabular}{|c|c|c|c|c|c|c|}
\hline 
Time & $0$ & $0.4$ & $0.8$ & $1.2$ & $1.6$ & $2$ \\ 
\hline 
RMSE & $0$ & $0.0151$ & $0.0138$ & $0.0126$ & $0.0149$ & $0.0140$ \\ 
\hline 
\end{tabular} 
\end{table}
\begin{figure}
\includegraphics[scale=0.32]{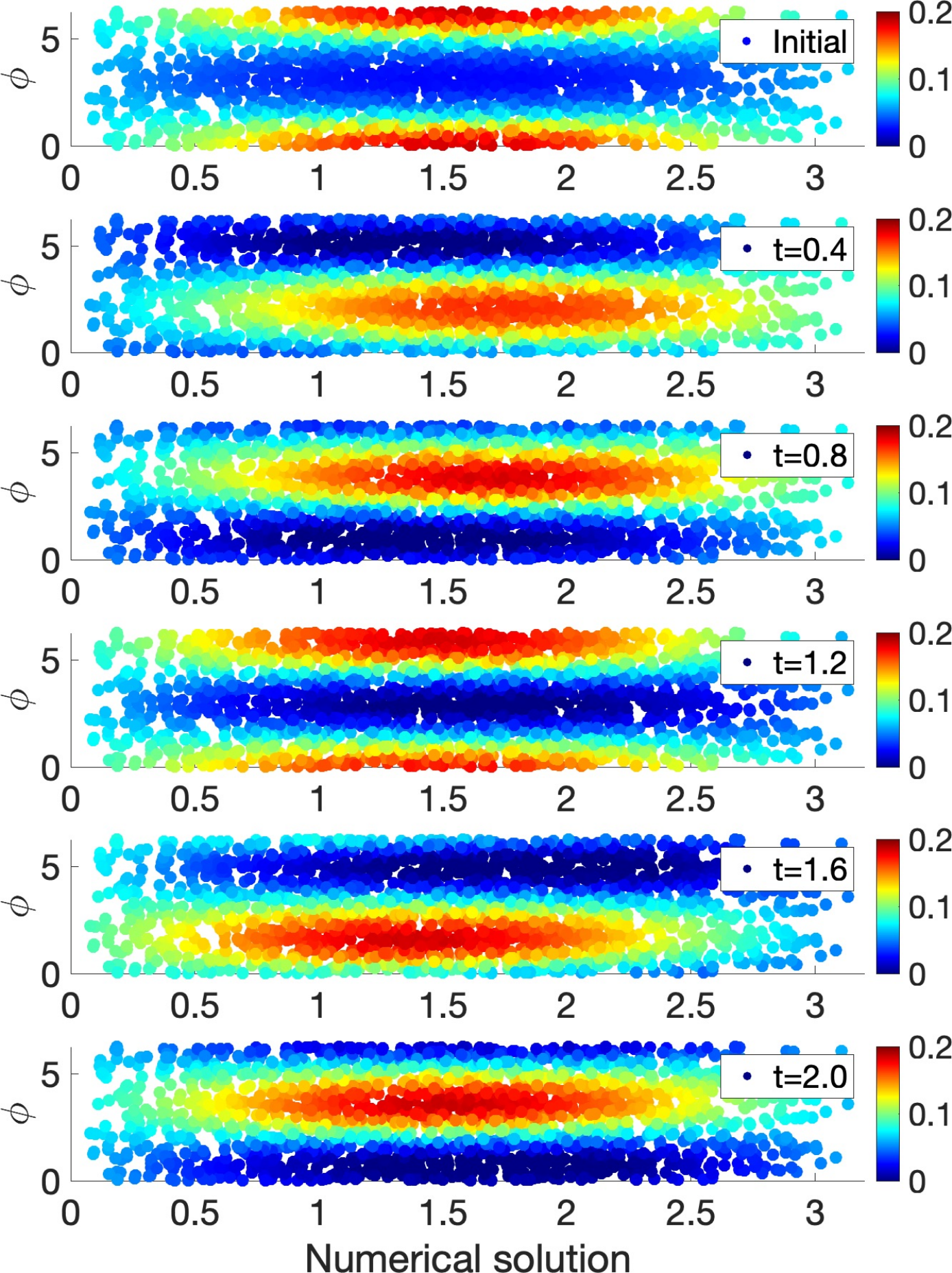} 
\includegraphics[scale=0.32]{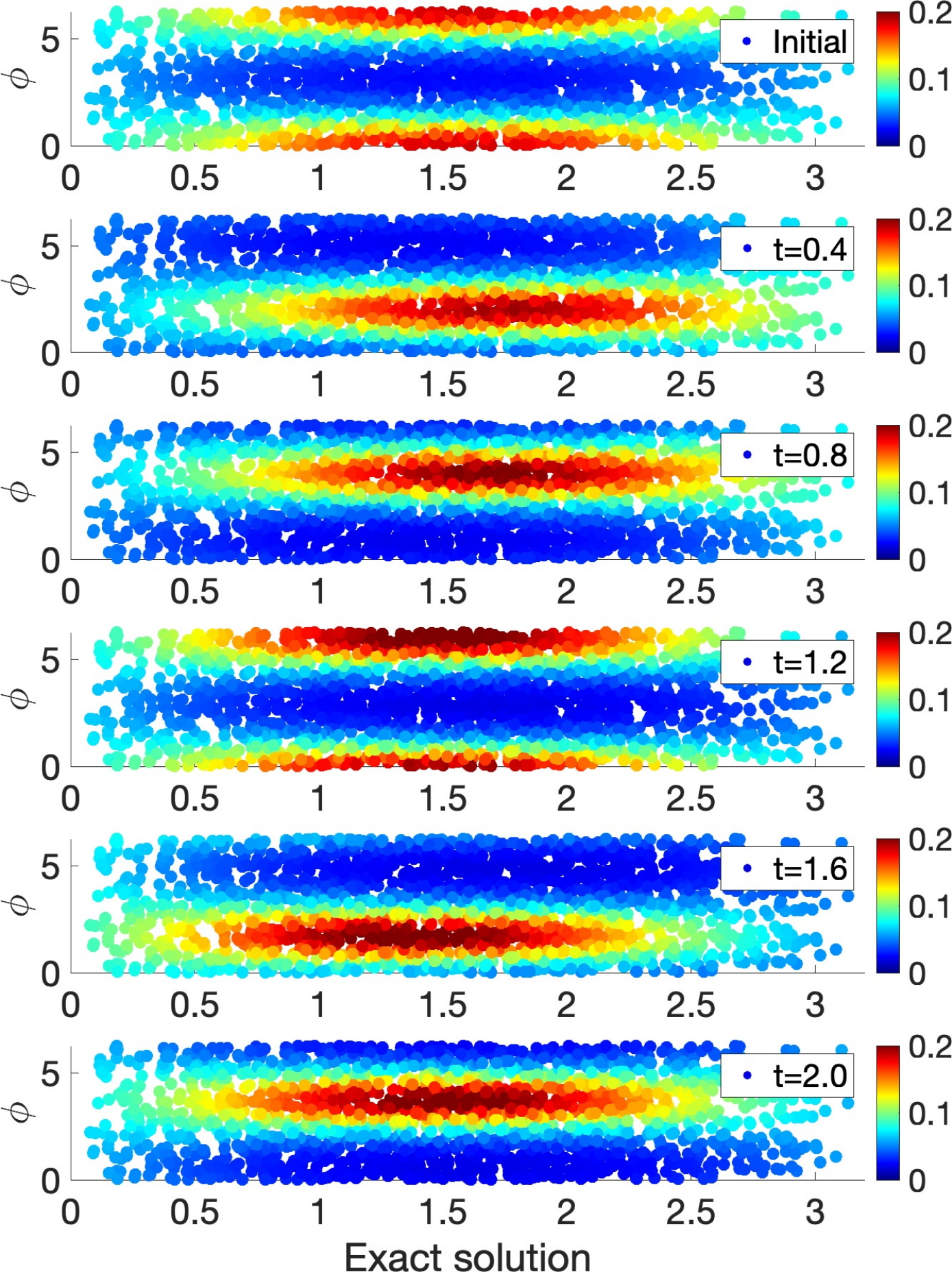} 
  \caption{Left: The numerical solution  $\rho^k_i$ in terms of $\theta, \varphi$   over $2000$ points $\{\my_i\}_{i=1}^{2000} \subset S^2 \subset \mathbb{R}^3$.   Right: The exact solution given by the oscillated von Mises–Fisher distribution ${\rho_v}(\theta, \varphi,t)$ in \eqref{VonE}.  We plot 6 snapshots at $t = 0, 0.4, 0.8, 1.2, 1.6, 2.0$.}\label{fig:groud}
\end{figure}

}

\subsection{\textnormal{\textbf{Example I: Fokker-Planck evolution on dumbbell.}} }

\

Suppose $(\theta,\phi) \in [0, 2\pi) \times [0, \pi)$, then we have the following dumbbell {in} $\mathbb{R}^{200}$ parametrized as $(x,y,z,0,\cdots, 0)=f_1(\theta, \phi) \in \mathbb{R}^{200}$, where
\begin{align}\label{parametrization of the dumbbell}
& r=\sqrt{\sqrt{1+0.95^4(\cos(2\phi)^2-1)}+0.95^2\cos(2\phi)} \\
& x=r \sin(\phi)\cos(\theta),\quad 
 y=r \sin(\phi)\sin(\theta),\quad  z=r \cos(\phi).  \nonumber 
\end{align}
After {composition} with a dilation and rotation map $f_2$ of $\mathbb{R}^{200}$ , we have an embedded dumbbell $\mm \subset \mathbb{R}^{200}$. {Suppose} $f_2 \circ f_1(\theta, \phi)$  is the parametrization of $\mm$.  We sample $4000$ points $(\theta_{1}, \phi_{1}), \cdots, (\theta_{4000}, \phi_{4000})$ on $[0, 2\pi) \times [0, \pi)$. Let $\mx_i=f_2 \circ f_1(\theta_i, \phi_i)$, then we have a non uniform sample $\{\mx_i\}_{i=1}^{4000}$ on $\mm$. We apply the diffusion map to find the reaction coordinates $\{\my_i\}_{i=1}^{4000}$ of $\{\mx_i\}_{i=1}^{4000}$ in $\mathbb{R}^3$, i.e. $\{\my_i\}_{i=1}^{4000}$ can be regarded {as} a non uniform sample on a dumbbell  $\nn \subset \mathbb{R}^3$. 

Suppose $\psi_i$ is the $i$ th eigenfunction of the Laplace-Beltrami operator on $\nn$.  Assume the initial density $\rho^0$ is $\psi_2$ plus some constant (so that $\rho^0$ is positive) as shown in Fig \ref{initial and final dumbbell}. Assume the equilibrium density $\pi$ is $\psi_8$ plus some constant as shown in Fig \ref{initial and final dumbbell}. We first obtain the approximated Voronoi cell volumes $|\tilde{C}_i|_{i=1}^{4000}$ and the areas $\tilde{\Gamma}_{ij}$ from Algorithm \ref{Voronoi cell on manifold algorithm} by taking the bandwidth $r=0.16$ and threshold $s=0$. Then we adjust the initial data, i.e., we replace $\rho^0$ by $c\rho^0$ such that  \eqref{adjust} holds.  We set the time step $\Delta t=0.05$. Let $T=k \Delta t$ for the integer $k$ and $1 \leq k \leq 20000$, i.e., we iterate the scheme for $20000$ times and set the final time to be $T=20000*\Delta t = 1000.$ We use the unconditional stable explicit scheme \eqref{559semi} to solve $\rho^k$.  We compare the numerical relative error in maximum norm with the theoretic relative error, $|\lambda_2|^k=0.9997^k$ in \eqref{gap_error}, in the semilog-plot in Fig \ref{error comparison dumbbell}. The exponential convergence rate is exactly same. To clearly see the dynamics of the change of the density over the $4000$ points, we plot $\rho^k$ for $k=20, 60, 100, 160, 220, 4000$, correspondingly $T=1, 3, 5, 8, 11, 200$ in Fig \ref{time evolution dumbbell}

\begin{figure}
  \includegraphics[scale=0.45]{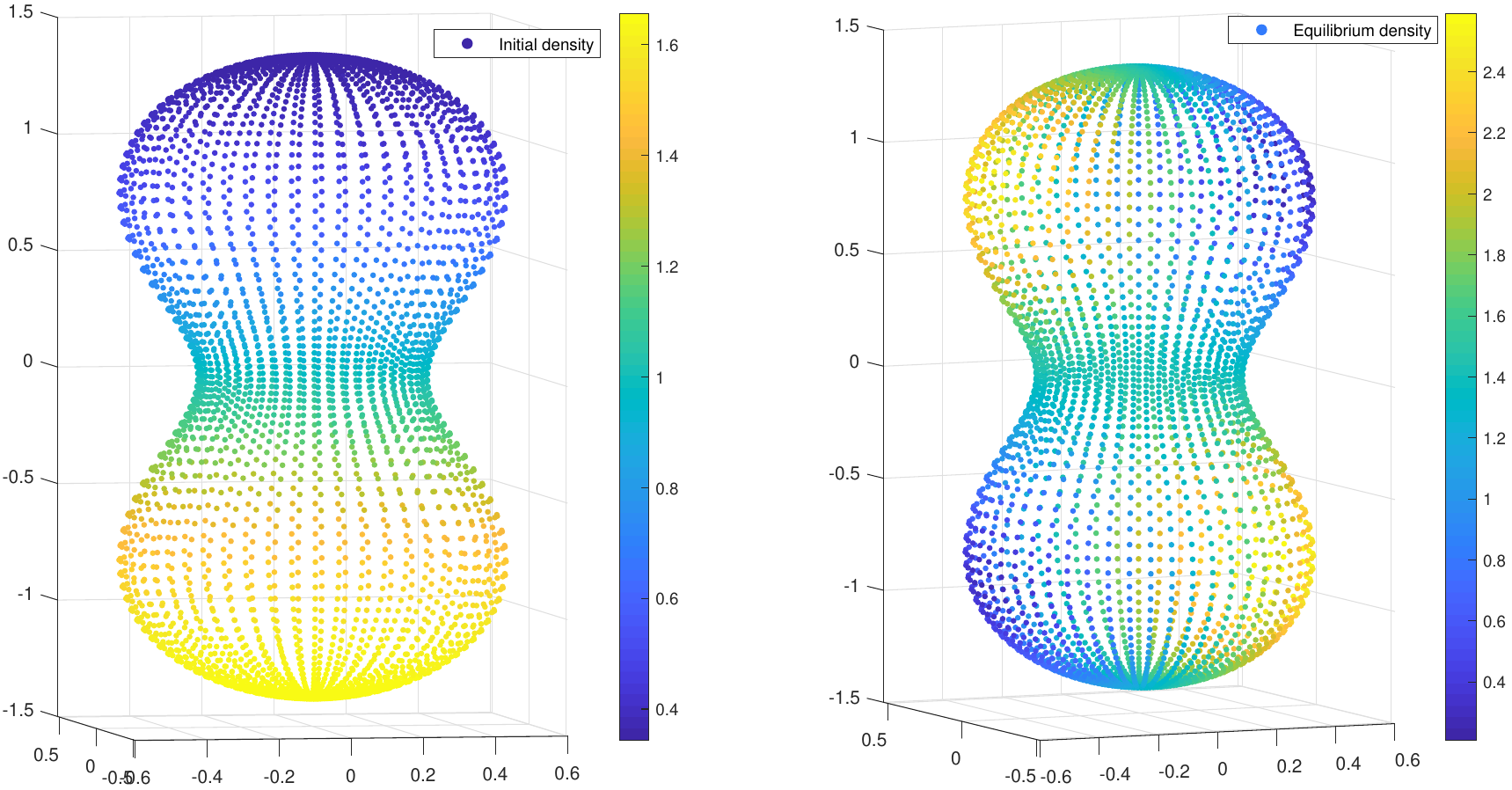} 
  \caption{Left: The initial density is the second eigenfunction of the Laplace Beltrami operator on a dumbbell $\nn \subset \mathbb{R}^3$ plus a constant. We plot it over $4000$ points $\{\my_i\}_{i=1}^{4000} \subset \nn \subset \mathbb{R}^3$.  Right: The equilibrium density is the eighth eigenfunction of the Laplace Beltrami operator on a dumbbell $\nn \subset \mathbb{R}^3$ plus a constant. We plot it over $4000$ points $\{\my_i\}_{i=1}^{4000} \subset \nn \subset \mathbb{R}^3$.  }\label{initial and final dumbbell}
\end{figure}

\begin{figure}
  \includegraphics[scale=0.8]{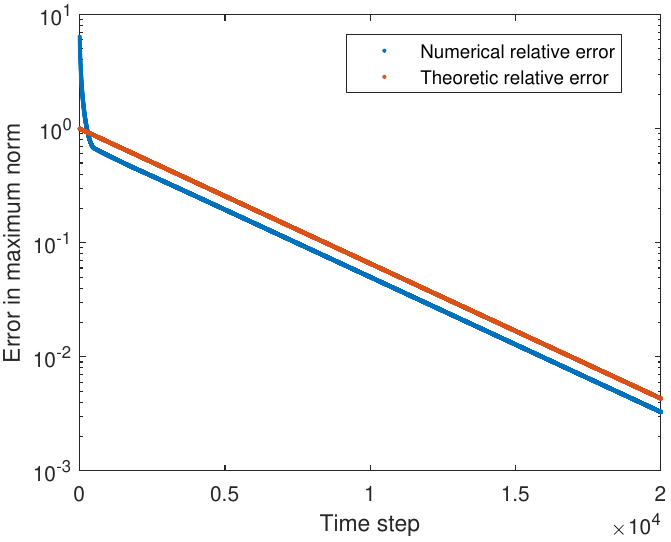} 
  \caption{The semilog-plot comparison between the numerical relative error with theoretic relative error.  The numerical relative error is the error from the  unconditional stable explicit scheme \eqref{559semi} with $\Delta t=0.05$ and $1 \leq k \leq 20000$. The theoretic relative error is {based} on \eqref{gap_error} with  $|\lambda_2|^k=0.9997^k$ .}\label{error comparison dumbbell}
\end{figure}

\begin{figure}
\includegraphics[width=14cm, height=13cm]{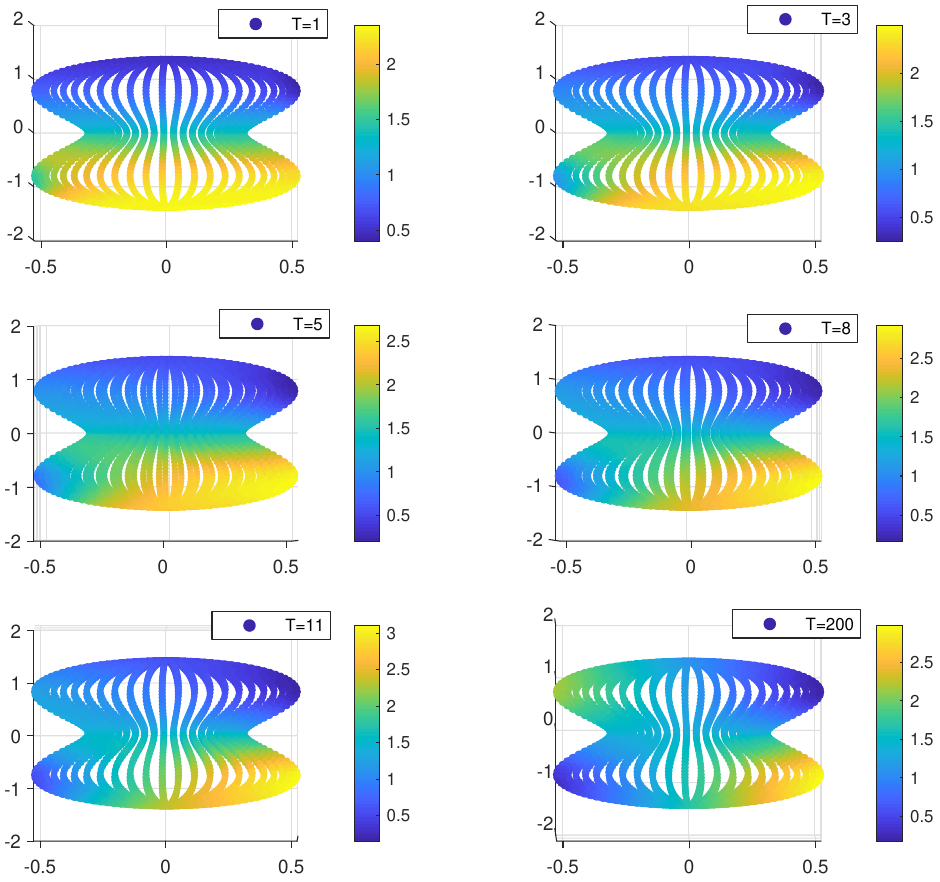} 
\caption{The density dynamics $\rho^k$ from the  unconditional stable explicit scheme \eqref{559semi} with $\Delta t=0.05$. We plot $\rho^k$
for $k=20, 60, 100, 160, 220, 4000$, correspondingly on  time $T=1, 3, 5, 8, 11, 200$.}\label{time evolution dumbbell}
\end{figure}

{\subsection{\textnormal{\textbf{Example II: Fokker-Planck evolution on Klein bottle.}} }
  
\

Suppose $(\theta,\phi) \in [0, 2\pi) \times [0, 2\pi)$, then we have the following Klein bottle in $\nn \subset \mathbb{R}^{4}$ parametrized as $(x,y,z,w)=f(\theta, \phi) \in \mathbb{R}^{4}$, where
\begin{align}\label{parametrization of klein bottle}
& x=(1+0.3\cos(\theta))\cos(\phi) \\
& y=(1+0.3\cos(\theta))\sin(\phi) \nonumber\\
& z=0.3\sin(\theta)\cos(\frac{\phi}{2}) \nonumber\\
& w=0.3\sin(\theta)\sin(\frac{\phi}{2}) \nonumber
\end{align} 
We sample $2000$ points $(\theta_{1}, \phi_{1}), \cdots, (\theta_{2000}, \phi_{2000})$ on $[0, 2\pi) \times [0, 2\pi)$. Let $\my_i= f(\theta_i, \phi_i)$, then we have  non uniform samples $\{\my_i\}_{i=1}^{2000}$ on $\nn$. We can regard them as the reaction coordinates of 2000 points sampled on $\mm$ (a manifold diffeomorphic to a Klein bottle) in some high dimensional space. In this example, we will visualize the functions on the Klein bottle by two methods. First, consider the projection from $\mathbb{R}^4$ to $\mathbb{R}^3$  by $(x,y,z,w) \rightarrow (x,y,z)$. The restriction of the projection on $\nn$ maps the Klein bottle to a pinched torus in $\mathbb{R}^3$. Second, consider  the projection from $\mathbb{R}^4$ to $\mathbb{R}^3$  by $(x,y,z,w) \rightarrow (y,z,w)$. The restriction of the projection on $\nn$ maps the Klein bottle to a Roman surface in $\mathbb{R}^3$. For any function on the Klein bottle, we will visualize it by plotting it on both the  pinched torus and the Roman surface.
\begin{figure}
  \includegraphics[scale=0.9]{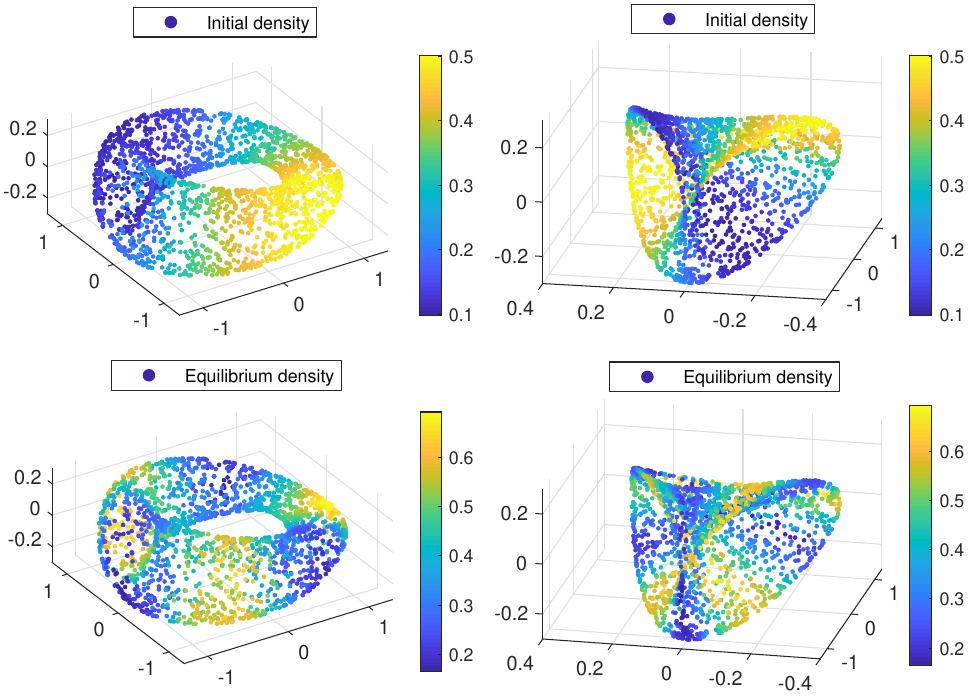} 
  \caption{Top two panels: The initial density is the second eigenfunction of the Laplace Beltrami operator on a Klein bottle  $\nn \subset \mathbb{R}^4$ plus a constant. We plot it over the pinched torus and the roman surface in $\mathbb{R}^3$ respectively.  Bottom two panels: The equilibrium density is the seventh eigenfunction of the Laplace Beltrami operator on a Klein bottle  $\nn \subset \mathbb{R}^4$ plus a constant. We plot it over the pinched torus and the roman surface in $\mathbb{R}^3$ respectively. }\label{initial and final bottle}
\end{figure}

Suppose $\psi_i$ is the $i$ th eigenfunction of the Laplace-Beltrami operator on $\nn$.  Assume the initial density $\rho^0$ is $\psi_2$ plus some constant (so that $\rho^0$ is positive) as shown in Fig \ref{initial and final bottle}. Assume the equilibrium density $\pi$ is $\psi_7$ plus some constant as shown in Fig \ref{initial and final bottle}. We first obtain the approximated Voronoi cell volumes $|\tilde{C}_i|_{i=1}^{2000}$ and the areas $\tilde{\Gamma}_{ij}$ from Algorithm \ref{Voronoi cell on manifold algorithm} by taking the bandwidth $r=0.23$ and threshold $s=0$. Then we adjust the initial data, i.e., we replace $\rho^0$ by $c\rho^0$ such that  \eqref{adjust} holds.  We set the time step $\Delta t=0.05$. Let $T=k \Delta t$ for the integer $k$ and $1 \leq k \leq 10000$, i.e., we iterate the scheme for $10000$ times and set the final time to be $T=10000*\Delta t = 500.$ We use the unconditional stable explicit scheme \eqref{559semi} to solve $\rho^k$.  We compare the numerical relative error in maximum norm with the theoretic relative error, $|\lambda_2|^k=0.9993^k$ in \eqref{gap_error}, in the semilog-plot in Fig \ref{error comparison bottle}. The exponential convergence rate is exactly {the} same. To clearly see the dynamics of the change of the density over the $2000$ points, we plot $\rho^k$ for $k=50, 1000, 2000, 10000$, correspondingly $T=2.5, 50, 100, 500$ in Fig \ref{time evolution bottle}

\begin{figure}
  \includegraphics[scale=0.8]{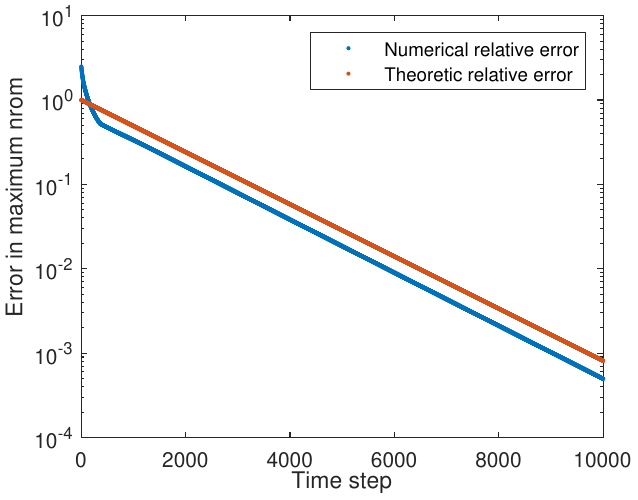} 
  \caption{The semilog-plot comparison between the numerical relative error with theoretic relative error in the Klein bottle example.  The numerical relative error is the error from the  unconditional stable explicit scheme \eqref{559semi} with $\Delta t=0.05$ and $1 \leq k \leq 10000$. The theoretic relative error is base on \eqref{gap_error} with  $|\lambda_2|^k=0.9993^k$.}\label{error comparison bottle}
\end{figure}

\begin{figure}
  \includegraphics[width=13cm, height=16cm]{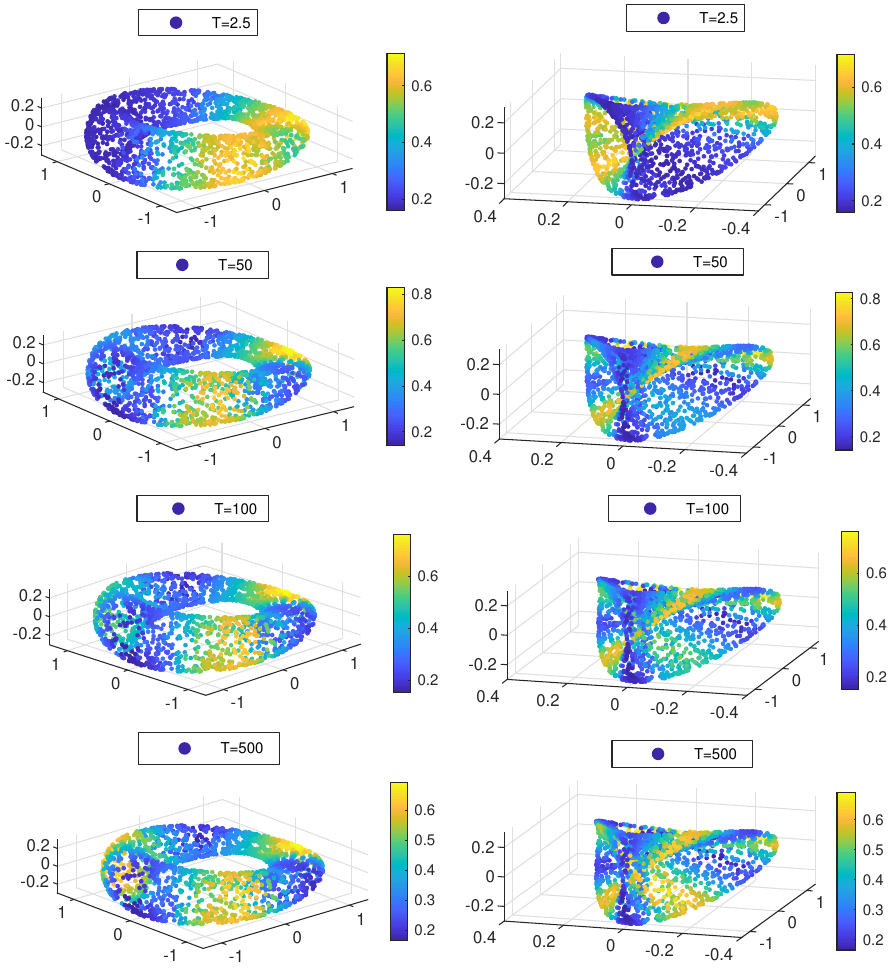} 
  \caption{The density dynamics $\rho^k$ from the  unconditional stable explicit scheme \eqref{559semi} with $\Delta t=0.05$. On the left four panels, we plot $\rho^k$ for $k=50, 1000, 2000, 10000$ corresponding to time  $T=2.5, 50, 100, 500$ on the pinched torus.  On the right four panels, we plot $\rho^k$ for $k=50, 1000, 2000, 10000$ corresponding to time  $T=2.5, 50, 100, 500$ on the Roman surface.}\label{time evolution bottle}
\end{figure}
}

\subsection{\textnormal{\textbf{Example III: The ``breakup" of Pangaea via Fokker-Planck evolution on sphere.}} }

\
In this example, we use the Fokker-Planck evolution on sphere to simulate the dynamics of the altitude of continents and the depth  of oceans for earth based on the dataset for initial distribution of Pangaea {supercontinent (250 million years ago)} and the equilibrium distribution of the current earth.

Suppose $\{\my_i\}_{i=1}^{2000}$ are the points on the unit sphere $\nn=S^2 \subset \mathbb{R}^3$, i.e., $\{\my_i\}_{i=1}^{2000}$ are the reaction coordinates of $2000$ points on  $\mm$ (a manifold diffeomorphic to a sphere) in some high dimensional space. Assume the initial density $\rho^0_i$ at $\{\my_i\},\, i=1,\cdots, 2000$ are extracted from  the Pangaea continents map file \cite{wiki}  as shown in Fig \ref{fig_earth1} (down left). { The value of the initial density $\rho^0_i\in\{1,2\}$ where $1$ represents  oceans and $2$ represents continents.}  Assume the equilibrium $\{\pi_i\}$ at $\{\my_i\}$ are collected from the ETOPO5 topography data \cite{RN179} expressing the altitude of continents and the depth  of oceans for earth. { The value of the equilibrium $\{\pi_i\}$ ranges from $-7000$ to $7000$ where the positive values represent the altitude of continents, negative values represent the depth of oceans and $0$ represents sea level. Before plugging into the Fokker-Planck equation, we add a constant $c_p$ to $\pi_i$ such that $\pi_i>0$ for all $i$. However, when showing the evolution of continents in figures, we subtract this constant and present the true physical altitudes.} 

 We first obtain the approximated Voronoi cell volumes $|\tilde{C}_i|_{i=1}^n$ and areas $\tilde{\Gamma}_{ij}$ from Algorithm \ref{Voronoi cell on manifold algorithm} by taking the bandwidth $r=0.3$ and threshold $s=0$. Then we adjust the initial data, i.e., we replace $\rho^0$ by $c\rho^0$ such that { the total mass condition}  \eqref{adjust} holds.     We set the time step $\Delta t=0.05$. Let $T=k \Delta t$ for the integer $k$ and $1 \leq k \leq 10000$, i.e., we iterate the scheme for $10000$ times and set the final time to be $T=10000*\Delta t = 500.$ We use the {unconditionally} stable explicit scheme \eqref{559semi} to solve $\rho^k$.  In Fig \ref{fig_earth1} (up), the numerical relative error in maximum norm is semilog-plotted using circles. Compared with decay of the theoretic relative error $|\lambda_2|^k$ in \eqref{gap_error}, blue line in the semilog-plot, the exponential convergence rate is exactly same.  The initial 3D plot of Pangaea continents is shown in Fig \ref{fig_earth1} (down left) while the final 3D plot at $T=500$ of the simulated altitude  and depth  of continents and oceans  are shown in Fig \ref{fig_earth1} (down right)\footnote{The altitude and depth exceed the range $[-3800\text{m},3800\text{m}]$ is cut off for clarity.}.  To clearly see the dynamics of  altitude and depth of continents and oceans at $n$ points with longitude and latitude, starting from {the} same Pangaea continents  with time step $\Delta t=0.05$, four snapshots at { $T=0, 1, 10, 75$ of the dynamics   are shown in Fig  \ref{fig_earth2}. A video is also provided to show the dynamics of the density \url{https://youtu.be/j5XBPdQhEEs}. Here we used nonuniform time intervals since the shapes of continentals (the region with positive altitudes) quickly move from the initial Pangaea supercontinents towards the equilibrium shape of current continents.} If we only care about the shape of the continents and keep the binary-valued density during the shape evolution, we refer to \cite{gao2020inbetweening} for the thresholding adjustment method.

\begin{figure}
  \includegraphics[scale=0.4]{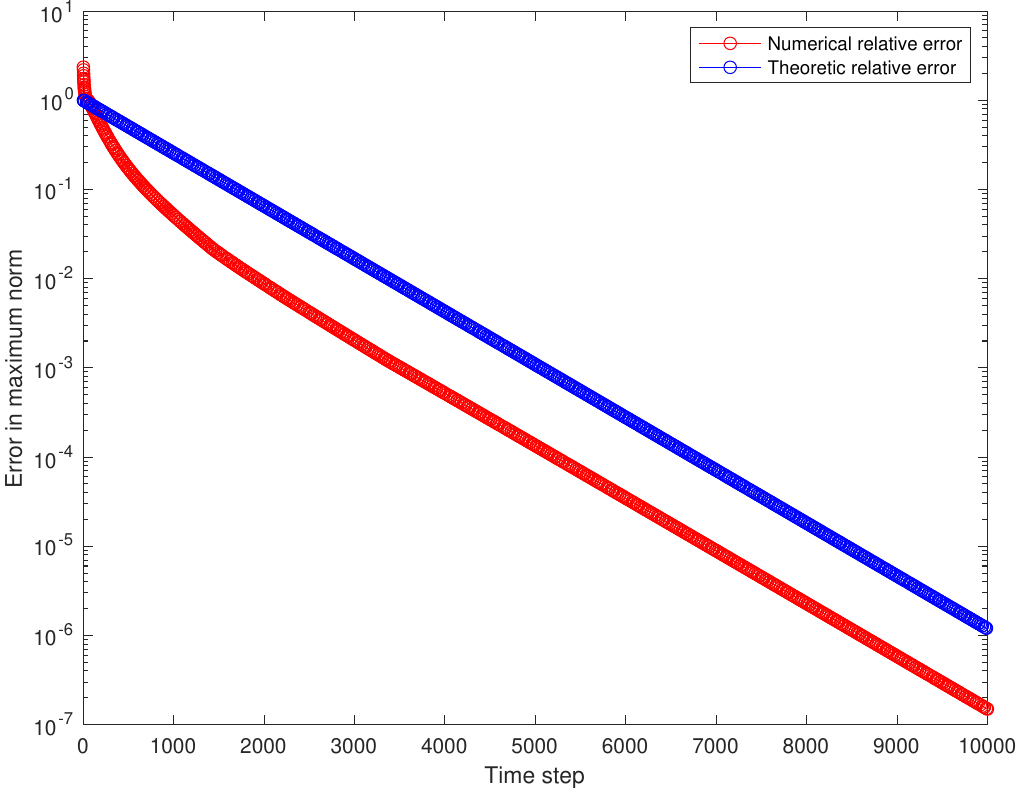}
   
  \includegraphics[scale=0.4]{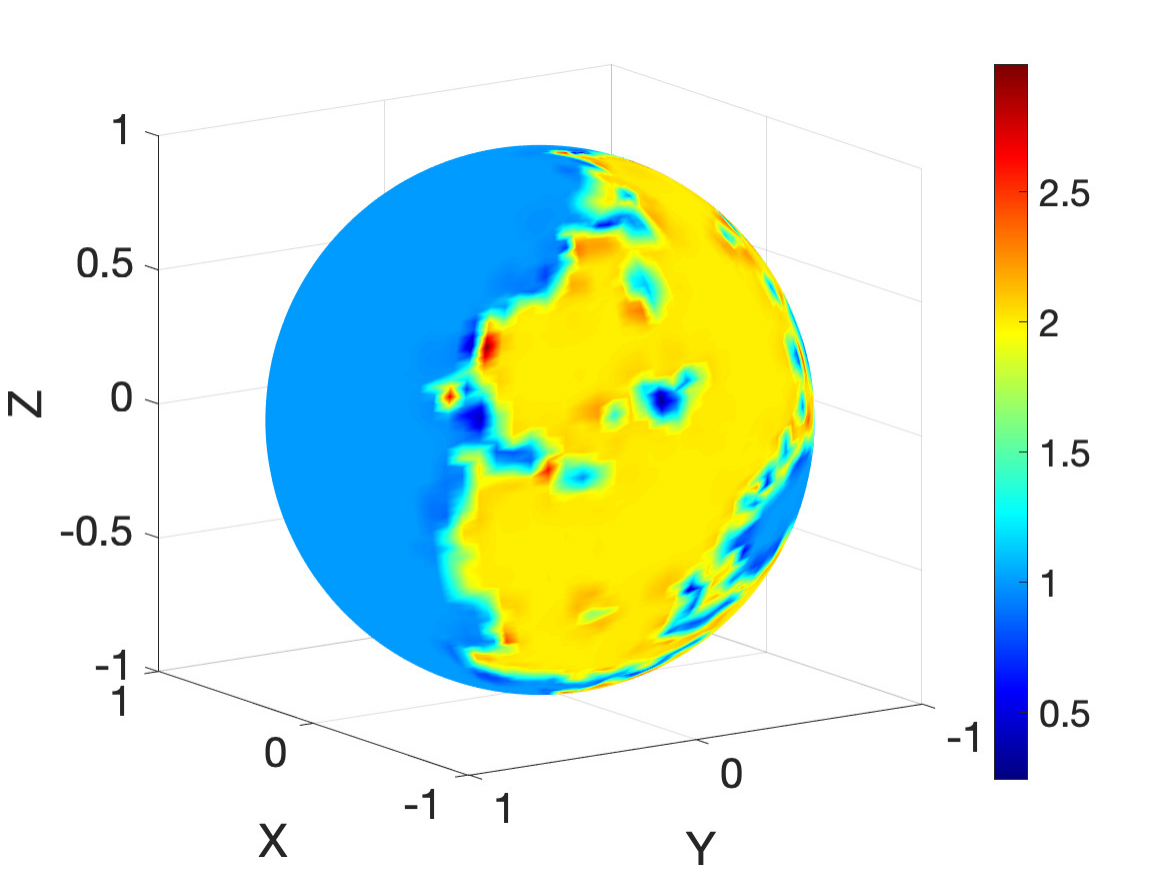} 
  \includegraphics[scale=0.4]{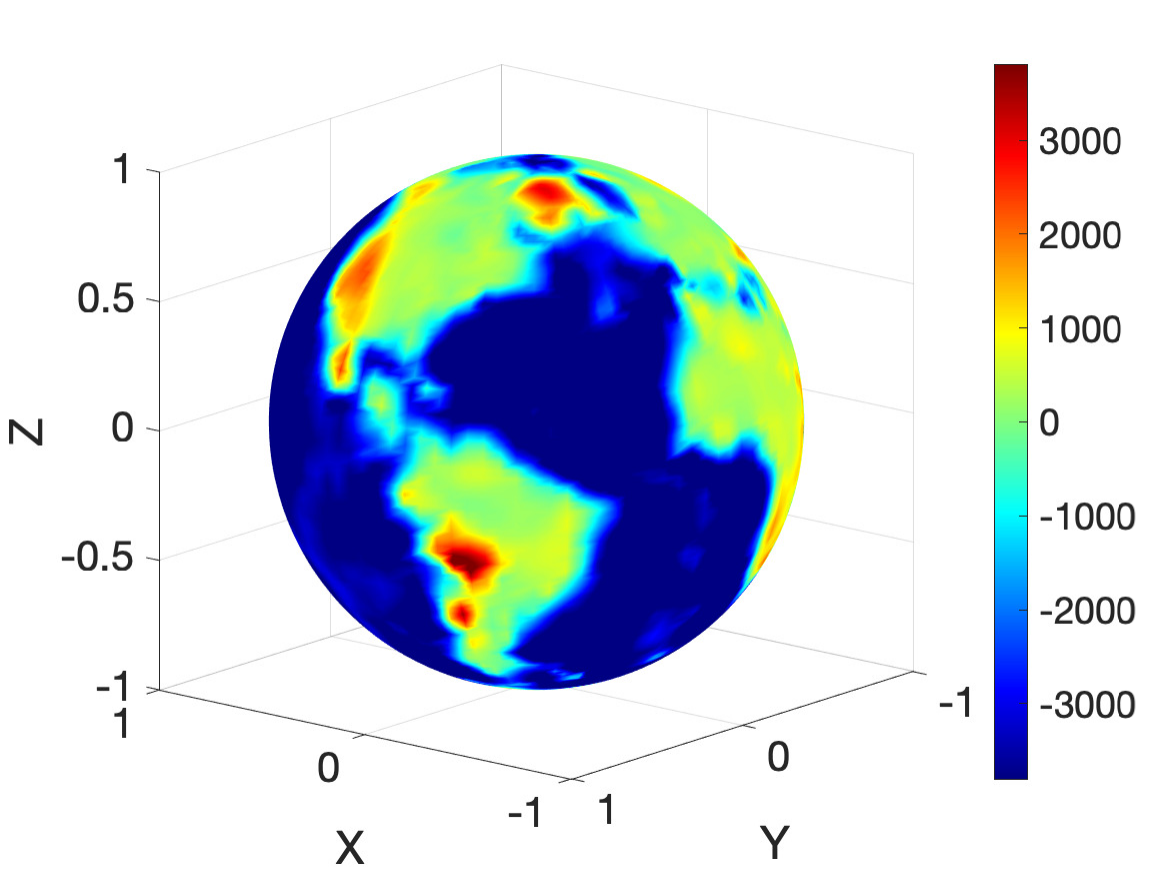} 
  \caption{Simulations for the density dynamics of altitude and depth of continents and oceans  starting from pangaea (down left) to the final altitude of land-ocean (down right) with parameters  $\ud t = 0.05$, $T=500$. (up)The semilog-plot comparison between the numerical relative error in maximum norm  (blue circle) with theoretic relative error $|\lambda_2|^k=0.9985^k$ (blue line)  in \eqref{gap_error}. }\label{fig_earth1}
\end{figure}
\begin{figure}
\includegraphics[width=15.5cm, height=10cm]{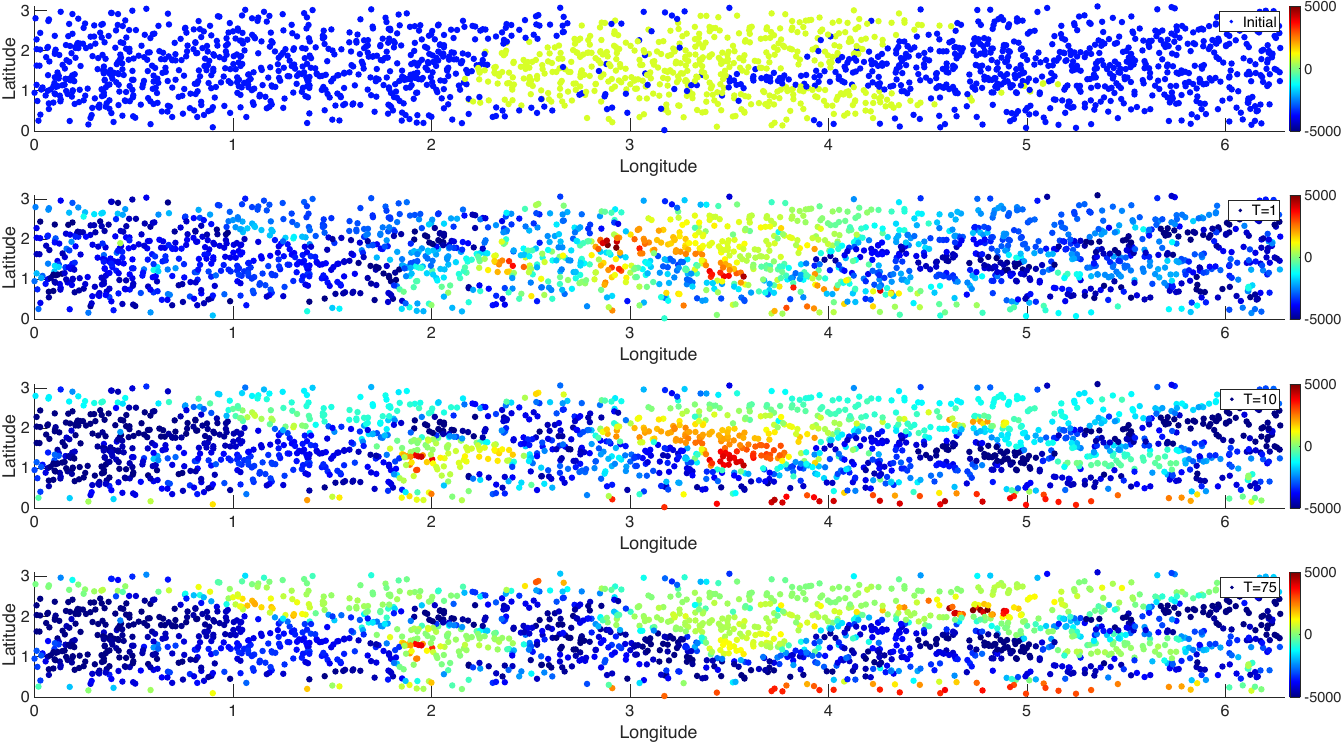}
  \caption{2D Snapshots for the density dynamics of altitude and depth of continents and oceans at $2000$ points with longitude and latitude  starting from Pangaea with parameters  { $\ud t = 0.05$, $T=0, 1, 10, 75$.} }\label{fig_earth2}
\end{figure}

\section{Discussion}
{We focus on the analysis of the dynamics of a physical system with a manifold structure. The underlying manifold structure of the system is reflected through a point cloud in a high dimensional space. By applying the diffusion map, we find the reaction coordinates so that those data points are reduced onto a manifold in a low dimensional space.  Based on the reaction coordinates, we propose an implementable, unconditionally stable, finite volume scheme for a Fokker-Planck equation which incorporates both the structure of the manifold in the low dimensional space and the equilibrium information. The finite volume scheme defines an approximated Markov process (random walk) on the point cloud with an approximated transition probability and jump rate. We also provide the weighted $L^2$ convergence analysis of the finite volume scheme to the Fokker-Planck equation on the manifold in the low dimensional space. The efficiency, unconditional stability, and accuracy of the data-driven solver proposed in this paper are justified theoretically.  Although we construct several numerical examples to illustrate our data-driven solver, there are still many interesting directions issued from practical problems for future works. An important direction is the manifold-related applications such as the optimal network partitions and the transition path in chemical reactions, especially on the high dimensional practical datasets.} 

\section*{Acknowledgement}

Nan Wu thanks the valuable discussion with Professor Hau-Tieng Wu and Chao Shen. Yuan Gao was supported by the National Science Foundation (NSF) under award DMS-2204288.   Jian-Guo Liu was supported in part by the National Science Foundation (NSF) under award DMS-2106988.

\bibliographystyle{plain}
\bibliography{bibmpmp}

\begin{thebibliography}{10}

\bibitem{wiki}
https://commons.wikimedia.org/wiki/file:pangaea\_continents.png.

\bibitem{RN179}
Christopher Amante and Barry~W Eakins.
\newblock {\em ETOPO1 1 arc-minute global relief model: procedures, data
  sources and analysis}.
\newblock US Department of Commerce, National Oceanic and Atmospheric
  Administration, National Environmental Satellite, Data, and Information
  Service, National Geophysical Data Center, Marine Geology and Geophysics
  Division, 2009.

\bibitem{bakry2014analysis}
Dominique Bakry, Ivan Gentil, Michel Ledoux, et~al.
\newblock {\em Analysis and geometry of Markov diffusion operators}, volume
  103.
\newblock Springer, 2014.

\bibitem{bates2014embedding}
Jonathan Bates.
\newblock The embedding dimension of laplacian eigenfunction maps.
\newblock {\em Applied and Computational Harmonic Analysis}, 37(3):516--530,
  2014.

\bibitem{beem1975pseudo}
John~K Beem.
\newblock Pseudo-{R}iemannian manifolds with totally geodesic bisectors.
\newblock {\em Proceedings of the American Mathematical Society},
  49(1):212--215, 1975.

\bibitem{Belkin_Niyogi:2003}
M.~Belkin and P.~Niyogi.
\newblock {Laplacian Eigenmaps for Dimensionality Reduction and Data
  Representation}.
\newblock {\em Neural. Comput.}, 15(6):1373--1396, 2003.

\bibitem{berard1994embedding}
Pierre B{\'e}rard, G{\'e}rard Besson, and Sylvain Gallot.
\newblock Embedding {R}iemannian manifolds by their heat kernel.
\newblock {\em Geometric \& Functional Analysis GAFA}, 4(4):373--398, 1994.

\bibitem{berry2015nonparametric1}
Tyrus Berry, Dimitrios Giannakis, and John Harlim.
\newblock Nonparametric forecasting of low-dimensional dynamical systems.
\newblock {\em Physical Review E}, 91(3):032915, 2015.

\bibitem{berry2015nonparametric2}
Tyrus Berry and John Harlim.
\newblock Nonparametric uncertainty quantification for stochastic gradient
  flows.
\newblock {\em SIAM/ASA Journal on Uncertainty Quantification}, 3(1):484--508,
  2015.

\bibitem{calder2019improved}
Jeff Calder and Nicolas~Garcia Trillos.
\newblock Improved spectral convergence rates for graph laplacians on
  $\varepsilon$-graphs and {k-NN} graphs.
\newblock {\em Applied and Computational Harmonic Analysis}, 60:123--175, 2022.

\bibitem{calder2020lipschitz}
Jeff Calder, Nicolas~Garcia Trillos, and Marta Lewicka.
\newblock Lipschitz regularity of graph {L}aplacians on random data clouds.
\newblock {\em SIAM Journal on Mathematical Analysis}, 54(1):1169--1222, 2022.

\bibitem{cheng2021eigen}
Xiuyuan Cheng and Nan Wu.
\newblock Eigen-convergence of gaussian kernelized graph {L}aplacian by
  manifold heat interpolation.
\newblock {\em Applied and Computational Harmonic Analysis}, 61:132--190, 2022.

\bibitem{chow2012fokker}
Shui-Nee Chow, Wen Huang, Yao Li, and Haomin Zhou.
\newblock Fokker-planck equations for a free energy functional or markov
  process on a graph.
\newblock {\em Archive for Rational Mechanics and Analysis}, 203(3):969--1008,
  2012.

\bibitem{CKLMN}
R.~R. Coifman, I.~G. Kevrekidis, S.~Lafon, M.~Maggioni, and B.~Nadler.
\newblock Diffusion maps, reduction coordinates, and low dimensional
  representation of stochastic systems.
\newblock {\em Multiscale Modeling \& Simulation}, 7(2):842--864, 2008.

\bibitem{coifman2006diffusion}
Ronald~R Coifman and St{\'e}phane Lafon.
\newblock Diffusion maps.
\newblock {\em Applied and Computational Harmonic Analysis}, 21(1):5--30, 2006.

\bibitem{Deuflhard_Weber_2005}
Peter Deuflhard and Marcus Weber.
\newblock Robust perron cluster analysis in conformation dynamics.
\newblock {\em Linear Algebra and its Applications}, 398:161–184, Mar 2005.

\bibitem{Donoho_Grimes:2003}
D.~L. Donoho and C.~Grimes.
\newblock {Hessian eigenmaps: Locally linear embedding techniques for
  high-dimensional data}.
\newblock {\em Proceedings of the National Academy of Sciences},
  100(10):5591--5596, 2003.

\bibitem{dunson2019spectral}
David~B. Dunson, Hau-Tieng Wu, and Nan Wu.
\newblock Spectral convergence of graph {L}aplacian and heat kernel
  reconstruction in ${L}^\infty$ from random samples.
\newblock {\em Applied and Computational Harmonic Analysis}, 55:282--336, 2021.

\bibitem{E_Li_Vanden-Eijnden_2008}
W.~E, T.~Li, and E.~Vanden-Eijnden.
\newblock Optimal partition and effective dynamics of complex networks.
\newblock {\em Proceedings of the National Academy of Sciences},
  105(23):7907–7912, Jun 2008.

\bibitem{Ebook}
Weinan E, Tiejun Li, and Eric Vanden-Eijnden.
\newblock {\em Applied stochastic analysis}.
\newblock Graduate studies in mathematics. American Mathematical Society, 2019.

\bibitem{weinan2006}
Weinan E and Eric Vanden-Eijnden.
\newblock Towards a theory of transition paths.
\newblock {\em J. Stat. Phys.}, 123(3):503, 2006.

\bibitem{E_Vanden-Eijnden_2010}
Weinan E and Eric Vanden-Eijnden.
\newblock Transition-path theory and path-finding algorithms for the study of
  rare events.
\newblock {\em Annual Review of Physical Chemistry}, 61(1):391–420, Mar 2010.

\bibitem{ekeland1999convex}
Ivar Ekeland and Roger Temam.
\newblock {\em Convex Analysis and Variational Problems}, volume~28.
\newblock SIAM, 1999.

\bibitem{Erbar_2014}
Matthias Erbar.
\newblock Gradient flows of the entropy for jump processes.
\newblock {\em Annales de l’Institut Henri Poincaré, Probabilités et
  Statistiques}, 50(3):920–945, Aug 2014.

\bibitem{esposito2019nonlocal}
Antonio Esposito, Francesco~S Patacchini, Andr{\'e} Schlichting, and Dejan
  Slep{\v{c}}ev.
\newblock Nonlocal-interaction equation on graphs: gradient flow structure and
  continuum limit.
\newblock {\em Archive for Rational Mechanics and Analysis}, 240(2):699--760,
  2021.

\bibitem{eymard2000finite}
Robert Eymard, Thierry Gallou{\"e}t, and Rapha{\`e}le Herbin.
\newblock Finite volume methods.
\newblock {\em Handbook of numerical analysis}, 7:713--1018, 2000.

\bibitem{gao2020inbetweening}
Yuan Gao, Guangzhen Jin, and Jian-Guo Liu.
\newblock Inbetweening auto-animation via fokker-planck dynamics and
  thresholding.
\newblock {\em Inverse Problems \& Imaging}, 15(5):843, 2021.

\bibitem{gao21}
Yuan Gao, Tiejun Li, Xiaoguang Li, and Jian-Guo Liu.
\newblock Transition path theory for langevin dynamics on manifold: optimal
  control and data-driven solver.
\newblock {\em to appear in Multiscale Modeling \& Simulation,
  arXiv:2010.09988}, 2022.

\bibitem{yg20}
Yuan Gao and Jian-Guo Liu.
\newblock A note on parametric bayesian inference via gradient flows.
\newblock {\em Annals of Mathematical Sciences and Applications}, 2:261--282,
  2020.

\bibitem{gao2021random}
Yuan Gao and Jian-Guo Liu.
\newblock Random walk approximation for irreversible drift-diffusion process on
  manifold: ergodicity, unconditional stability and convergence.
\newblock {\em arXiv preprint arXiv:2106.01344}, 2021.

\bibitem{GL22}
Yuan Gao and Jian-Guo Liu.
\newblock Revisit of macroscopic dynamics for some non-equilibrium chemical
  reactions from a hamiltonian viewpoint.
\newblock {\em Journal of Statistical Physics}, 189(2):1--57, 2022.

\bibitem{gao2022selection}
Yuan Gao and Jian-Guo Liu.
\newblock A selection principle for weak {KAM} solutions via freidlin-wentzell
  large deviation principle of invariant measures.
\newblock {\em arXiv preprint arXiv:2208.11860}, 2022.

\bibitem{GL22t}
Yuan Gao and Jian-Guo Liu.
\newblock Thermodynamic limit of chemical master equation via nonlinear
  semigroup.
\newblock {\em arXiv preprint arXiv:2205.09313}, 2022.

\bibitem{gilbarg2015elliptic}
David Gilbarg and Neil~S Trudinger.
\newblock {\em Elliptic partial differential equations of second order}, volume
  224.
\newblock springer, 2015.

\bibitem{hsu2002stochastic}
Elton~P Hsu.
\newblock {\em Stochastic analysis on manifolds}, volume~38.
\newblock American Mathematical Soc., 2002.

\bibitem{jones2008manifold}
Peter~W Jones, Mauro Maggioni, and Raanan Schul.
\newblock Manifold parametrizations by eigenfunctions of the {L}aplacian and
  heat kernels.
\newblock {\em Proceedings of the National Academy of Sciences},
  105(6):1803--1808, 2008.

\bibitem{Lafon06}
Stephane Lafon and Ann~B Lee.
\newblock Diffusion maps and coarse-graining: A unified framework for
  dimensionality reduction, graph partitioning, and data set parameterization.
\newblock {\em IEEE transactions on pattern analysis and machine intelligence},
  28(9):1393--1403, 2006.

\bibitem{lai2018point}
Rongjie Lai and Jianfeng Lu.
\newblock Point cloud discretization of fokker--planck operators for committor
  functions.
\newblock {\em Multiscale Modeling $\&$ Simulation}, 16(2):710--726, 2018.

\bibitem{li2018large}
Lei Li and Jian-Guo Liu.
\newblock Large time behaviors of upwind schemes by jump processes.
\newblock {\em Math. Comp.}, 89:2283--2320, 2020.

\bibitem{li2009probabilistic}
Tiejun Li, Jian Liu, and E~Weinan.
\newblock Probabilistic framework for network partition.
\newblock {\em Physical Review E}, 80(2):026106, 2009.

\bibitem{liu2019rate}
Anning Liu, Jian-Guo Liu, and Yulong Lu.
\newblock On the rate of convergence of empirical measure in $\8$-wasserstein
  distance for unbounded density function.
\newblock {\em Quarterly of Applied Mathematics}, 77(4):811--829, 2019.

\bibitem{Maas_2011}
Jan Maas.
\newblock Gradient flows of the entropy for finite markov chains.
\newblock {\em Journal of Functional Analysis}, 261(8):2250–2292, Oct 2011.

\bibitem{MMS2009}
Philipp Metzner, Christof Schütte, and Eric Vanden-Eijnden.
\newblock Transition path theory for markov jump processes.
\newblock {\em Multiscale Modeling \& Simulation}, 7(3):1192–1219, Jan 2009.

\bibitem{Mielke_Renger_Peletier_2014}
Alexander Mielke, D.~R.~Michiel Renger, and Mark~A. Peletier.
\newblock On the relation between gradient flows and the large-deviation
  principle, with applications to markov chains and diffusion.
\newblock {\em Potential Analysis}, 41(4):1293–1327, Nov 2014.

\bibitem{nadler2006diffusion}
Boaz Nadler, St{\'e}phane Lafon, Ronald~R Coifman, and Ioannis~G Kevrekidis.
\newblock Diffusion maps, spectral clustering and reaction coordinates of
  dynamical systems.
\newblock {\em Applied and Computational Harmonic Analysis}, 21(1):113--127,
  2006.

\bibitem{portegies2016embeddings}
Jacobus~W Portegies.
\newblock Embeddings of {R}iemannian manifolds with heat kernels and
  eigenfunctions.
\newblock {\em Communications on Pure and Applied Mathematics}, 69(3):478--518,
  2016.

\bibitem{Noe11}
Jan-Hendrik Prinz, Hao Wu, Marco Sarich, Bettina Keller, Martin Senne, Martin
  Held, John~D. Chodera, Christof Schütte, and Frank Noé.
\newblock Markov models of molecular kinetics: Generation and validation.
\newblock {\em The Journal of Chemical Physics}, 134(17):174105, May 2011.

\bibitem{Mauro2011}
Mary~A. Rohrdanz, Wenwei Zheng, Mauro Maggioni, and Cecilia Clementi.
\newblock Determination of reaction coordinates via locally scaled diffusion
  map.
\newblock {\em The Journal of Chemical Physics}, 134(12):124116, 2011.

\bibitem{Roweis_Saul:2000}
S.~T. Roweis and L.~K. Saul.
\newblock Nonlinear dimensionality reduction by locally linear embedding.
\newblock {\em Science}, 290(5500):2323--2326, 2000.

\bibitem{NoeLu11}
Christof Schütte, Frank Noé, Jianfeng Lu, Marco Sarich, and Eric
  Vanden-Eijnden.
\newblock Markov state models based on milestoning.
\newblock {\em The Journal of Chemical Physics}, 134(20):204105, May 2011.

\bibitem{singer2012vector}
Amit Singer and H-T Wu.
\newblock Vector diffusion maps and the connection {L}aplacian.
\newblock {\em Communications on pure and applied mathematics},
  65(8):1067--1144, 2012.

\bibitem{singer2016spectral}
Amit Singer and Hau-Tieng Wu.
\newblock Spectral convergence of the connection laplacian from random samples.
\newblock {\em Information and Inference: A Journal of the IMA}, 6(1):58--123,
  2016.

\bibitem{Tenenbaum_deSilva_Langford:2000}
J.~B. Tenenbaum, V.~{de Silva}, and J.~C. Langford.
\newblock {A Global Geometric Framework for Nonlinear Dimensionality
  Reduction}.
\newblock {\em Science}, 290(5500):2319--2323, 2000.

\bibitem{trillos2020error}
Nicol{\'a}s~Garc{\'\i}a Trillos, Moritz Gerlach, Matthias Hein, and Dejan
  Slep{\v{c}}ev.
\newblock Error estimates for spectral convergence of the graph {L}aplacian on
  random geometric graphs toward the laplace--beltrami operator.
\newblock {\em Foundations of Computational Mathematics}, 20(4):827--887, 2020.

\bibitem{Dejan15}
Nicol{\'a}s~Garcia Trillos and Dejan Slep{\v{c}}ev.
\newblock On the rate of convergence of empirical measures in
  $\8$-transportation distance.
\newblock {\em Canadian Journal of Mathematics}, 67(6):1358--1383, 2015.

\bibitem{wu2018think}
Hau-Tieng Wu and Nan Wu.
\newblock Think globally, fit locally under the manifold setup: Asymptotic
  analysis of locally linear embedding.
\newblock {\em The Annals of Statistics}, 46(6B):3805--3837, 2018.

\bibitem{yuan2020continuum}
Amber Yuan, Jeff Calder, and Braxton Osting.
\newblock A continuum limit for the pagerank algorithm.
\newblock {\em European Journal of Applied Mathematics}, 33(3):472--504, 2022.

\end{thebibliography}


\appendix
\section{Theorems about embedding by eigenfunctions of Laplacian}\label{appendix A}
 Let $\Delta$ be the Laplace-Beltrami operator of a closed smooth Riemannian manifold $\mm$.  Let $\{\lambda_i\}_{i=0}^\infty$ be the eigenvalues of $-\Delta$, and 
\begin{align}
\Delta \psi_i =-\lambda_i \psi_i,
\end{align}
where $\psi_i$ is the corresponding eigenfunction normalized in $L^2(\mm)$. We have $0=\lambda_0 < \lambda_1\leq \lambda_2 \leq  \cdots$.

In this section, we review the theorems about embedding the manifold $\mm$ by using the eigenfunctions of $\Delta$.  In \cite{berard1994embedding}, the authors provide a theorem about spectral embedding by using all the eigenvalues and eigenfunctions of $\Delta$ into the Hilbert space $\ell^2$. 
\begin{thm}(B\'erard-Besson-Gallot, \cite{berard1994embedding})
{ Let $M$ be a $d$ dimensional smooth closed Riemannian manifold. Then, for $\mx \in \mm$ 
\begin{align}
\Psi(\mx)=(2t)^{\frac{d+2}{4}}\sqrt{2}(4\pi)^{\frac{d}{4}}(e^{-\lambda_1t}\psi_1(\mx), \cdots, e^{-\lambda_qt}\psi_q(\mx), \cdots), 
\end{align}
is an embedding of $M$ into $\ell^2$ for all $t>0$.}
\end{thm}

\cite{jones2008manifold} improves the above result locally. They show that one can use finite eigenfunctions of Laplace-Beltrami operator to embed the manifold locally. The result can be briefly summarized as follows. 
\begin{thm}(Jones-Maggioni-Schul, \cite{jones2008manifold})
 Let $\mm$ be a $d$ dimensional smooth closed Riemannian manifold, for each $\mx \in M$, there are $j_1 \leq \cdots \leq j_d$ and the constants $C_1, \cdots, C_d$ such that 
\begin{align}
\Psi(\mx)=(C_1\psi_{j_1}(\mx), \cdots, C_d\psi_{j_d}(\mx)), 
\end{align}
is locally a bi-Lipschitz chart.
\end{thm}

Moreover, the next theorem \cite{portegies2016embeddings} says that we can use the eigenvalues and eigenfunctions of the Laplace-Beltrami operator to construct an almost isometric embedding of the manifold into some Euclidean space.

\begin{thm}\label{almost isometric embedding}(Portegies, \cite{portegies2016embeddings})
{Let $\mm$ be a $d$ dimensional smooth closed Riemannian manifold.  Suppose $\ric_{\mm}  \geq (d-1)k$,  the injectivity radius of $\mm$, $i(\mm) \geq i_0$ and the volume of $\mm$, $\vol(\mm) \leq V$. For any $\epsilon>0$, there is a $\mathbf{t}_0(\epsilon, d , k, i_0)$ such that for $t<\mathbf{t}_0$, there is $C(\epsilon, d ,k, i_0, V, t )$, if $q>C$, then for $\mx \in \mm$
\begin{align}
\Psi(\mx)=(2t)^{\frac{d+2}{4}}\sqrt{2}(4\pi)^{\frac{d}{4}}(e^{-\lambda_1t}\psi_1(\mx), \cdots, e^{-\lambda_qt}\psi_q(\mx)), 
\end{align}
is an embedding of $\mm$ into $\mathbb{R}^q$ such that $1-\epsilon<\|\nabla \Psi\|_{op}<1+\epsilon$. Here $\| \cdot \|_{op}$ is the operator norm.}
\end{thm}

{Based on Theorem \ref{Bates embedding},  the smallest $q$ that 
\begin{align}
\Psi_1(\mx)=(\psi_1(\mx), \cdots,\psi_q(\mx)), 
\end{align}
is a smooth embedding of $\mm$ is called the embedding dimension of $\mm$. Based on Theorem \ref{almost isometric embedding}, the smallest $q$ that 
\begin{align}
\Psi_2(\mx)=(2t)^{\frac{d+2}{4}}\sqrt{2}(4\pi)^{\frac{d}{4}}(e^{-\lambda_1t}\psi_1(\mx), \cdots, e^{-\lambda_qt}\psi_q(\mx)), 
\end{align}
is an almost isometric embedding of $\mm$ is called the almost isometric embedding dimension of $\mm$.} We expect the embedding dimension is much smalled than the almost isometric embedding dimension. Hence, for the dimension reduction purpose, we are looking for an embedding of the manifold rather than an  almost isometric embedding. 

{\section{Proof of Proposition \ref{perpendicular to Voronoi face}}\label{proof local regularity of VFace}

Since $\delta$ is less than the injectivity radius, there is a Euclidean ball $B^{T_{y_i}\nn}_{\delta}(0)$ of radius $\delta$ in the tangent space $T_{y_i}\nn$ of $\nn$ at $\my_i$ such that the exponential map $\exp_{y_i}: B^{T_{y_i}\nn}_{\delta}(0) \rightarrow B_{\delta}(\my_i)$ is a diffeomorphism. Suppose $y_j=\exp_{y_i}(w)$.  We illustrate this setup in Figure \ref{bisector proof figure}. It is sufficient to prove that $\exp_{y_i}^{-1}(B_{\delta}(\my_i) \cap G_{ij})$ is a $d-1$ dimensional submanifold of $T_{y_i}\nn$. For any $v \in \exp_{y_i}^{-1}(B_{\delta}(\my_i) \cap G_{ij})$, by the definition of the bisector, we have
\begin{align}
d^2_{\nn}(\exp_{y_i}(v),\exp_{y_i}(w)) =d^2_{\nn}(\exp_{y_i}(v), y_j)=d^2_{\nn}(\exp_{y_i}(v), y_i)=|v|^2. \label{per vor face 1}
\end{align}
Note that 
\begin{align}
d^2_{\nn}(\exp_{y_i}(v),\exp_{y_i}(w)) =|v-w|^2+f(v,w). \label{per vor face 2}
\end{align}
$f(v,w)$ is a smooth function on $B^{T_{y_i}\nn}_{\delta}(0) \times B^{T_{y_i}\nn}_{\delta}(0)$. In particular, 
\begin{align}
f(v,w)=-\frac{1}{3} R_{y_i}(v,w,v,w)+O((|v|^2+|w|^2)^{\frac{5}{2}})\label{per vor face 3}
\end{align}
for $|v|$ and $|w|$ small, where $R_{y_i}$ is the curvature tensor at $y_i$. Combine \eqref{per vor face 1} and  \eqref{per vor face 2}, we have that 
\begin{align}
|w|^2-2v \cdot w+f(v,w)=0.\label{per vor face 4}
\end{align}

\begin{figure}
\centering
\includegraphics[scale=0.24]{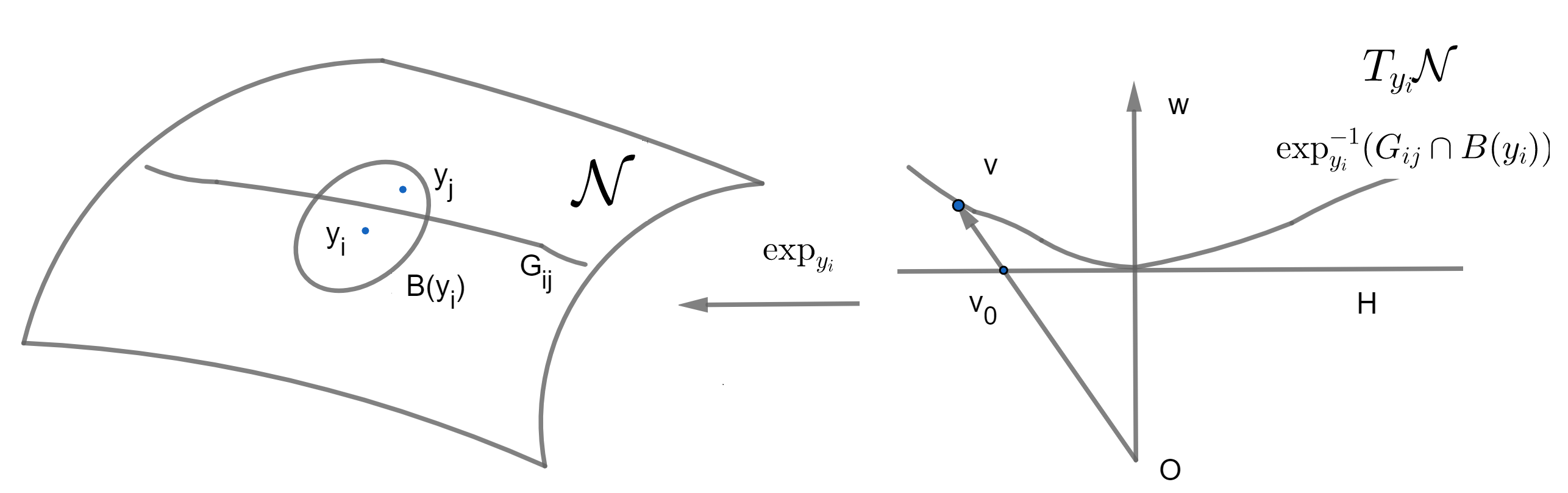}
\caption{An illustration to the proof of Proposition \ref{perpendicular to Voronoi face}}\label{bisector proof figure}
\end{figure}

As $y_j$ is a fixed point,  we treat $w$ as a fixed vector. Therefore, we use the notation $f_w(x)=f(x,w)$ to indicate $f_w$ is a function of $x \in  B^{T_{y_i}\nn}_{\delta}(0)$. Then, we can define a smooth function $F(t,x)$ on $\mathbb{R} \times B^{T_{y_i}\nn}_{\delta}(0)$ as
\begin{align}
F(t,x)=|w|^2-2(1+t)x \cdot w +f_w((1+t)x).
\end{align}
Let $H$ be the $d-1$ dimensional hyperplane that perpendicularly bisects $w$ in $T_{y_i}\nn$. Let $v_0$ be the vector on $H$ so that $v=(1+t_0)v_0$.  By \eqref{per vor face 4}, $F(t_0, v_0)=0$. If we can show that $\frac{\partial F(t_0, v_0)}{\partial t} \not= 0$, then by the Implicit Function Theorem, there is a ball $B$ centered at $v_0$ so that $t=g(x)$ for $x \in B$ and $g$ is differentiable. Hence, $(1+g(x))x$ for $x \in B \cap H$ is a chart for $\exp_{y_i}^{-1}(B_{\delta}(\my_i) \cap G_{ij})$ around $v$. 
We calculate $\frac{\partial F(t_0, v_0)}{\partial t}$:
\begin{align}
\frac{\partial F(t_0, v_0)}{\partial t}=-2v_0 \cdot w+ \nabla f_w(v) \cdot v_0 .
\end{align}
By \eqref{per vor face 3}, $|\nabla f_w(v)|=O(|v||w|^2)$ where the constant depends on the sectional curvatures at $\my_i$. Since the manifold is compact, the sectional curvatures have upper and lower bounds.
Hence, 
\begin{align}
\frac{\partial F(t_0, v_0)}{\partial t}=-2v_0 \cdot w+ \nabla f_w(v) \cdot v_0 <-2|v_0| |w| \cos(\theta)+|v_0|O(|v||w|^2), 
\end{align}
where $\theta$ is the angle between $v_0$ and $w$. Since $v_0 \in B^{T_{y_i}\nn}_{\delta}(0) \cap H$, $\cos(\theta)>\frac{|w|}{2\delta}$. $v \in B^{T_{y_i}\nn}_{\delta}(0)$, so $|v|<\delta$. Therefore, when $\delta$ is small enough,
\begin{align}
\frac{\partial F(t_0, v_0)}{\partial t} <|v_0| |w|^2 (-\frac{1}{\delta}+O(\delta))<0. 
\end{align}

Next, we prove the second part of the proposition. $\my^* \in M_{ij}$ follows from the construction. Note that by the triangle inequality and the definition of the bisector, the geodesic sphere centered at $\my_i$ through $\my^*$ is tangent to $M_{ij}$ at $\my^*$. Hence, by Gauss's Lemma, the minimizing geodesic between $y_i$ and $y_j$ is perpendicular to $M_{ij}$ at $\my^*$.
}

\section{Proof of Proposition \ref{approximation of the volume of a Voronoi cells} and Proposition \ref{approximation of the area of a Voronoi face} }\label{appendix B}
We start from a study of the matrix $C_{n,r}(\my_k)$ in  Definition \ref{definition of the map iota} and relate it to its continuous form. 
{Consider the local covariance matrix  $C_{\my_k,B^{\mathbb{R}^\ell}_{\sqrt{r}}(\my_k) \cap \nn}$ defined as follows.
\begin{align} 
C_{\my_k,B^{\mathbb{R}^\ell}_{\sqrt{r}}(\my_k) \cap \nn}= & \int_{B^{\mathbb{R}^\ell}_{\sqrt{r}}(\my_k) \cap \nn}(\my-\my_k)(\my-\my_k)^{\top} \rho^{**} (\my) dV_{\nn}(\my) \in\mathbb{R}^{\ell \times \ell}. 
\end{align}
}
 Suppose $C_{\my_k,B^{\mathbb{R}^\ell}_{\sqrt{r}}(\my_k) \cap \nn}$ has the following eigendecomposition:
\begin{align}
C_{\my_k,B^{\mathbb{R}^\ell}_{\sqrt{r}}(\my_k) \cap \nn}=U(\my_k) \Lambda(\my_k) U(\my_k) ^\top  \in O(\ell),
\end{align}
where $\Lambda(\my_k)$ is a diagonal matrix with the diagonal entries to be eigenvalues of $C_{\my_k,B^{\mathbb{R}^\ell}_{\sqrt{r}}(\my_k) \cap \nn}$. Moreover, we have $\Lambda_{11}(\my_k) \geq \Lambda_{22}(\my_k) \geq \cdots \geq \Lambda_{\ell \ell}(\my_k)$. $U(\my_k) \in O(\ell)$ consists of the corresponding orthonormal eigenvectors of $C_{\my_k,B^{\mathbb{R}^\ell}_{\sqrt{r}}(\my_k) \cap \nn}$. Intuitively, $C_{\my_k,B^{\mathbb{R}^\ell}_{\sqrt{r}}(\my_k) \cap \nn}$ is the continuous form of the matrix $C_{n,r}(\my_k)$.

By setting $\epsilon=\sqrt{r}$ in Proposition 3.2 in \cite{wu2018think}, we have the following lemma.
\begin{lem}\label{A Lemma1}
Assume that $T_{\my_k} \nn$ is generated by the first $d$ standard basis of $\mathbb{R}^\ell$.
\begin{align}
\Lambda(\my_k)&=\frac{|S^{d-1}| P(\my_k) r^{\frac{d+2}{2}}}{d(d+2)} \Big( \begin{bmatrix}
I_{d\times d} & 0 \\
0 & 0 \\
\end{bmatrix}+O(r) \Big), \\
U(\my_k)&= \begin{bmatrix}
X_1 & 0 \\
0 & X_2\\
\end{bmatrix}+O(r),
\end{align}
where $X_1 \in O(d)$ and $X_2 \in O(\ell-d)$.
\end{lem}
Above lemma says that the first $d$ eigenvectors of  $C_{\my_k,B^{\mathbb{R}^\ell}_{\sqrt{r}}(\my_k) \cap \nn}$  form an orthonormal basis of $T_{\my_k} \nn$ up to an error of order $O(r)$. Note that, for simplicity, we assume $T_{\my_k} \nn$ is generated by the first $d$ standard basis of $\mathbb{R}^\ell$ so that $U(\my_k)$ can be expressed in the above block form.
Suppose $C_{n,r}(\my_k)$ has the following eigendecomposition:
\begin{align}
C_{n,r}(\my_k)=U_n(\my_k) \Lambda_n(\my_k) U_n(\my_k) ^\top.
\end{align}
$ \Lambda_n(\my_k)$ is a diagonal matrix with the diagonal entries to be eigenvalues of $C_{n,r}(\my_k)$. Moreover, we have $\Lambda_{n,11}(\my_k) \geq \Lambda_{n,22}(\my_k) \geq \cdots \geq \Lambda_{n, \ell \ell}(\my_k)$. $U_n(\my_k) \in O(\ell)$ consists of the corresponding orthonormal eigenvectors of $C_{n,r}(\my_k)$.

The relation between the eigenstructure of $C_{\my_k,B^{\mathbb{R}^\ell}_{\sqrt{r}}(\my_k) \cap \nn}$ and $C_{n,r}(\my_k)$ is discussed in  Lemma E.4 in \cite{wu2018think}. 

\begin{lem}\label{A Lemma2}
Assume that $T_{\my_k} \nn$ is generated by the first $d$ standard basis of $\mathbb{R}^\ell$. When $n$ is large enough, with probability greater than $1-\frac{1}{n^2}$, for all $\my_k$,
\begin{align}
\Lambda_n(\my_k) & =\Lambda(\my_k)+O(\sqrt{\frac{\log n}{n r^{-\frac{d}{2}-2}}}), \\
U_n(\my_k)&= \begin{bmatrix}
X'_1 & 0 \\
0 & X'_2\\
\end{bmatrix}U(\my_k)+O(\sqrt{\frac{\log n}{n r^{\frac{d}{2}-2}}}),
\end{align}
where $X'_1 \in O(d)$ and $X'_2 \in O(\ell-d)$.
\end{lem}

\begin{rem}
Above lemma follows from Lemma E.4 in \cite{wu2018think} if we choose $\epsilon=\sqrt{r}$ and $\rho \rightarrow \infty$ in Case 0 of Lemma E.4 in \cite{wu2018think}. In fact, Case 0 of Lemma E.4 in \cite{wu2018think} focuses on the first $d$ eigenpairs of the matrix $C_{\my_k,B^{\mathbb{R}^\ell}_{\sqrt{r}}(\my_k) \cap \nn}$ of which we need to recover.
\end{rem}

If we combine Lemma \ref{A Lemma1} and Lemma \ref{A Lemma2}, we have 
\begin{align}
\Lambda_n(\my_k) & =\frac{|S^{d-1}| P(\my_k) r^{\frac{d+2}{2}}}{d(d+2)} \begin{bmatrix}
I_{d\times d} & 0 \\
0 & 0 \\
\end{bmatrix}+O(r^{\frac{d}{2}+2})+O(\sqrt{\frac{\log n}{n r^{-\frac{d}{2}-2}}}), \\
U_n(\my_k)&= \begin{bmatrix}
U_1 & 0 \\
0 & U_2\\
\end{bmatrix}+O(r)+O(\sqrt{\frac{\log n}{n r^{\frac{d}{2}-2}}}),
\end{align}
where $U_1 \in O(d)$ and $U_2 \in O(\ell-d)$. If $\frac{nr^{\frac{d}{2}}}{\log n} \rightarrow \infty$ as $n \rightarrow \infty$, then $\sqrt{\frac{\log n}{n r^{\frac{d}{2}-2}}} \leq r$. If $\frac{nr^{\frac{d}{2}+2}}{\log n} \rightarrow \infty$ as $n \rightarrow \infty$, then $\sqrt{\frac{\log n}{n r^{\frac{d}{2}-2}}} \leq r$ and $\sqrt{\frac{\log n}{n r^{-\frac{d}{2}-2}}} \leq r^{\frac{d}{2}+2}$. Hence, we have the following proposition.

\begin{prop}\label{structure of discrete PCA}
Assume that $T_{\my_k} \nn$ is generated by the first $d$ standard basis of $\mathbb{R}^\ell$. If $\frac{nr^{\frac{d}{2}}}{\log n} \rightarrow \infty$ as $n \rightarrow \infty$, then  with probability greater than $1-\frac{1}{n^2}$, for all $\my_k$,
\begin{align}
U_n(\my_k)&= \begin{bmatrix}
U_1 & 0 \\
0 & U_2\\
\end{bmatrix}+O(r),
\end{align}
where $U_1 \in O(d)$ and $U_2 \in O(\ell-d)$.

If $\frac{nr^{\frac{d}{2}+2}}{\log n} \rightarrow \infty$ as $n \rightarrow \infty$, then  with probability greater than $1-\frac{1}{n^2}$, for all $\my_k$,
\begin{align}
\Lambda_n(\my_k) & =\frac{|S^{d-1}| P(\my_k) r^{\frac{d+2}{2}}}{d(d+2)} \begin{bmatrix}
I_{d\times d} & 0 \\
0 & 0 \\
\end{bmatrix}+O(r^{\frac{d}{2}+2}), \\
U_n(\my_k)&= \begin{bmatrix}
U_1 & 0 \\
0 & U_2\\
\end{bmatrix}+O(r),
\end{align}
where $U_1 \in O(d)$ and $U_2 \in O(\ell-d)$.
\end{prop}

Above proposition should be understood in the following way. If  $n$ and $r$ satisfy $\frac{nr^{\frac{d}{2}}}{\log n} \rightarrow \infty$ as $n \rightarrow \infty,$ then we have an approximation of the tangent space of $\nn$ at $\my_k$, i.e. the first $d$ eigenvectors of $C_{n,r}(\my_k)$ are the basis of $T_{\my_k} \nn$ up to an error of order $O(r)$.  If $n$ and $r$ satisfy $\frac{nr^{\frac{d}{2}+2}}{\log n} \rightarrow \infty$ as $n \rightarrow \infty$, the first $d$ eigenvectors of $C_{n,r}(\my_k)$ are the basis of $T_{\my_k} \nn$ up to an error of order $O(r)$. Moreover, there are $d$ significantly large eigenvalues of $C_{n,r}(\my_k)$ which are  close to the first $d$ eigenvalues of  $C_{\my_k,B^{\mathbb{R}^\ell}_{\sqrt{r}}(\my_k) \cap \nn}$.

Next, we show that the map $\tilde{\iota}_k$ in the definition \ref{definition of the map iota} restricted on $ B^{\mathbb{R}^\ell}_{r}(\my_k) \cap \nn$ is a $1+O(r)$ bi-Lipschitz homeomorphism.
\begin{lem}\label{bi lip of iota}
Suppose $r \rightarrow 0$ and $\frac{nr^{\frac{d}{2}}}{\log n} \rightarrow \infty$ as $n \rightarrow \infty$. Suppose $r$ is small enough, then  with probability greater than $1-\frac{1}{n^2}$, for all $\my_k$ and any $\my, \my' \in B^{\mathbb{R}^\ell}_{r}(\my_k) \cap \nn$, we have 
\begin{align}
\|\tilde{\iota}_k(\my')-\tilde{\iota}_k(\my)\|_{\mathbb{R}^d}=\|\iota_k(\my'-\my)\|_{\mathbb{R}^d}=d_\nn(\my,\my')(1+O(r)).
\end{align}
\end{lem}
\begin{proof}
$\|\tilde{\iota}_k(\my')-\tilde{\iota}_k(\my)\|_{\mathbb{R}^d}=\|\iota_k(\my'-\my)\|_{\mathbb{R}^d}$ follows from the definition. Next, we prove $\|\iota_k(\my'-\my)\|_{\mathbb{R}^d}=d_\nn(\my,\my')(1+O(r))$. For simplicity, we assume $\my_k=0$ and $T_{\my_k} \nn$ is generated by the first $d$ standard basis of $\mathbb{R}^\ell$.
For any $\my \in \mathbb{R}^\ell$, we use the following notation to simplify the proof:
\begin{align} \label{vectornotation}
\my =[\![v ,\,v^\bot ]\!]\in \mathbb{R}^\ell\,,
\end{align}
where $v \in T_{\my_k} \nn$ forms the first $d$ coordinates of $y$ and $v^\bot \in  T^\bot_{\my_k} \nn$ forms the last $\ell-d$ coordinates of $\my$. For any $\my, \my' \in B^{\mathbb{R}^\ell}_{r}(\my_k) \cap \nn$, suppose $\my =[\![v_1 ,\,v_1^\bot ]\!]$ and $\my'=[\![v_2,\,v_2^\bot ]\!]$. Due to the manifold structure of $\nn$, we have
\begin{align}
\|v_1^\bot-v_2^\bot\|_{\mathbb{R}^{\ell-d}} \leq \mathcal{C}_1 r \|v_1-v_2\|_{\mathbb{R}^d},
\end{align}
for some constant $\mathcal{C}_1$ depending on the curvature of $\nn$. Hence,
\begin{align}
\|v_1-v_2\|_{\mathbb{R}^d} \leq \|\my'-\my\|_{\mathbb{R}^\ell}  \leq  \|v_1-v_2\|_{\mathbb{R}^d} \sqrt{1+\mathcal{C}^2_1 r^2},
\end{align}
which is equivalent to 
\begin{align}\label{length vs tangent length}
\|v_1-v_2\|_{\mathbb{R}^d} =\|\my'-\my\|_{\mathbb{R}^\ell} (1+O(r^2)).
\end{align}
Moreover, suppose $\{\beta_{n,r,1}, \cdots, \beta_{n,r,d}\}$ are orthonormal eigenvectors corresponding to $C_{n,r}(\my_k)$'s largest $d$ eigenvalues. Then, by Proposition \ref{structure of discrete PCA}
\begin{align}
\beta_{n,r,1}=[\![\beta_i , 0]\!]+O(r),
\end{align}
where $\{\beta_i\}_{i=1}^d$ form an orthonormal basis of $T_{\my_k} \nn \approx \mathbb{R}^d$.
\begin{align}
\|\iota_k(\my'-\my)\|_{\mathbb{R}^d}=& \|v_1-v_2\|_{\mathbb{R}^d}+\|\my'-\my\|_{\mathbb{R}^\ell} O(r) \\
=& \|\my'-\my\|_{\mathbb{R}^\ell} (1+O(r^2))+\|\my'-\my\|_{\mathbb{R}^\ell} O(r) = \|\my'-\my\|_{\mathbb{R}^\ell}  (1+O(r)) , \nonumber
\end{align}
where we apply \eqref{length vs tangent length} in the second last step.

By equation \eqref{euclidean vs geodesic}, we know that $d_\nn(y,\my') \leq 2D_1 r$. Hence, by Lemma \ref{geodesic vs euclidean},
\begin{align}
\|\iota_k(\my'-\my)\|_{\mathbb{R}^d}= & \|\my'-\my\|_{\mathbb{R}^\ell}  (1+O(r)) = d_\nn(\my,\my')(1+O(d^2_\nn(\my,\my'))) (1+O(r)) \\ 
=& d_\nn(\my,\my')(1+O(r^2))  (1+O(r))= d_\nn(\my,\my') (1+O(r)). \nonumber
\end{align}
\end{proof}

We introduce the following notations to prove the following lemma and proposition. Denote the boundary of $C_k$ by $\partial C_k$. Denote the boundary of $\tilde{\iota}_k(C_k)$ by $\partial \tilde{\iota}_k(C_k)=\tilde{\iota}_k(\partial C_k)$. Let $\tilde{C}_{k,0}$ be the Voronoi cell in $\mathbb{R}^d$ containing $0$ constructed in the Step 4 in Algorithm  \ref{Voronoi cell on manifold algorithm}. Denote the boundary of $\tilde{C}_{k,0}$ by $\partial \tilde{C}_k$. 
Denote $d^{\mathbb{R}^{d}}_{\mathcal{H}}(S_1, S_2)$ be the Hausdorff distance between two sets $S_1$ and $S_2$ in $\mathbb{R}^d$ with respect to the Euclidean metric.

\begin{lem} \label{boundary comparison lemma}
{If $n$ is large enough, for $r$ satisfying Assumption \ref{assumption on voronoi cell},}  with probability greater than $1-\frac{1}{n^2}$, for all $\my_k$  , $d^{\mathbb{R}^{d}}_{\mathcal{H}}(\partial \tilde{\iota}_k(C_k), \partial \tilde{C}_k)=O(r^2)$. 
\end{lem}

\begin{proof}  For simplicity, in this proof, we use $|\cdot|$ to denote $\| \cdot \|_{\mathbb{R}^d}$. 
{Based on Assumption \ref{assumption on voronoi cell}, the requirement of  Lemma \ref{bi lip of iota} holds.} Recall that in Assumption \ref{assumption on voronoi cell}, we assume $B^{\mathbb{R}^\ell}_{r}(\my_k) \cap \{\my_i\}_{i=1}^n=\{\my_{k,1}, \cdots, \my_{k, N_k}\}$.  We have $C_k \subset B^{\mathbb{R}^\ell}_{r}(\my_k)$. Moreover, if $\Gamma_{kj}$ is  a Voronoi surface of $C_k$ between $\my_k$ and $\my_j$, then $\my_j \in B^{\mathbb{R}^\ell}_{r}(\my_k)$. We denote $\Gamma_{k,i}$ to be the Voronoi face between $\my_k$ and $\my_{k,i}$.

The proof has two steps, first we show that for any $v \in \partial \tilde{\iota}_k(C_k)$, $d_{\mathbb{R}^d}(v, \partial \tilde{C}_k)=O(r^2)$. We need to consider two cases in this step.

\textbf{Case 1: $v \in \tilde{C}_{k,0}$}

\begin{figure}[h!]
\centering
\subfigure[The Case when $v=A \in \tilde{C}_{k,0}$]{
\includegraphics[width=0.45\columnwidth]{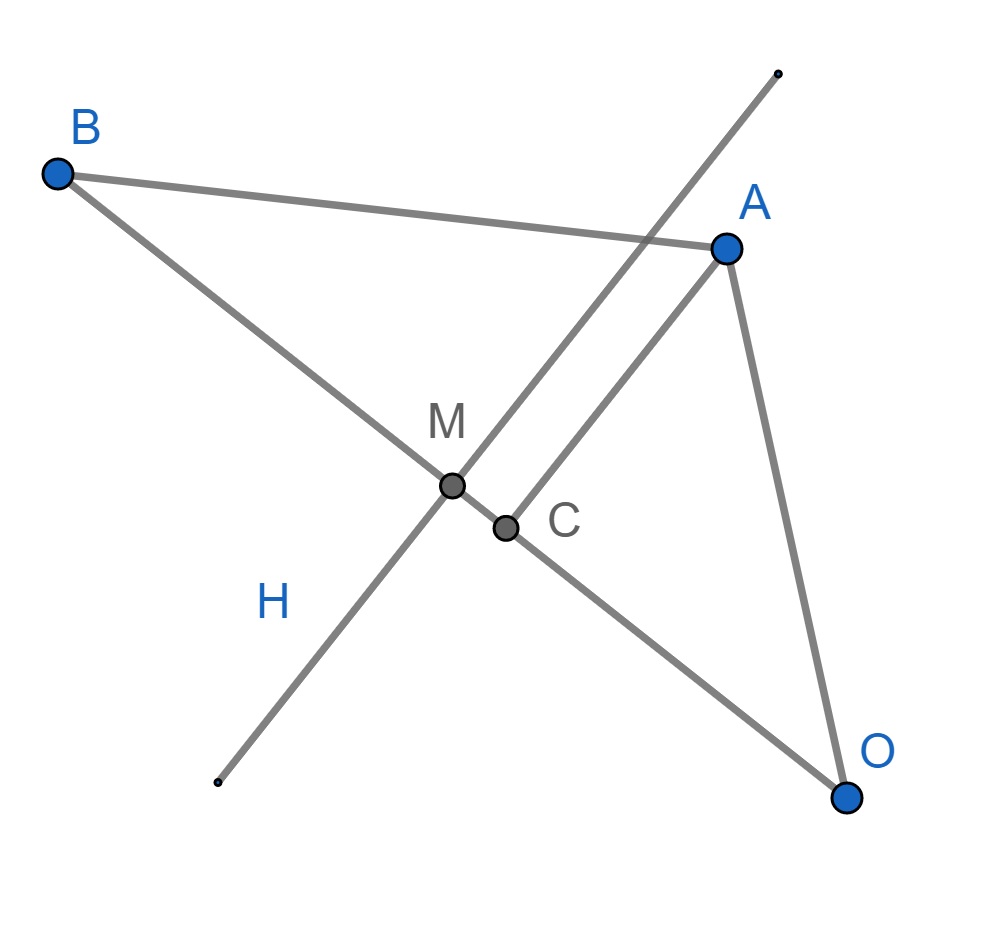}\label{Lemma9Case1}
}
\subfigure[The Case when $v=A  \not\in \tilde{C}_{k,0}$]{
\includegraphics[width=0.45\columnwidth]{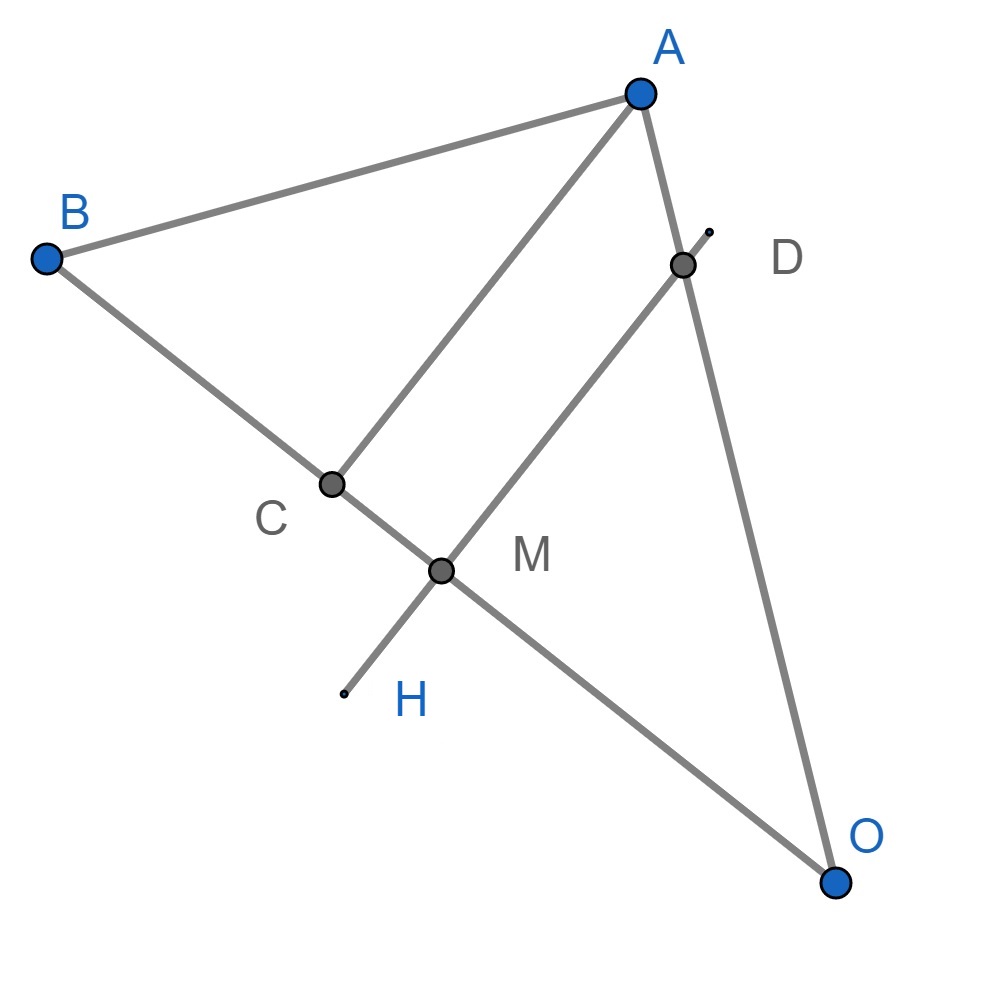}\label{Lemma9Case2}
}
\caption{Illustrations of Case 1 and Case 2 in the proof of Lemma \ref{boundary comparison lemma}.}
\end{figure}

Suppose $v=\tilde{\iota}_k(\my)$ for some $\my \in \partial C_k$. Moreover $\my \in \Gamma_{k,i}$.  In other word, $\my$ is on the Voronoi face between $\my_k$ and $\my_{k,i} \in B^{\mathbb{R}^\ell}_{r}(\my_k)$. As shown in figure \ref{Lemma9Case1}, let $O$ be the origin in $\mathbb{R}^d$. Let $A=v$ and $B=\tilde{\iota}_k(\my_{k,i})$. 
Let $H$ be the hyperplane in $\mathbb{R}^d$ which perpendicularly bisects $OB$. $M$ is the intersection of $H$ and $OB$. Let $C$ be the point on $OB$ so that $AC$ is perpendicular to $OB$. Since we assume $A \in  \tilde{C}_{k,0}$, $C \in OM$. We have
\begin{align}
d_{\mathbb{R}^d}(A,H)=& |CM|=\frac{|CB|-|CO|}{2}=\frac{|CB|^2-|CO|^2}{2(|CB|+|CO|)} \\
=& \frac{|AB|^2-|AC|^2-(|AO|^2-|AC|^2)}{2|BO|}=\frac{(|AB|+|AO|)(|AB|-|AO|)}{2|BO|}
\end{align}
Since $\my \in \Gamma_{k,i}$, by equation \eqref{euclidean vs geodesic}, $d_\nn(\my,\my_{k,i})=d_\nn(\my,\my_{k})=a \leq D_1 r$. By Lemma \ref{bi lip of iota}, $|AB|=a(1+O(r))$ and $|AO|=a(1+O(r))$, hence $(|AB|+|AO|)(|AB|-|AO|)=2a^2 O(r) \leq 2 D^2_1 O(r^3)$. By Lemma \ref{bi lip of iota}, $|BO|=d_\nn(\my_k, \my_{k,i})(1+O(r))$. By (2) in Assumption \ref{assumption on voronoi cell}, $d_\nn(\my_k, \my_{k,i}) \geq  D_2 r$. Hence,
\begin{align}
d_{\mathbb{R}^d}(A,H)=\frac{(|AB|+|AO|)(|AB|-|AO|)}{2|BO|} \leq \frac{D^2_1}{D_2}O(r^2)=O(r^2).
\end{align}
Since $ \tilde{C}_{k,0}$ is convex, $d_{\mathbb{R}^d}(A,\partial \tilde{C}_k) \leq  d_{\mathbb{R}^d}(A,H)$. The conclusion follows.  Note that if $A \not \in \tilde{C}_{k,0}$, we still have $d_{\mathbb{R}^d}(A,H)=O(r^2)$. However, it is not true that  $d_{\mathbb{R}^d}(A, \partial \tilde{C}) \leq  d_{\mathbb{R}^d}(A,H)$.

\textbf{Case 2: $v \not \in \tilde{C}_{k,0}$}

Suppose $v=\tilde{\iota}_k(\my)$ for some $\my \in \partial C_k$.  As shown in figure \ref{Lemma9Case2}, let $O$ be the origin in $\mathbb{R}^d$. Let $A=v$. Suppose $OA$ intersects with $\partial \tilde{C}_k$ at $D$.  $D \in \tilde{\Gammaf}_{k,j}$, where $\tilde{\Gammaf}_{k,j}$ is Voronoi face in $\mathbb{R}^d$ between $O$ and $B=\tilde{\iota}_k(\my_{k,j})$. $H$ is the hyperplane that perpendicularly bisects $OB$. $M$ is the intersection between $H$ and $OB$. Note that $ \tilde{\Gammaf}_{k,j} \subset H$. Let $C$ be the point on $OB$ so that $AC$ is perpendicular to $OB$. Since we assume $A  \not \in  \tilde{C}_{k,0}$, $C \in BM$. We have
\begin{align}
|CM|=&\frac{|CO|-|CB|}{2}=\frac{|CO|^2-|CB|^2}{2(|CB|+|CO|)} \\
=& \frac{|AO|^2-|AC|^2-(|AB|^2-|AC|^2)}{2|BO|}=\frac{(|AB|+|AO|)(|AO|-|AB|)}{2|BO|}.
\end{align}
$\my \in \partial C_k$ but we may not have $\my \in \Gamma_{k,j}$,  therefore, $a =d_\nn(\my,\my_{k,j}) \geq d_\nn(\my,\my_{k})=b$. $d_\nn(\my,\my_{k,j}) \leq  d_\nn(\my,\my_{k})+ d_\nn(\my_{k,j},\my_{k})$, hence by equation \eqref{euclidean vs geodesic}, $b  \leq D_1 r$ and $a  \leq 2D_1 r$. By Lemma \ref{bi lip of iota}, $|AB|=a(1+O(r))$ and $|AO|=b(1+O(r))$.  Since $a \geq b$ and $|AB| \leq |AO|$, we have $ 0 \leq a-b=D_1 O(r^2)$. Hence, $|AO|-|AB|=D_1 O(r^2)$.  
By Lemma \ref{bi lip of iota} and Assumption \ref{assumption on voronoi cell}, 
\begin{align}
& |BO|=d_\nn(\my_k, \my_{k,j})(1+O(r)) \geq  D_2 r (1+O(r)),\\
& |AO|=d_\nn(\my,\my_{k}) (1+O(r)) \leq  D_1 r (1+O(r)).
\end{align}
Hence,
$(|AB|+|AO|)(|AB|-|AO|) \leq 2|AO|(|AB|-|AO|)  \leq 2 D^2_1 O(r^3)$.
Moreover,
\begin{align}
|CM|=\frac{(|AB|+|AO|)(|AB|-|AO|)}{2|BO|} \leq \frac{D^2_1}{D_2}O(r^2)=O(r^2).
\end{align}
At last,
\begin{align}
|AD|=\frac{|CM||AO|}{|OC|} \leq \frac{|CM||AO|}{|OM|}=\frac{2|CM||AO|}{|OB|} \leq \frac{2|CM| D_1 r(1+O(r))}{ D_2 r(1+O(r))}=O(r^2).
\end{align}
Since $d_{\mathbb{R}^d}(A,\partial \tilde{C}) \leq  |AD|$, the conclusion follows.

In the second step,  we show that for any $v \in  \tilde{C}_k $, $d_{\mathbb{R}^d}(v, \partial \tilde{\iota}_k(C_k))=O(r^2)$. The proof  follows the similar argument as the first step, so we omit it.
\end{proof}

Now we prove the first main proposition.

\textbf{\underline{Proof of Proposition \ref{approximation of the volume of a Voronoi cells}}}

\begin{proof}
By Assumption \ref{assumption on voronoi cell}, for any $\my \in \partial C_k$, $\frac{1}{2}D_2 r \leq d_\nn(\my, \my_k) \leq  D_1 r$. By Lemma \ref{bi lip of iota}, any for $v \in \partial \tilde{\iota}_k(C_k)$, $\frac{1}{2}D_2 r(1+O(r)) \leq \|v\|_{\mathbb{R}^d} \leq D_1 r(1+O(r))$. Hence, $\frac{1}{2}D_2 r+O(r^2) \leq \|v\|_{\mathbb{R}^d} \leq D_1 r+O(r^2)$. By Lemma \ref{boundary comparison lemma} and the triangle inequality, for any $v' \in \partial \tilde{C}_k$,  $\frac{1}{2}D_2 r+O(r^2) \leq \|v'\|_{\mathbb{R}^d} \leq D_1 r+O(r^2)$. Since $\tilde{C}_{k,0}$ is convex, we conclude that there is a constant $\Omega$ such that $|\tilde{C}_{k,0}|=\Omega r^d +O(r^{d+1})$.  By Lemma \ref{boundary comparison lemma} and the fact that $\tilde{C}_{k,0}$ is convex, $|\tilde{\iota}_k(C_k)|=\Omega r^d +O(r^{d+1})=|\tilde{C}_{k,0}|(1+O(r))$.  By Lemma \ref{bi lip of iota}, $|C_k|=|\tilde{\iota}_k(C_k)| (1+O(r))^d=|\tilde{\iota}_k(C_k)|(1+O(r))$. Therefore, $|\tilde{C}_{k}|=|\tilde{C}_{k,0}|=|C_k|(1+O(r))$.
\end{proof}
\textbf{\underline{Proof of Proposition \ref{approximation of the area of a Voronoi face}}}

\begin{proof}
We provide a sketch of the proof. Use $|\cdot |$ to denote the $d-1$ dimensional Hausdorff measure. $\partial A$ denotes the topological boundary of a set $A$. Suppose $B^{\mathbb{R}^\ell}_{r}(\my_k) \cap \{\my_i\}_{i=1}^n=\{\my_{k,1}, \cdots, \my_{k, N_k}\}$.   Suppose  $\Gamma_{k,i}$ is the Voronoi face between $\my_k$ and $\my_{k,i}$.

Step 1 We approximate the Voronoi face $\Gamma_{k,i}$ by a region in a $d-1$ dimensional affine subspace in $\mathbb{R}^\ell$.

Suppose the minimizing geodesic intersects the bisector $G$ between $\my$ and $\my_{k,i}$ at $\my^*_{k,i}$. Then, {by Assumption \ref{assumption on voronoi cell} and Proposition \ref{perpendicular to Voronoi face}}, there is a $d-1$ dimensional subspace $S_{k,i}$ of $T_{\my^*_{k,i}} \nn$ which is perpendicular to the tangent vector of the minimizing geodesic at $\my^*_{k,i}$. If we realize $T_{\my^*_{k,i}} \nn$ as a subspace of $\mathbb{R}^\ell$, then the affine subspace $\my^*_{k,i}+S_{k,i}$ is tangent to $G$ at $\my^*_{k,i}$.  Without loss of generality, we rotate and translate the manifold $\nn$ so that $\my^*_{k,i}=0$ and $S_{k,i}$ is identified with the subspace of $\mathbb{R}^\ell$ generated by the first $d-1$ standard basis. {By Assumption \ref{assumption on voronoi cell} and Proposition \ref{perpendicular to Voronoi face}}, there is an open subset of $S_{k,i}$ and denote $L_{k,i}$ to be its closure such that for any $\my \in \Gamma_{k,i}$, we have
\begin{align}
\my=(u, g_1(u), \cdots, g_{\ell-d+1}(u)), \label{chart of Voronoi face}
\end{align}
where $u \in L_{k,i} \subset \mathbb{R}^{d-1}$ and $g_i: \mathbb{R}^{d-1} \rightarrow \mathbb{R}$. Moreover, $g_j(u)$ is smooth and  $g_j(0)=0$ and $\nabla g_j(0)=0$. The second order derivative of $g_i$ can be bounded by the curvature of $\nn$ at $\my_k$ and $\my_{k,i}$. By (1) in Assumption \ref{assumption on voronoi cell}, $\Gamma_{k,i} \subset C_k \subset B^{\mathbb{R}^\ell}_r(\my_k)$, hence for any $\my \in \Gamma_{k,i}$, $\|\my-\my_k\|_{\mathbb{R}^{\ell}} \leq r$. $\|\my^*_{k,i}-\my_k\|_{\mathbb{R}^{\ell}} \leq d_\nn(\my^*_{k,i},\my_k)=\frac{1}{2}d_\nn(\my_{k,i},\my_k) ) \leq \frac{1}{2} D_1 r$. Since $\my_{k,i}=0$, 
\begin{align}
\|\my\|_{\mathbb{R}^{\ell}}=\|\my-\my^{*}_{k,i}\|_{\mathbb{R}^{\ell}} \leq \|\my-\my_k\|_{\mathbb{R}^{\ell}}+\|\my^*_{k,i}-\my_k\|_{\mathbb{R}^{\ell}} \leq (1+\frac{1}{2} D_1)r.
\end{align}

By \eqref{chart of Voronoi face}, for any $u \in L_{k,i}$, $\|u\|_{\mathbb{R}^{d-1}} \leq \|\my\|_{\mathbb{R}^{\ell}} \leq (1+\frac{1}{2} D_1)r$. Thus, $L_{k,i}$ is contained in a $d-1$ dimensional ball of radius $(1+\frac{1}{2} D_1)r$ in $\mathbb{R}^{d-1}$. Hence 
\begin{align}\label{Step 0 area comparison}
|L_{k,i}| \leq |S^{d-1}| (1+\frac{1}{2} D_1)^{d-1}r^{d-1}
\end{align}

\eqref{chart of Voronoi face} implies that 
\begin{align}\label{Step 1 area comparison}
|\Gamma_{k,i}|=|L_{k,i}|+O(|L_{k,i}|^{\frac{d}{d-1}})=|L_{k,i}|+O(r^d), 
\end{align}
where we use $|L_{k,i}| \leq  |S^{d-1}| (1+\frac{1}{2} D_1)^{d-1}r^{d-1}$ in the last step. 
Moreover, 
\begin{align}\label{Step 1 H distance}
d^{\mathbb{R}^{\ell}}_{\mathcal{H}}(\partial\Gamma_{k,i}, \partial L_{k,i})=\max_{u \in \partial L_{k,i}}\sqrt{ g^2_1(u)+ \cdots, g^2_{\ell-d+1}(u)}= O(r^2),
\end{align}
where $d^{\mathbb{R}^{\ell}}_{\mathcal{H}}$ is the Hausdorff distance with respect to the Euclidean metric of $\mathbb{R}^{\ell}$.

Step 2 

This step is an analogue of Lemma \ref{bi lip of iota} when we apply $\tilde{\iota}_k$ to the affine subspace $\my^*_{k,i}+T_{\my^*_{k,i}} \nn$. 
If we identify both $T_{\my^*_{k,i}} \nn$ and $T_{\my_k} \nn$ as the subspaces of $\mathbb{R}^\ell$, then we show that $T_{\my^*_{k,i}} \nn$ is a small perturbation of $T_{\my_k} \nn$ when $r$ is small. For simplicity, we rotate and translate the manifold so that  $\my_k=0$ and $T_{\my_k} \nn$ is generated by the first $d$ standard orthonormal basis $\{e_1, \cdots, e_d\}$ of $\mathbb{R}^\ell$. By the manifold structure of $\nn$, there is an orthonormal basis  $\{e'_1, \cdots, e'_d\}$ of $T_{\my^*_{k,i}} \nn$ with $e'_i=e_i+O(r^2)$. 
By Proposition \ref{structure of discrete PCA} and the similar argument in Lemma \ref{bi lip of iota}, we can show that $\tilde{\iota}_k$ restricted on  the affine subspace $\my^*_{k,i}+T_{\my^*_{k,i}} \nn$ is a $1+O(r)$ bi-Lipschitz homeomorphism.

Step 3

For simplicity,denote  $\tilde{\iota}_k(\my^*_{k,i}+ L_{k,i})$ by $\tilde{\iota}_k(L_{k,i})$, denote $\tilde{\iota}_k(\my^*_{k,i}+\partial L_{k,i})$ by $\tilde{\iota}_k( \partial L_{k,i})$  and denote $\partial \tilde{\iota}_k(\my^*_{k,i}+L_{k,i})$  by $\partial \tilde{\iota}_k(L_{k,i})$.  Since $\tilde{\iota}_k$ restricted on  the affine subspace $\my^*_{k,i}+T_{\my^*_{k,i}} \nn$ is  homeomorphism, $\partial \tilde{\iota}_k(L_{k,i})=\tilde{\iota}_k( \partial L_{k,i})$. Moreover, Lemma \ref{bi lip of iota} shows that $\tilde{\iota}_k$ restricted on $ B^{\mathbb{R}^\ell}_{r}(\my_k) \cap \nn$ is a homeomorphism. Hence, $\partial \tilde{\iota}_k(\Gamma_{k,i})=\tilde{\iota}_k( \partial \Gamma_{k,i})$. Since  $\tilde{\iota}_k$ is a projection, 
\begin{align}\label{Step 3 H distance}
d^{\mathbb{R}^{d}}_{\mathcal{H}}(\partial \tilde{\iota}_k(\Gamma_{k,i}), \partial \tilde{\iota}_k(L_{k,i}))=d^{\mathbb{R}^{d}}_{\mathcal{H}}(\tilde{\iota}_k( \partial \Gamma_{k,i}), \tilde{\iota}_k( \partial L_{k,i})) \leq d^{\mathbb{R}^{\ell}}_{\mathcal{H}}(\partial\Gamma_{k,i}, \partial L_{k,i})= O(r^2),
\end{align}
where we use \eqref{Step 1 H distance} in the last step. Since  $\tilde{\iota}_k$ is a projection,  $\tilde{\iota}_k(L_{k,i})$ is a subset of a $d-1$ dimensional affine subspace of $\mathbb{R}^d$. Since  $\tilde{\iota}_k$ restricted on  the affine subspace $\my^*_{k,i}+T_{\my^*_{k,i}} \nn$ is a $1+O(r)$ bi-Lipschitz homeomorphism, 
\begin{align}\label{Step 3 area comparison}
|L_{k,i}|=|\tilde{\iota}_k(L_{k,i})|(1+O(r)). 
\end{align}

Step 4

Recall in the step (4) in Algorithm \ref{Voronoi cell on manifold algorithm}, we find the Voronoi cell decomposition of $\{0, \tilde{\iota}_k(\my_{k,1}), \cdots, \tilde{\iota}_k(\my_{k,N_k})\}$ in $\mathbb{R}^d$. The Voronoi cell containing $0$ is $\tilde{C}_{k,0}$.  The Voronoi face between $0$ and $\tilde{\iota}_k(\my_{k,i})$ is denoted as $\tilde{\Gammaf}_{k,i}$. If $\my \in \partial \Gamma_{k,i}$, then there is a third point $\my_{k,j}$ such that $d_{\nn}(\my, \my_k)=d_{\nn}(\my, \my_{k,i})=d_{\nn}(\my, \my_{k,j})$. By using the similar argument in Lemma \ref{boundary comparison lemma}, we can show that
\begin{align}
d^{\mathbb{R}^{d}}_{\mathcal{H}}(\partial \tilde{\iota}_k(\Gamma_{k,i}), \partial \tilde{\Gammaf}_{k,i})=O(r^2).
\end{align}
By \eqref{Step 3 H distance}, we have
\begin{align}
d^{\mathbb{R}^{d}}_{\mathcal{H}}( \partial \tilde{\iota}_k(L_{k,i}), \partial \tilde{\Gammaf}_{k,i})=O(r^2).
\end{align}

Step 5

By (1) in Assumption \ref{assumption on voronoi cell}, $C_k \subset B^{\mathbb{R}^\ell}_{r}(\my_k) \cap \nn$. Since  $\tilde{\iota}_k$ is a projection,  $\tilde{\iota}_k(C_k)$ is in the ball of radius $r$ centered at $0$ in $\mathbb{R}^d$. By Lemma \ref{boundary comparison lemma}, $\tilde{C}_k$ is in the ball of radius $2r$  centered at $0$ in $\mathbb{R}^d$, when $r$ is small enough. $ \tilde{\Gammaf}_{k,i}$ is a convex polygon and is in a $d-1$ dimensional affine subspace $H_{k,i}$ in $\mathbb{R}^d$. We know that $\partial \tilde{\Gammaf}_{k,i}=\cup_j  \mathcal{C}_j$, where each $\mathcal{C}_j$ is a $d-2$ dimensional convex polygon.  Each $\mathcal{C}_j$ is a ball of radius  $2r$.   Hence, we have $\mathcal{H}^{d-2}(\partial \tilde{\Gammaf}_{k,i})=O(r^{d-2})$ and any $O(r^2)$ neighborhood of $\partial  \tilde{\Gammaf}_{k,i}$ in $H_{k,i}$ has $d-1$ Hausdorff measure $O(r^d)$.  Since $d^{\mathbb{R}^{d}}_{\mathcal{H}}( \partial \tilde{\iota}_k(L_{k,i}), \partial \tilde{\Gamma}_{k,i})=O(r^2)$ and  $\tilde{\iota}_k(L_{k,i})$ is a subset of a $d-1$ dimensional affine subspace of $\mathbb{R}^d$, we rotate and translate $\tilde{\iota}_k(L_{k,i})$ so that $\partial \tilde{\iota}_k(L_{k,i})$ is in a $O(r^2)$ neighborhood of $\partial  \tilde{\Gamma}_{k,i}$ in $H_{k,i}$. Therefore,
\begin{align}\label{Step 5 area comparison}
|\tilde{\iota}_k(L_{k,i})|=| \tilde{\Gammaf}_{k,i}|+O(r^d)
\end{align}
Combine \eqref{Step 0 area comparison}, \eqref{Step 1 area comparison}, \eqref{Step 3 area comparison} and \eqref{Step 5 area comparison}, we have
\begin{align}
|\Gamma_{k,i}|=|\tilde{\Gammaf}_{k,i}|+O(r^d).
\end{align}
If  $\my_\ell=\my_{k,i} \in B^{\mathbb{R}^\ell}_{r}(\my_k)$, then $\tilde{A}_{kl}=|\tilde{\Gammaf}_{k,i}|$. So, $|\Gamma_{k\ell}|=\tilde{A}_{kl}+O(r^d)$.  Similarly, $|\Gamma_{\ell k}|=|\Gamma_{k \ell}|=\tilde{A}_{kl}+O(r^d)$. Hence, $|\Gamma_{k\ell}|=\frac{\tilde{A}_{kl}+\tilde{A}_{kl}}{2}+O(r^d)=A_{k \ell} +O(r^d)$. If  $A_{k \ell} \geq a_1 r^d$, we automatically have the conclusion. If $A_{k \ell} <a_1 r^d$, then $|\Gamma_{k\ell}|=O(r^d)$ and $|\tilde{\Gamma}_{k \ell}|=a_1 r^d$. So, we also have $|\Gamma_{k\ell}|=|\tilde{\Gamma}_{k \ell}|+O(r^d)$.

\end{proof}
\
\section{Other standard time discretizations}\label{otherS}
Lemma \ref{explicit-lem} below gives the maximum principle, and exponential convergence for an explicit scheme under { Courant-Friedrichs-Lewy (CFL) condition.}  Lemma \ref{implicit-lem} below gives the unconditional maximum principle, and exponential convergence for an implicit scheme.  
\begin{lem}\label{explicit-lem}
{Let $\tilde{\lmd}_i$ be the approximated jump rate and $\tilde{P}_{ij}$ be the approximated transition probability defined in \eqref{definition of delta tilde and P tilde}.}
Let $\Delta t$ be the time step and consider the explicit scheme for \eqref{mp1_pron}
\begin{equation}\label{explicit}
\frac{\rho_i^{n+1}|\tilde{C}_i| - \rho_{i}^n |\tilde{C}_i|}{\Delta t} = \left( \sum_{j\in \text{VF}(i)} \tilde{\lmd}_j \tilde{ P}_{ij} \rho_j^n |\tilde{ C}_j| - \tilde{\lmd}_i \rho_i^n |\tilde{ C}_i|   \right).
\end{equation}
With the detailed balance property \eqref{db_cfl}, and the CFL condition for $\Delta t$
\begin{equation}\label{tm_cfl}
\Delta t \leq \frac{1}{\tilde{\lmd}_i} =  \frac{2|\tilde{ C}_i|\pi_i}{\sum_{j\in \text{VF}(i)} \frac{\pi_i+ \pi_j}{|y_i-y_j|}|\tilde{\Gamma}_{ij}|},
\end{equation}
we have
\begin{enumerate}[(i)]
\item the conversational law for $ \rho_i^{k+1}|\tilde{C}_i|$, i.e.
\begin{equation}
\sum_i  \rho_i^{k+1} |\tilde{ C}_i| = \sum_i \rho_i^{k} |\tilde{ C}_i|;
\end{equation}
\item the equivalent updates 
for $u_i^{k+1}=\frac{\rho_i^{k+1}}{\pi_i}$ 
\begin{equation}\label{matrix_ex}
u^{k+1} = (I+\Delta t Q) u^k, \quad  Q:= \{b_{ij}\} \text{ with } b_{ij}:= \left\{\begin{array}{cc}
-\tilde{\lmd}_i  , \quad &j=i;\\
\tilde{\lmd}_i \tilde{P}_{ji}, \quad &j\neq i;
\end{array}\right.
\end{equation}
\item  the  maximum principle for $\frac{\rho_i}{\pi_i}$
\begin{equation}
\max_ij \frac{\rho^{k+1}_j }{\pi_j}\leq \max_j \frac{\rho^{k}_j }{\pi_j}.
\end{equation}
\item the $\ell^\8$ contraction
\begin{equation}
\max_i \left| \frac{\rho^{k+1}_i}{\pi_i}-1\right| \leq \max_i \left| \frac{\rho^{k}_i}{\pi_i}-1\right| ;
\end{equation}
\item the exponential convergence
\begin{equation}
\left\| \frac{\rho^{k}_i}{\pi_i}-1\right\|_{\ell^\8} \leq c  |\lambda_2|^k, \quad |\lambda_2|<1,
\end{equation}
where $\lambda_2$ is the second eigenvalue (in terms of the magnitude) of $(I+\Delta t Q)$.
\end{enumerate}
\end{lem}

\begin{lem}\label{implicit-lem}
{Let $\tilde{\lmd}_i$ be the approximated jump rate and $\tilde{P}_{ij}$ be the approximated transition probability defined in \eqref{definition of delta tilde and P tilde}.}
Let $\Delta t$ be the time step and consider the implicit scheme 
\begin{equation}\label{implicit}
\frac{\rho_i^{n+1}}{\pi_i} = \frac{\rho^n_i}{\pi_i}-\tilde{\lmd}_i \Delta t \frac{\rho^{n+1}_i}{\pi_i}  +  \Delta t\sum_{j\in \text{VF}(i)} \tilde{\lmd}_i \tilde{ P}_{ji}   \frac{\rho^{n+1}_j }{\pi_j}.
\end{equation}
We have the following unconditional properties:
\begin{enumerate}[(i)]
\item the conversational law for $ \rho_i^{k+1}|\tilde{C}_i|$, i.e.
\begin{equation}
\sum_i  \rho_i^{k+1} |\tilde{ C}_i| = \sum_i \rho_i^{k} |\tilde{ C}_i|;
\end{equation}
\item the equivalent updates 
for $u_i^{k+1}=\frac{\rho_i^{k+1}}{\pi_i}$ { with same $Q$ in \eqref{matrix_ex}}
\begin{equation}
(I-\Delta t Q)u^{k+1} =  u^k;
\end{equation}
\item  the  maximum principle for $\frac{\rho_i}{\pi_i}$
\begin{equation}
\max_i \frac{\rho^{k+1}_j }{\pi_j}\leq \max_j \frac{\rho^{k}_j }{\pi_j}.
\end{equation}
\item the $\ell^\8$ contraction
\begin{equation}
\max_i \left| \frac{\rho^{k+1}_i}{\pi_i}-1\right| \leq \max_i \left| \frac{\rho^{k}_i}{\pi_i}-1\right| ;
\end{equation}
\item the exponential convergence
\begin{equation}\label{exp_im}
\left\| \frac{\rho^{k}_i}{\pi_i}-1\right\|_{\ell^\8} \leq c  |\lambda_2|^k, \quad |\lambda_2|<1,
\end{equation}
where $\lambda_2$ is the second eigenvalue (in terms of the magnitude) of $(I-\Delta t Q)^{-1}$. 
\end{enumerate}
\end{lem}

The proof of the two lemmas is same as Proposition \ref{error bound for the unconditional stable explicit scheme} and we omit it. 
{ The advantage of the explicit scheme \eqref{explicit} is its efficiency but the disadvantage is the requirement of CFL condition on $\Delta t$; see \eqref{tm_cfl}. Indeed, the convergence rate for the explicit scheme \eqref{explicit} is very slow since the spectral gap vanishes as $\Delta t \to 0.$  
On the other hand, the unconditionally stable implicit scheme \eqref{implicit} gives the exponential convergence \eqref{exp_im} with fast rate when we take $\Delta t$ large enough but it is  time-consuming to solve the inverse matrix in practice. }

{ \section{Computations of source term in von Mises-Fisher's ground-truth distribution}\label{app_von}
Recall the definition of von Mises-Fisher's distribution with oscillated parameters. We compute the source term in \eqref{source46}
\begin{equation}
\begin{aligned}
g(\theta, \varphi, t) = &\rho \Big[ \kappa^2 \eta^2_\theta -\kappa \eta  +    \kappa    \eta_\theta {\cot \theta}  + \frac{1}{\sin^2 \theta}(\kappa^2 \eta^2_\varphi + \kappa \eta_{\varphi\varphi} )- \frac{C'}{C} \kappa'- (\kappa' \eta + \kappa \eta_t)  \Big],\\
&C'/C =     \frac{\sinh \kappa - \kappa \cosh \kappa}{\kappa \sinh\kappa}.
\end{aligned}
\end{equation}
Using
\begin{align*}
\eta_\theta = -\cos a \sin \theta + \sin a \cos \theta \cos(\varphi -b), \quad \eta_{\theta\theta} =-\cos a \cos \theta - \sin a \sin \theta \cos (\varphi-b) = -\eta,\\
\eta_\varphi = - \sin a \sin \theta \sin (\varphi -b), \quad \eta_{\varphi\varphi} =  -\sin a \sin \theta \cos (\varphi -b) = \cos a \cos \theta - \eta,\\
\eta_t = -a'\sin a \cos \theta + a' \cos a \sin \theta \cos(\varphi-b) + b' \sin a \sin \theta \sin(\varphi-b).
\end{align*}
We obtain
\begin{align*}
\frac{g(\theta, \varphi, t)}{\rho} = &   \kappa^2 [\eta^2_\theta + \frac{\eta_\varphi^2}{\sin^2\theta} ] +\kappa [ -\eta  +     \eta_\theta {\cot \theta}  - \frac{\sin a  \cos (\varphi -b)}{\sin \theta}   ] - \frac{C'}{C} \kappa'- (\kappa' \eta + \kappa \eta_t) ,\\
=&   \kappa^2 [\eta^2_\theta + \sin^2 a \sin^2(\varphi-b) ] +\kappa [ -\eta  +     \eta_\theta {\cot \theta}  - \frac{\sin a  \cos (\varphi -b)}{\sin \theta}  ] - \frac{C'}{C} \kappa'- (\kappa' \eta + \kappa \eta_t) ,\\
=&   \kappa^2 [\eta^2_\theta + \sin^2 a \sin^2(\varphi-b) ] +\kappa [ -\eta     -\cos a\cos \theta - \sin a\sin \theta \cos(\varphi -b)     ] - \frac{C'}{C} \kappa'- (\kappa' \eta + \kappa \eta_t) ,\\
=&   \kappa^2 [\eta^2_\theta + \sin^2 a \sin^2(\varphi-b) ] -2\kappa  \eta   - \frac{C'}{C} \kappa'- (\kappa' \eta + \kappa \eta_t), 
\end{align*}
which gives the source term \eqref{gss}.
}

\end{document}